\documentclass[english]{article}
\usepackage[T1]{fontenc}
\usepackage[numbers]{natbib}
\usepackage{filecontents}
\usepackage[latin9]{inputenc}
\usepackage[letterpaper]{geometry}
\usepackage{pifont}
\usepackage{float}
\usepackage{amsmath}
\usepackage{graphicx}
\usepackage{color}
\usepackage{epstopdf}
\usepackage{adjustbox}
\usepackage{diagbox}
\usepackage{multirow}
\usepackage{algorithm,algorithmic}
\usepackage{array}
\usepackage{mathrsfs}
\usepackage{amsmath, amssymb, amsthm}
\usepackage{subfigure}
\usepackage{tikz}
\usepackage{pgfplots}
\usepackage[hidelinks]{hyperref}
\usepackage{comment}

\graphicspath{{./Figures/}}

\newcommand{\ms}{\text{ms}}
\newcommand{\glo}{\text{glo}}
\newcommand{\aux}{\text{aux}}
\newcommand{\tforall}{\text{ for all }}
\newcommand{\tand}{\quad\text{and}\quad}
\theoremstyle{definition}
\newtheorem{theorem}{Theorem}[section]
\newtheorem{example}{Example}[section]
\newtheorem{assumption}{Assumption}[section]
\newtheorem{lemma}{Lemma}[section]
\newtheorem{remark}{Remark}

\usepackage{amsfonts} 
\newcommand\norm[1]{\left\Vert#1\right\Vert}

\newcommand\abs[1]{\lvert#1\rvert}

\newcommand\dt{\,\text{d}t}
\newcommand\dx{\,\text{d}x}

\newcommand{\Nv}{N_\mathrm{v}}
\newcommand{\Ne}{N_\mathrm{e}}
\DeclareMathOperator{\spa}{span}

\usepackage{color}
\usepackage{adjustbox}
\usepackage{diagbox}
\definecolor{black}{rgb}{0,0,0}

\definecolor{red}{rgb}{1,0,0}

\definecolor{blue}{rgb}{0,0,1}

\usepackage{multirow}
\newcommand{\revi}[1]{\textcolor{black}{#1}}
\newcommand{\revii}[1]{\textcolor{black}{#1}}
\makeatother

\usepackage{babel}
\usepackage{authblk}

\title{\sc Fast Online Adaptive Enrichment for Poroelasticity with High Contrast}
\author{Xin Su\footnote{Department of Mathematics, Texas A\&M University, College Station, TX 77843, United States} \quad and \quad Sai-Mang Pun\footnote{Department of Chemistry, The University of Hong Kong; and Hong Kong Quantum AI Lab, Hong Kong SAR. The work was partly done when the author was affiliated with Department of Mathematics in Texas A\&M University.}}

\date{\today}
 
\begin{document}
\maketitle
\begin{abstract}
In this work, we develop an online adaptive enrichment method within the framework of the Constraint Energy Minimizing Generalized Multiscale Finite Element Method (CEM-GMsFEM) for solving the linear heterogeneous poroelasticity models with coefficients of high contrast. The proposed method makes use of information of residual-driven error indicators to enrich the multiscale spaces for both the displacement and the pressure variables in the model. Additional online basis functions are constructed in oversampled regions accordingly and are adaptively chosen to reduce the error the most. A complete theoretical analysis of the online enrichment algorithm is provided and justified by thorough numerical experiments.
\end{abstract}

\section{Introduction}

Modeling the deformation of porous media saturated by incompressible viscous fluid plays an important role in a wide range of applications such as reservoir engineering in the field of geomechanics \cite{zoback2010reservoir}  or environmental safety due to overburden subsidence and compaction \cite{geomechanics}.  In order to obtain a reasonable model, it is vital to couple the flow of the fluid with the behavior of the surrounding solid. The so-called Biot's model that couples a Darcy flow with linear elastic behavior of the porous medium was proposed in \cite{biot1941general, biot1956theory}. 
More generally, mathematical models of the poroelastic multiphase flow problems can also be considered \cite{shen2022sequential}. 
The corresponding well-posedness results were developed in \cite{showalter2000diffusion}.

If the medium is homogeneous, typical finite element techniques can be used to simulate the poroelastic behavior, see for instance \cite{chaabane2018splitting, ern2009posteriori,lewis1998finite,murad1994stability}. 
However, if the material is strongly heterogeneous, the displacement and pressure might oscillate on fine scale. It is prohibitively expensive to use fine scale method due to the level of details incorporated in fine-scale models.  Coarse models are often used to get around the expensive computational cost under such circumstances. These methods include the generalized multiscale finite element method (GMsFEM) \cite{gmsfem_poro1,gmsfem_poro2, efendiev2013generalized}, the variational multiscale methods \cite{hughes1998variational}, the localized orthogonal decomposition technique \cite{lod_poro,maalqvist2017generalized,maalqvist2014localization}, and the constraint energy minimizing GMsFEM (CEM-GMsFEM) \cite{chung2017constraint, fu2019computational}. All these techniques attempt to construct multiscale spaces to obtain accurate coarse-scale solutions, which could reflect the spatial fine-scale features. 
 
Such construction of the multiscale spaces is usually regarded as offline computation. The offline stage requires one-time computation effort in advance. With the reduced degrees of freedom, the simulation can be efficient, especially when simulating the evolution of the solutions in our case. The approach in \cite{fu2019computational} is classified as offline method while the present work takes care of global information in the equations such as source term, which is classified as the online method. The online approaches are brought up because the slow error decay in the offline method after a certain number of basis functions are selected. Two online approaches are proposed based on the GMsFEM framework. The error decay of the online GMsFEM proposed in \cite{online_mixed,chung2016adaptive,chung2015residual} is proportional to $1-C\Lambda$, where $\Lambda$ is the smallest eigenvalue (across all the coarse blocks) whose corresponding eigenvector is not included in the coarse space with some generic constant $C$ ensuring the positivity of the convergence rate. The error decay of the fast online adaptive approach proposed in \cite{chung2018fast}  is proportional to some user-defined parameters if the oversampling region is properly chosen and the number of offline basis function is sufficient. 

Based on the previous work \cite{fu2019computational}, we adopt the idea presented in \cite{chung2018fast} to develop online adaptive enrichment strategies for the poroelasticity problem. 
The computation is divided into the offline and online stages. 
In the offline stage, we construct the offline basis functions with CEM-GMsFEM as in \cite{fu2019computational} for both displacement and pressure variables. For the temporal discretization, the implicit Euler scheme is utilized. 
In the online stage, we aim to construct online basis functions (for both displacement and pressure) according to the information of the residual indicators based on the previously obtained solution in the offline stage. Our theoretical result shows that with the proposed online enrichment strategy, the online adaptive algorithm leads to a rapid convergence to the continuous solution. 
We carry out numerical experiments where we take into account high contrast Young's modulus and high contrast permeability fields. 
In the first example the online adaptive procedure is performed at the final time step while in the second example, the online adaptive procedure is performed every five time steps. All results show that the error decay significantly after enriching the multiscale spaces with our method. Keep enriching the multiscale spaces every several time steps can get an even better result. We prove that if the oversampling regions are properly chosen and the number of offline basis function is sufficient, the error decay of both displacement and pressure is proportional to online tolerance parameters and controlled by the error from previous time step. This means that the error decay can be made close to 0 if the solution of the previous time step is good enough.

An organization of the paper is as follows. In Section \ref{sec:modelpb}, we present some preliminaries related to the model problem considered in this work. In Section \ref{sec:offline}, we outline the CEM-GMsFEM for the construction of the offline basis functions, and in Section \ref{sec:online}, we present our online adaptive enrichment algorithm. Numerical results, showing the expected performance of the proposed method, are presented in Section \ref{sec:num}. Section \ref{sec:anal} is devoted to the theoretical analysis of the proposed method.  Finally, concluding remarks are drawn in Section \ref{sec:con}. 

\section{Preliminaries}\label{sec:modelpb}
\subsection{Model problem}
Let $\Omega \subset \mathbb{R}^d$ ($d \in \{ 2, 3 \}$) be a bounded and polyhedral Lipschitz domain and $T > 0$ be a fixed time. We consider the problem of linear poroelasticity: 
Find the pressure 
$p\colon [0,T]\times \Omega\to \mathbb{R}$ and the displacement field $u\colon [0,T]\times \Omega\to \mathbb{R}^d$ such that 
\begin{subequations}
\label{eq:model}
\begin{alignat}{3}
\label{eq:model1}
-\nabla\cdot \sigma(u) +  \nabla (\alpha p)  &= 0\phantom{f}\qquad\text{in } (0,T] \times \Omega, \\
\label{eq:model2}
\partial_t \bigg( \alpha \nabla\cdot  u + \frac{1}{M}  p \bigg)  - \nabla \cdot \bigg( \frac{\kappa}{\nu} \nabla p\bigg) &= f\phantom{0}\qquad\text{in } (0,T] \times \Omega, 
\end{alignat}
\end{subequations}	
with boundary and initial conditions
\begin{subequations}
\begin{alignat}{3}
\label{eq:init1}
u&=0\phantom{p^0}\qquad \text{on }(0,T]\times \partial\Omega,\\
\label{eq:init2}
p&=0\phantom{p^0}\qquad \text{on }(0,T]\times \partial\Omega,\\
p(\cdot,0)&=p^0\phantom{0}\qquad\text{in }\Omega.\label{eq:init3}
\end{alignat}
\end{subequations}
For the sake of simplicity, we only consider homogeneous Dirichlet boundary here. The extension to other types of boundary conditions can also be considered (e.g., see \cite{henning2014localized}). In this model, the primary sources of the heterogeneity are the stress tensor $\sigma(u)$, the permeability $\kappa \in L^\infty(\Omega)$, and the Biot-Willis fluid-solid coupling coefficient $\alpha\in[0,1]$. We denote by $M$ the Biot modulus and by $\nu$ the fluid viscosity. Both are assumed to be constant. Moreover, $f \in L^2(0,T;\Omega)$ is a source term representing injection or production processes and $p^0 \in L^2(\Omega)$. Body forces, such as gravity, are neglected. In the case of a linear elastic stress-strain constitutive relation, the stress and strain tensors are expressed as
 \begin{equation*}
 {\sigma(u)} := 2\mu  {\varepsilon(u)} + \lambda (\nabla\cdot  {u}) \,  {\mathcal{I}} \quad \text{and} \quad
 {\varepsilon}( {u}) := \frac{1}{2} \Big( \nabla  {u} + (\nabla  {u})^T \Big),
 \end{equation*}
 where $\mathcal{I} \in \mathbb{R}^{d \times d}$ is the identity tensor. Moreover, $\lambda> 0$ and $\mu>0$ are the Lam\'e coefficients, 
 which can be expressed in terms of the Young's modulus $E>0$
 and the Poisson ratio $\nu_p\in(-1, 1/2)$ as follows: 
 \begin{equation*}
 \lambda:=\frac{\nu_p}{(1-2\nu_p)(1+\nu_p)}E, \tand
 \mu:=\frac{1}{2(1+\nu_p)}E.
 \end{equation*}
In the considered cases of heterogeneous media, the coefficients $\mu$, $\lambda$, $\kappa$, and $\alpha$ may be highly oscillatory.
 
 \subsection{Function spaces}
In this subsection, we clarify the notation used throughout the article. 
We write $(\cdot,\cdot)$ to denote the inner product in $L^2(\Omega)$ and $\norm{\cdot}$ for the corresponding norm. 
Let $H^1(\Omega)$ be the classical Sobolev space with norm $\norm{v}_1 := \big( \norm{v}^2 + \norm{\nabla v}^2 \big)^{1/2}$ for all $v \in H^1(\Omega)$ and $H_0^1(\Omega)$ the subspace of functions having a vanishing trace. We denote the corresponding dual space of $H_0^1(\Omega)$ by $H^{-1}(\Omega)$. 
Moreover, we write $L^r(0,T; X)$ for the Bochner space with the norm 
$$ \norm{v}_{L^r(0,T;X)} := \bigg( \int_0^T \norm{v}_X^r \dt \bigg)^{1/r}, \quad 1\leq r < \infty, \quad \norm{v}_{L^\infty(0,T;X)} := \sup_{0 \leq t \leq T} \norm{v}_X,$$
where $(X,\norm{\cdot}_X)$ is a Banach space. Also, we define $H^1(0,T;X) := \{ v \in L^2(0,T;X) : \partial_t v \in L^2(0,T;X) \}$.
To shorten notation, we define the spaces for the displacement $u$ and the pressure $p$ by
  \begin{equation*}
  V:=[H_0^1(\Omega)]^d,\quad Q:=H^1_0(\Omega).
  \end{equation*}
  
 \subsection{Variational formulation and discretization}
In this subsection, we present the variational formulation corresponding to the system \eqref{eq:model}. We first multiply the equations \eqref{eq:model1} and \eqref{eq:model2} with test functions from $V$ and $Q$, respectively. Then, applying Green's formula and making use of the boundary conditions \eqref{eq:init1} and \eqref{eq:init2}, we obtain the following variational problem: find $u(t, \cdot)\in V$ and $p(t, \cdot)\in Q$ such that
\begin{subequations}\label{eq:weak}
	\begin{alignat}{3}
   a(u,v) - d(v,p) &= 0, \label{eqn:v1} \\
   d(\partial_t u,q) + c(\partial_t p,q) + b(p,q) &= (f,q), \label{eqn:v2}      
\end{alignat} 
for all $v\in V$, $q\in Q$, and  
\begin{alignat}{3}
p(0,\cdot)=p^0 \in Q. \label{eqn:v3}
\end{alignat}
\end{subequations}
The bilinear forms are defined by 
   \begin{eqnarray*}
   && a(u,v) := \int_{\Omega} \sigma(u) : \varepsilon(v)\dx, \quad b(p,q) := \int_{\Omega} \frac{\kappa}{\nu}\, \nabla p\cdot \nabla q \dx, \\
   && c(p,q) := \int_{\Omega} \frac{1}{M}\, p\, q\dx, \tand  d(u,q) := \int_{\Omega} \alpha\,  (\nabla \cdot u)q\dx.
\end{eqnarray*}      
Note that \eqref{eqn:v1} and \eqref{eqn:v3} can be used to define a consistent initial value $u^0:= u(0,\cdot) \in V$. 
With Korn's inequality \cite{BrennerScott, ciarlet1988mathematical}, there exist two constants $c_\sigma$ and $C_\sigma$ such that 
$$ c_\sigma \norm{v}_1^2 \leq a(v,v) =: \norm{v}^2_a \leq C_\sigma \norm{v}_1^2$$
for all $v \in V$, where $c_\sigma$ depends on $\operatorname{essinf}_{x\in \Omega}\mu(x)$ while  $C_\sigma$ depends on $\operatorname{esssup}_{x\in \Omega} \mu(x)$ and $\operatorname{esssup}_{x\in \Omega} \lambda(x)$. Similarly, there exist two positive constants $c_\kappa$ and $C_\kappa$ such that 
$$ c_\kappa \norm{q}_1^2 \leq b(q,q) =: \norm{q}^2_b \leq C_\kappa \norm{q}_1^2$$ 
for all $q \in Q$. Here, $c_\kappa$ depends on $\operatorname{essinf}_{x\in \Omega}\kappa(x)$ and $C_\kappa$ depends on $\operatorname{esssup}_{x\in \Omega} \kappa(x)$. 
We remark that the well-posedness result of \eqref{eq:weak} can be found in \cite{showalter2000diffusion}. 

Here are some notations that would be used later in this article.  Let $V^\prime$ (resp., $Q^\prime$) be the dual space of $V$ (resp., $Q$) with the duality product denoted by $\langle \cdot, \cdot\rangle_a$ (resp., $\langle \cdot, \cdot\rangle_b$) and norm $\|\cdot\|_{a^\prime}=\sup_{0\not=v\in V}|\langle \cdot, v\rangle_a|/\|v\|_a$ (resp., $\|\cdot\|_{b^\prime}=\sup_{0\not=q\in Q}|\langle \cdot, q\rangle_b|/\|q\|_b$). 
$b(\cdot, \cdot) $ is a continuous bilinear form with continuity constant $\beta_1$, i.e., for all $(v,q)\in V\times Q$, $|d(v,q)|\leq \beta_1 \|v\|_a\|q\|_c$, where $\beta_1 =\max_{x\in\Omega}\{\alpha(M/\lambda)^{\frac{1}{2}}\}$ and $\|\cdot\|_c$ is induced by $c(\cdot,\cdot)$. For all $q\in Q$, it holds that $\|q\|_c\leq \beta_2 \|q\|_b$, where $\beta_2=\frac{C_p}{M}$ and $ C_p$ is the  Poincar\'e constant.

We use a temporal discretized approach as a reference solution. For the time discretization, let $\tau$ be a uniform time step and define $t_n=n\tau$ for $n \in \{  0,1,\cdots,N \}$ and $T=N\tau$. The semi-discretization in time by the backward Euler method yields the following semi-discrete problem: given initial data $p^0$ and $u^0$, find $\{ u^n \}_{n=1}^N \subset V$ and $\{ p^n \}_{n=1}^N \subset Q$ such that
\begin{subequations}\label{eq:semi}
	\begin{alignat}{3}
	a(u^n,v) - d(v,p^n) &= 0, \label{eq:semi_a}\\
	d(D_{\tau}u^n,q) + c(D_{\tau}p^n,q) + b(p^n,q) &= ( f^n,q ),      \label{eq:semi_b}
	\end{alignat} 
\end{subequations}
for all $v\in V$ and $q\in Q$. 
Here, $D_{\tau}$ denotes the discrete time derivative, i.e., $D_{\tau}u^n:=(u^n-u^{n-1})/\tau$ and $f^n:=f(t_n)$.

To fully discretize the variational problem \eqref{eq:weak}, we introduce a conforming partition $\mathcal{T}^h$ for the computational domain 
$\Omega$ with (local) grid sizes $h_{K}:=\text{diam}(K)$ for $K\in \mathcal{T}^h$ and ${h:=\max_{K\in \mathcal{T}^h}h_K}$. We remark that $\mathcal{T}^h$ is referred to as the \textit{fine grid}. Next, let $V_h$ and $Q_h$ be the standard finite element spaces of first order with respect to the fine grid $\mathcal{T}^h$, i.e.,
$$ V_h := \{ v = (v_i)_{i=1}^d \in V: \text{each} ~ v_i \lvert_K \text{ is a polynomial of degree} \leq 1 \text{ for all } K \in \mathcal{T}^h \}, $$
$$ Q_h := \{ q \in Q: q\lvert_K \text{ is a polynomial of degree} \leq 1 \text{ for all } K \in \mathcal{T}^h \}. $$
Let $N_{V,\text{fine}}$ and $N_{Q,\text{fine}}$ be the dimension of $V_h$ and $Q_h$, respectively.
The fully discretization of \eqref{eq:weak} read as follows: given $p_h^0$ and $u_h^0$, find $\{ u_h^n \}_{n=1}^N \subset V_h$ and $\{ p_h^n \}_{n=1}^N \subset Q_h$ such that
\begin{subequations}\label{eq:weak1}
	\begin{alignat}{3}
	a(u_h^n,v) - d(v,p_h^n) &= 0, \label{eq:weak1_a}\\
	d(D_{\tau}u_h^n,q) + c(D_{\tau}p_h^n,q) + b(p_h^n,q) &= (f^n,q),      \label{eq:weak1_b}
	\end{alignat} 
\end{subequations}
for all $v\in V_h$ and $q\in Q_h$. 
Here, the initial value $p_h^0\in Q_h$ is set to be the $L^2$
projection of $p^0\in Q$. The initial value $u_h^0$ for the displacement can be obtained by solving 
\begin{equation} \label{eq:initial_u}
a(u_h^0,v)=d(v,p_h^0)
\end{equation}
for all $v \in V_h$.

\section{Offline Multiscale Method}\label{sec:offline}
In this section, we briefly present the CEM-GMsFEM for the linear heterogeneous poroelasticity model \cite{fu2019computational}. The method generates multiscale spaces that are used at the offline stage simulation in our proposed approach. The construction of the multiscale spaces consists of two steps. The first step is to construct auxiliary function spaces. Based on the auxiliary spaces, we can then construct multiscale spaces containing basis functions whose energy are minimized in some subregions of the domain. Then, a multiscale solution can be computed using these energy-minimized basis functions.

Let $\mathcal{T}^H$ be a conforming partition of the computational domain $\Omega$ such that $\mathcal{T}^h$ is a refinement of $\mathcal{T}^H$.
We call $\mathcal{T}^H$ the \textit{coarse grid} and each element of $\mathcal{T}^H$ is called a coarse block. We denote $H:=\displaystyle{\max_{K\in \mathcal{T}^H}\text{diam}(K)}$ the coarse grid size. 
Let $\Nv$ be the total number of (interior) nodes of $\mathcal{T}^H$ and $\Ne$ be the total number of coarse elements. We remark that the coarse element $K \in \mathcal{T}^H$ is a closed subset (of the domain $\Omega$) with nonempty interior and piecewise smooth boundary. 
Let $\{x_i\}_{i=1}^{\Nv}$ be the set of nodes in $\mathcal{T}^H$. 

\subsection{Auxiliary spaces}
The CEM-GMsFEM starts with the auxiliary basis functions by solving spectral problems on each coarse element $K_i$ over the spaces $V(K_i):= {V \vert}_{ K_i}$ and $Q(K_i):= {Q \vert}_{ K_i}$. This step serves as a foundation of local model reduction. More precisely, we consider the local eigenvalue problems (of Neumann type): find $(\lambda_j^i,v_j^i)\in \mathbb{R}\times V(K_i)$ such that
\begin{equation}\label{eq:eig1}
a_i(v_j^i,v)=\lambda_j^i s^1_i(v_j^i,v)
\end{equation}
for all $v \in V(K_i)$ and find 
$(\zeta_j^i,q_j^i)\in \mathbb{R}\times Q(K_i)$ such that 
\begin{equation}\label{eq:eig2}
b_i(q_j^i,q)=\zeta_j^i s^2_i(q_j^i,q)
\end{equation}
for all $q \in Q(K_i)$, where 
$$a_i(u,v) := \int_{K_i} \sigma(u) : \varepsilon(v)\dx, \quad b_i(p,q) := \int_{K_i} \frac{\kappa}{\nu} \nabla p\cdot \nabla q \dx,$$ 
$$s^1_i(u,v):=\int_{K_i} \tilde{\sigma}u\cdot v\dx, \quad s^2_i(p,q):=\int_{K_i} \tilde{\kappa}pq\dx,$$ 
\begin{equation*}
\tilde\sigma := \sum_{i=1}^{\Nv}(\lambda+2\mu) | \nabla \chi_i^1 |^2,\tand
\tilde\kappa := \sum_{i=1}^{\Nv}\frac{\kappa}{\nu} | \nabla \chi_i^2 |^2.
\end{equation*}
The functions ${\chi_i^1}$ and ${\chi_i^2}$ are neighborhood-wise defined
partition of unity functions \cite{bm97} on the coarse grid. To be more precise, for $k \in \{1,2\}$ the function $\chi_i^{k}$ satisfies $H \abs{\nabla \chi_i^{k}} = O(1)$, $0 \leq \chi_i^{k+1} \leq 1$, and $\sum_{i=1}^{\Nv} \chi_i^{k} = 1$. One can take $\{ \chi_i^k \}_{i=1}^{\Nv}$ to be the set of  standard multiscale basis functions or the standard piecewise linear functions. We remark that the actual computation of \eqref{eq:eig1} and \eqref{eq:eig2} is done based on the underlying fine grid $\mathcal{T}^h$ on each coarse element $K_i$. 
 
Assume that the eigenvalues $\{\lambda_j^i\}$ (resp. $\{ \zeta_j^i \}$) are arranged in ascending order and that the eigenfunctions satisfy the normalization condition $s_i^1(v_j^i,v_j^i)=1$ (resp. $s_i^2(q_j^i,q_j^i)=1$) for any $i$ and $j$. 
Next, choose  $J_i^1\in \mathbb{N}^+$ and define the local auxiliary space $V_{\text{aux}}(K_i):= \text{span} \{v_j^{i}:1\leq j \leq J_i^1 \}$. 
Similarly, we choose $J_i^2 \in \mathbb{N}^+$ and define $Q_{\text{aux}}(K_i) := \text{span} \{q_j^i: 1 \leq j \leq J_i^2\}$. 
Based on these local spaces, we define the global auxiliary spaces $V_{\text{aux}}$ and $Q_{\text{aux}}$ by
$$V_{\text{aux}} := \bigoplus_{i=1}^{\Ne} V_{\text{aux}}(K_i) \quad \text{and} \quad Q_{\text{aux}} := \bigoplus_{i=1}^{\Ne} Q_{\text{aux}}(K_i).$$
The corresponding inner products for the global auxiliary multiscale spaces are defined by
\begin{eqnarray*}
s^1(u,v):=\sum_{i=1}^{\Ne}s_i^1(u,v) \tand 
s^2(p,q):=\sum_{i=1}^{\Ne}s_i^2(p,q) 
\end{eqnarray*}
for all $u,v \in V_{\text{aux}}$ and $p,q\in Q_{\text{aux}}$.

Further, we define projection operators $\pi_1: V \to V_{\text{aux}}$ and $\pi_2: Q \to Q_{\text{aux}}$ such that for all $v \in V$ and $q \in Q$ we have 
\begin{eqnarray*}
\pi_1(v):=\sum_{i=1}^{\Ne}\sum_{j=1}^{J_i^1}s^1_i(v,v_j^i)v_j^i \tand
\pi_2(q):=\sum_{i=1}^{\Ne}\sum_{j=1}^{J_i^2}s^2_i(q,q_j^i)q_j^i.
\end{eqnarray*}
We also denote the kernel of the operator $\pi_1$ and the kernel of the operator $\pi_2$ to be 
$$
\tilde{V}:=\{w\in V\mid \pi_1(w)=0\}\quad \text{and} \quad 
\tilde{Q}:=\{q\in Q\mid \pi_2(q)=0\}.
$$

\subsection{Multiscale spaces}
In this subsection, we construct the multiscale spaces for the practical computations. 
For each coarse element $K_i$, we define
the oversampled region $K_{i,\ell}\subset\Omega$ obtained by enlarging $K_i$ by $\ell$ layers, i.e.,
\begin{equation}
\label{eq:K_il}
 K_{i,0} := K_i, \quad K_{i,\ell} := \bigcup \left\{ K\in \mathcal{T}^H : K \cap  K_{i,\ell-1} \neq \emptyset \right\}, \quad \ell \in \mathbb{N}^+.
 \end{equation}
We define $V(K_{i,\ell}):=[H_0^1(K_{i,\ell})]^d$ and $Q(K_{i,\ell}):=H^1_0(K_{i,\ell})$. Then, for each pair of auxiliary functions $v_j^i\in V_{\text{aux}}$ and $q_j^i\in Q_{\text{aux}}$,
we solve the following  minimization problems:
find $\psi_{j,\text{ms}}^{(i)}\in V(K_{i,\ell})$ such that
\begin{equation}\label{eq:mineq1}
\psi_{j,\text{ms}}^{(i)}=\text{argmin}\Big\{a(\psi,\psi)+s^1\big(\pi_1(\psi)-v_j^i,\pi_1(\psi)-v_j^i\big):\,\psi\in V(K_{i,\ell})\Big\}
\end{equation}
and 
find $\phi_{j,\text{ms}}^{(i)} \in Q(K_{i,\ell})$ such that
\begin{equation}\label{eq:mineq2}
\phi_{j,\text{ms}}^{(i)}=\text{argmin}\Big\{b(\phi,\phi)+s^2\big(\pi_2(\phi)-q_j^i,\pi_2(\phi)-q_j^i\big):\,\phi\in Q(K_{i,\ell})\Big\}.
\end{equation}
Note that problem (\ref{eq:mineq1}) is equivalent to the local problem 
\begin{equation}\label{eq:mineq1_loc}
a(\psi_{j,\text{ms}}^{(i)},v)+s^1\left (\pi_1(\psi_{j,\text{ms}}^{(i)}),\pi_1(v)\right)=s^1\left(v_j^i,\pi_1(v)\right)
\end{equation}
for all $v\in V(K_{i,\ell})$, whereas problem (\ref{eq:mineq2}) is equivalent to
\begin{equation}\label{eq:mineq2_loc}
b(\phi_{j,\text{ms}}^{(i)},q)+s^2\left(\pi_2(\phi_{j,\text{ms}}^{(i)}),\pi_2(q)\right)=s^2\left(q_j^i,\pi_2(q)\right)
\end{equation}
for all $q\in Q(K_{i,\ell})$.
Finally, for fixed parameters $\ell$, $J_i^1$, and $J_i^2$, the multiscale spaces $V_{\text{ms}}$ and $Q_{\text{ms}}$ are defined by 
$$ V_{\text{ms}} := \text{span}\{ \psi_{j,\text{ms}}^{(i)}: 1\leq j \leq J_i^1,\ 1\leq i \leq \Ne \} \quad \text{and} \quad Q_{\text{ms}} :=\text{span} \{ \phi_{j,\text{ms}}^{(i)}: 1 \leq j \leq J_i^2,\ 1\leq i \leq \Ne\}.$$
In practice, the multiscale basis functions defined above are constructed based on the underlying fine grid imposed on the oversampled regions. See Appendix \ref{appen:basis} for more details about the actual implementation. 

We remark that the multiscale basis functions can be interpreted as approximations to global multiscale basis functions $\psi_j^{(i)}\in V$ and $\phi_j^{(i)}\in Q$, similarly defined by
 \begin{align}\label{eq:minglo1}
\psi_j^{(i)} &= \text{argmin}\Big\{a(\psi,\psi)+s^1\big(\pi_1(\psi)-v_j^i,\pi_1(\psi)-v_j^i\big): \,\psi\in V\Big\}, \\\label{eq:minglo2}
\phi_j^{(i)}&= \text{argmin}\Big\{b(\phi,\phi)+s^2\big(\pi_2(\phi)-q_j^i,\pi_2(\phi)-q_j^i\big): \,\phi\in Q\Big\}.
\end{align}
This is equivalent to solve the global problems
\begin{align}\label{eq:minglo1eqn}
a(\psi_{j}^{(i)},v)+s^1\big(\pi_1(\psi_{j}^{(i)}),\pi_1(v)\big)&=s^1\left(v_j^i,\pi_1(v)\right), \\
\label{eq:minglo2eqn}
b(\phi_{j}^{(i)},q)+s^2\left(\pi_2(\phi_{j}^{(i)}),\pi_2(q)\right)&=s^2\left(q_j^i,\pi_2(q)\right), 
\end{align}
for any $v \in V$ and $q \in Q$. 
These basis functions have global support in the domain $\Omega$ but, as shown in \cite{chung2017constraint}, decay exponentially outside some local (oversampled) region. This property plays a vital role in the convergence analysis of CEM-GMsFEM and justifies the use of local basis functions in $V_{\text{ms}}$ and $Q_{\text{ms}}$. 

\subsection{The CEM-GMsFEM method}
Once the multiscale spaces are constructed, the offline solution can be obtained by solving 
the following fully discrete scheme: find $\{ (u_{\text{ms}}^n,p_{\text{ms}}^n) \}_{n=1}^N \subset V_{\text{ms}} \times Q_{\text{ms}}$ such that 
\begin{subequations}\label{eq:weak2}
	\begin{alignat}{2}
	a(u_{\text{ms}}^n,v) - d(v,p_{\text{ms}}^n) &= 0, \\
	d(D_{\tau}u_{\text{ms}}^n,q) + c(D_{\tau}p_{\text{ms}}^n,q) + b(p_{\text{ms}}^n,q) &= (f^n,q),      
	\end{alignat} 
\end{subequations}
for all $(v,q) \in V_{\text{ms}} \times Q_{\text{ms}}$ with initial condition $u_{\text{ms}}^0 \in V_{\text{ms}}$ and $p_{\text{ms}}^0\in Q_{\text{ms}}$ defined by
$$b(p^0 - p_{\text{ms}}^0,q) = 0 \quad \text{and} \quad a(u_{\text{ms}}^0,v)=d(v,p_{\text{ms}}^0)$$
for all $q\in Q_{\text{ms}}$ and $v \in V_{\text{ms}}$. 

\section{Online Adaptive Method}\label{sec:online}
In this section, we present the online adaptive method for the poroelasticity problem within the CEM-GMsFEM. 
In general, one may need some post-processing procedures (such as adaptive enrichment of degrees of freedom) to achieve a better approximation in the multiscale space. 
The online algorithm starts with the construction of the online basis functions driven by the residual information. 
In practice, the online basis construction is implemented with an adaptive algorithm so that the online basis functions can help reduce the error to give a more accurate multiscale approximation. 

\subsection{Online basis functions}
We present the construction of the online basis functions for displacement and pressure variables. First, we define the residual functionals $r_n^1: V \to \mathbb{R}$ and $r_n^2 : Q\to \mathbb{R}$ at the time level $t_n$. Let $u_{\text{app}}$ and  $p_{\text{app}}$ be the current numerical approximations of displacement and pressure, respectively. That is, $u_{\text{app}} := u_{\text{ms}}^n$ and $p_{\text{app}} := p_{\text{ms}}^n$, which are obtained from \eqref{eq:weak2}. 
The residual functionals (at the time level $t_n$) are defined to be
\begin{equation}
\begin{split}
r_n^1(v) &:= d(v, p_{\text{app}}) - a(u_{\text{app}},v),  \\
r_n^2(q) &:= (f^n, q) - b(p_{\text{app}}, q) - c(D_\tau p_{\text{app}}, q) - d(D_\tau u_{\text{app}},q),
\end{split}
\label{eqn:glo_residual}
\end{equation}
for any $v \in V$ and $q \in Q$. 
In order to develop local error indicators, 
we also consider local residuals. For each coarse node $x_i$, we define a coarse neighborhood $\omega_i := \bigcup \{ K: x_i \in K, ~ K \in \mathcal{T}^H \}$. For each coarse neighborhood $\omega_i$, we define the local residual functionals $r_{n,i}^1: V  \to \mathbb{R}$ and $r_{n,i}^2 : Q \to \mathbb{R}$ such that 
\begin{equation}
r_{n,i}^1 (v) := r_{n}^1(\chi_{\omega_i}^1 v) 
\quad \text{and} \quad 
r_{n,i}^2(q) := r_{n}^2 (\chi_{\omega_i}^2 q)
\label{eqn:loc_residual}
\end{equation}
for any $v \in V$ and $q \in Q$. The functions ${\chi_{\omega_i}^1}$ and ${\chi_{\omega_i}^2}$ are neighborhood-wise defined
partition of unity functions \cite{bm97} on the coarse neighborhood $\omega_i$. 

Next, we define the online basis functions for displacement and pressure. 
The construction of the online basis function is related to the local residual. 
For any coarse neighborhood $\omega_i$ associated with the coarse node $x_i$, we define $\omega_{i,\ell}$ such that 
\begin{equation}
\label{eq:omega_i_neighbor}
 \omega_{i,\ell} := \left \{ \begin{array}{ll}
\omega_i & \ell = 0, \\
\displaystyle{\bigcup} \{ K \in \mathcal{T}^H : \omega_{i,l-1} \cap K \neq \emptyset \} & \ell \geq 1,
\end{array} \right.
\end{equation}
for any nonnegative integer $\ell \in \mathbb{N}$. 
For simplicity, we denote $\omega_i^+ = \omega_{i,\ell}$ the oversampled region with respect to $\omega_i$ with $\ell$ level of spatial enrichment.

Using the local residuals, we define the online basis functions $\delta_\ms^{(i)} \in V(\omega_i^+) := \left [ H_0^1(\omega_i^+) \right ]^d$ (for displacement) and $\rho_{\text{ms}}^{(i)} \in Q(\omega_i^+) := H_0^1(\omega_i^+)$ (for pressure) such that they are both supported in the oversampled region $\omega_i^+$. 
More precisely, the online basis functions $\delta_{\ms}^{(i)} \in V(\omega_i^+)$ and $\rho_{\ms}^{(i)} \in Q(\omega_i^+)$ are defined to be the solutions of the following (quasi-local) cell problems: 
\begin{equation}
\begin{split}
a (\delta_{\ms}^{(i)}, v) + s^1\left (\pi_1 (\delta_{\ms}^{(i)}), \pi_1(v) \right )&= r_{n,i}^1 (v) \quad \text{ for all } v \in V(\omega_i^+), \\
b(\rho_{\ms}^{(i)}, q) + s^2 \left ( \pi_2 ( \rho_{\ms}^{(i)}), \pi_2 (q) \right ) &= r_{n,i}^2(q) \quad \text{ for all } q \in Q(\omega_i^+).
\end{split}
\label{eqn:loc_online_basis}
\end{equation}
We remark that the online basis functions satisfy the homogeneous boundary condition on the boundary of the oversampled region since the global variational formulation is discretized within the continuous Galerkin setting. One can also construct discontinuous type of online basis functions (see \cite{pun2021online} for more details). 

In fact, these functions are localizations of the corresponding global online basis functions $\delta_{\glo}^{(i)}\in V$ and $\rho_{\glo}^{(i)}\in Q$ satisfying 
\begin{equation}
\begin{split}
a (\delta_{\glo}^{(i)}, v) + s^1\left (\pi_1 (\delta_{\glo}^{(i)}), \pi_1(v) \right )&= r_{n,i}^1 (v) \quad \text{ for all } v \in V, \\
b(\rho_{\glo}^{(i)}, q) + s^2 \left ( \pi_2 ( \rho_{\glo}^{(i)}), \pi_2 (q) \right ) &= r_{n,i}^2(q) \quad \text{ for all } q \in Q.
\end{split}
\label{eqn:glo_online_basis}
\end{equation}

We can construct local online basis functions for each $ r_{n,i}^1 (v)$ and $ r_{n,i}^2 (v)$ or for some selected $ r_{n,i}^1 (v)$ and $ r_{n,j}^2 (v)$ (with $i\in \mathcal{I}_1$ and $j\in \mathcal{I}_2$ for some index sets $\mathcal{I}_1$ and $\mathcal{I}_2$). After constructing the online basis functions, we can add them into our multiscale spaces for enrichment of degrees of freedom. 

In practice, similar to the one presented in \cite{chung2018fast}, one has to divide the actual construction of the online basis functions into several sub-iterations, such that the selected coarse neighborhoods are mutually exclusive in each sub-iteration. 
We also remark that one can construct the online basis functions on coarse oversampled regions $K_{i,\ell}$, which is an oversampled version of coarse element $K_i$ with oversampling parameter $\ell$. It leads to a simpler algorithmic implementation since any two coarse elements whose interior parts are non-overlapped. We will compare these two different strategies in the next section. 

\subsection{Online adaptive enrichment}\label{sec:algorithm}
In this section, we present an online adaptive algorithm with enrichment of online basis functions defined in the previous section. Once the online basis functions are constructed, we include those newly constructed functions into the multiscale space. For time-dependent problem like poroelasticity model, one may need online enrichment every several time levels (or whenever the global residual is larger than some given threshold). With this enriched space, we can compute a new numerical solution by solving the equation \eqref{eq:weak2}. 

First, we initially set the multiscale spaces to be $V_{\ms}^{(0)} := V_{\ms}$ and $Q_{\ms}^{(0)} := Q_{\ms}$. 
Next, we choose parameters $\theta \in [0, 1)$ and $\gamma \in [0,1)$, which determine the portion of online basis functions that are included in the spaces during each iteration. Let $V_{\ms}^{(k)}$ and $Q_{\ms}^{(k)}$ be the multiscale spaces after $k$ times of enrichments.
The online adaptive algorithm sequentially defines residual functionals by taking $u_{\text{app}} = u_{\ms}^{n,k}$ and $p_{\text{app}} = p_{\ms}^{n,k}$ in \eqref{eqn:glo_residual} with $u_{\ms}^{n,k}$ and $p_{\ms}^{n,k}$ being the solutions of \eqref{eq:weak2} over the multiscale spaces $V_{\ms}^{(k)}$ and $Q_{\ms}^{(k)}$. That is, we define $r_{n}^{1,k}$ and $r_{n}^{2,k}$ the global residual operators at $k$-th level of enrichment such that 
\begin{equation}
\begin{split}
r_n^{1,k}(v) &:= d(v, p_{\ms}^{n,k}) - a(u_{\ms}^{n,k},v),  \\
r_n^{2,k}(q) &:= (f^n, q) - b(p_{\ms}^{n,k}, q) - c(D_\tau p_{\ms}^{n,k}, q) - d(D_\tau u_{\ms}^{n,k},q), 
\end{split}
\label{eqn:glo_residual_k}
\end{equation}
for any $v \in V$ and $q \in Q$ and their associated localization
\begin{equation}
r_{n,i}^{1,k} (v) := r_n^{1,k}(\chi_{\omega_i}^1 v) \tand r_{n,i}^{2,k} (q) := r_n^{2,k}(\chi_{\omega_i}^2 q)
\label{eqn:loc_residual_k}
\end{equation}
for any $v \in V$ and $q \in Q$. 
This online adaptive method enriches the multiscale spaces 
$V_{\ms}^{(k)} \subset V_{\ms}^{(k+1)}$ and $Q_{\ms}^{(k)} \subset Q_{\ms}^{(k+1)}$ by adding online basis functions defined in \eqref{eqn:loc_online_basis}, and 
generates the updated multiscale solutions $u_{\ms}^{n,k+1} \in V_{\ms}^{(k+1)}$ and $p_{\ms}^{n,k+1} \in Q_{\ms}^{(k+1)}$ by Galerkin projection. We also define two local residuals $z_{n,i}^{1,k}$ and $z_{n,i}^{2,k}$ to be 
\begin{equation}
z_{n,i}^{1,k} (v) := r_n^{1,k}(\mathbf{1}_{\omega_i} v) \tand z_{n,i}^{2,k} (q) := r_n^{2,k}(\mathbf{1}_{\omega_i} q)
\label{eqn:loc_residual_indi}
\end{equation}
for any $v \in V$ and $q \in Q$. Here, $\mathbf{1}_{\omega_i}$ is the indicator function of the coarse neighborhood $\omega_i$. 
The complete procedure of the (neighborhood-based) online adaptive algorithm is listed in Algorithm \ref{algo:online}.

\begin{algorithm}[ht]
	\caption{Neighborhood-based Online adaptive algorithm for poroelasticity model}
	\begin{algorithmic}[1]
	\STATE {\bf Input:} A given time level $t^n$, a source term $f^n$, current approximation $(u_{\text{ms}}^n, p_{\text{ms}}^n)$, current multiscale spaces $V_{\text{ms}}$ and $Q_{\text{ms}}$, parameters $(\theta, \gamma, \ell) \in [0,1) \times [0,1) \times \mathbb{N}$, a tolerance accuracy $\texttt{Tol} >0$, and a small threshold $\varepsilon >0$. 
	\STATE Set $k = 0$, $V_{\ms}^{(k)} = V_{\ms}$, $Q_{\ms}^{(k)} = Q_{\ms}$, $u_{\text{ms}}^{n,k} = u_{\text{ms}}^n$, and $p_{\text{ms}}^{n,k} = p_{\text{ms}}^n$. 
	\STATE Compute global residual $\eta := \norm{r_n^{1,k}}_{a'} + \norm{r_n^{2,k}}_{b'}$. 
	\WHILE{$\abs{\eta - \texttt{Tol}} \geq \varepsilon$}
		\FOR{$i \in \{ 1, \cdots, \Nv \}$} 
			\STATE compute $\eta_{n,i}^{1,k} := \norm{z_{n,i}^{1,k}}_{a'}$ and $\eta_{n,i}^{2,k} := \norm{z_{n,i}^{2,k}}_{b'}$ (defined in \eqref{eqn:loc_residual_indi}). 
		\ENDFOR
		\STATE Enumerate the indices such that $\eta_{n,1}^{1,k} \geq \eta_{n,2}^{1,k} \geq \cdots \geq \eta_{n, \Nv}^{1,k}$ and $\eta_{n,1}^{2,k} \geq \eta_{n,2}^{2,k} \geq \cdots \geq \eta_{n, \Nv}^{2,k}$. 
		\STATE Find the smallest integers $m_1 = m_1(k)$ and $m_2 = m_2(k)$ such that 
		$$\sum_{i=m_1+1}^{\Nv} \left (\eta_{n,i}^{1,k} \right )^2 < \theta \sum_{i=1}^{\Nv} \left (\eta_{n,i}^{1,k} \right )^2 \tand 
		\sum_{i=m_2+1}^{\Nv} \left (\eta_{n,i}^{2,k} \right )^2 < \gamma \sum_{i=1}^{\Nv} \left (\eta_{n,i}^{2,k} \right )^2.$$
		\FOR{$i \in \{1,\cdots, m_1\}$} 
			\STATE Construct $\delta_{n}^{(i,k)}$ by solving 
			$a (\delta_{n}^{(i,k)}, v) + s^1\left (\pi_1 (\delta_{n}^{(i,k)}), \pi_1(v) \right )= r_{n,i}^{1,k} (v)$ for any $v \in V_h^0(\omega_i^+).$
		\ENDFOR
		\FOR{$i \in \{1,\cdots, m_2\}$} 
			\STATE Construct $\rho_{n}^{(i,k)}$ by solving 
			$b (\rho_{n}^{(i,k)}, q) + s^2\left (\pi_2 (\rho_{n}^{(i,k)}), \pi_2(q) \right )= r_{n,i}^{2,k} (q)$ for any $q \in Q_h^0(\omega_i^+).$
		\ENDFOR		
		\STATE Set $V_{\ms}^{(k+1)} = V_{\ms}^{(k)} \bigoplus \spa\left \{ \delta_{n}^{(i,k)}\right \}_{i=1}^{m_1}$ and $Q_{\ms}^{(k+1)} = Q_{\ms}^{(k)} \bigoplus \spa\left \{ \rho_{n}^{(i,k)}\right \}_{i=1}^{m_2}$. 
		\STATE Obtain $(u_{\ms}^{n,k+1}, p_{\ms}^{n,k+1})\in V_{\ms}^{(k+1)} \times Q_{\ms}^{(k+1)}$ by solving \eqref{eq:weak2} over the new spaces. 
		\STATE Update global residual $\eta'$ with new approximation $(u_{\ms}^{n,k+1}, p_{\ms}^{n,k+1})$. 
		\IF{$\abs{\eta - \eta'} \leq \varepsilon$} 
			\STATE Break the online enrichment process. 
		\ENDIF
		\STATE Set $k \leftarrow k+1$. 
	\ENDWHILE
	\STATE {\bf Output:} A (possibly) new multiscale approximation $(u_{\ms}^{n,k}, p_{\ms}^{n,k})$. Move to $t^{n+1}$. 
	\end{algorithmic}
	\label{algo:online}
\end{algorithm}

\section{Numerical Results}\label{sec:num}

In this section, we present some numerical results obtained by using the proposed online adaptive method. The following notations are used throughout this section. 
We denote $n$ the level of time discretization and $k$ the level of online enrichment. 
We also denote $u_{\text{dof}}^k$ and $p_{\text{dof}}^k$ the degrees of freedom for displacement and pressure, respectively. The
(relative) energy errors between the multiscale and the fine-scale solutions defined below 
by$$e_u :=\frac{\norm{u_{\ms}^{n,k}-u_h^n}_a}{\norm{u_h^n}_a} \tand e_p :=\frac{\norm{p_{\ms}^{n,k}-p_h^n}_b}{\norm{p_h^n}_b}$$
will be served as indicators evaluating the performance of algorithm. We remark that the errors $e_u$ and $e_p$ depend on the indices $n$ and $k$ but we simply omit them for simplicity. 
Moreover, instead of setting the thresholding parameters (i.e., \texttt{Tol} and $\varepsilon$ in the algorithm) in the numerical experiments, only fixed number of iterations will be executed for simplifying the algorithmic implementation. 

\begin{figure}[htbp!]
\centering
\includegraphics[width = 2.3in]{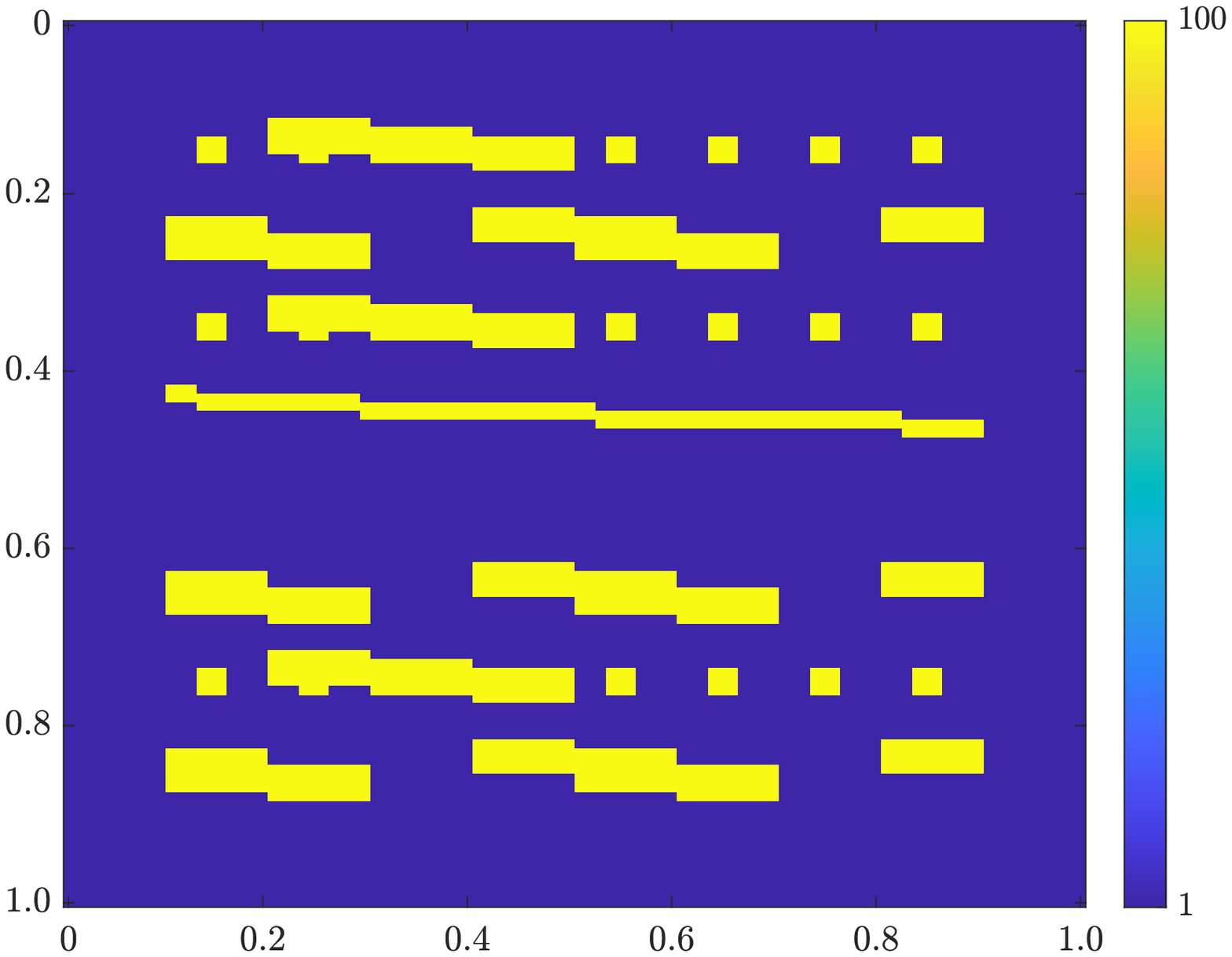}
\quad 
\includegraphics[width = 2.3in]{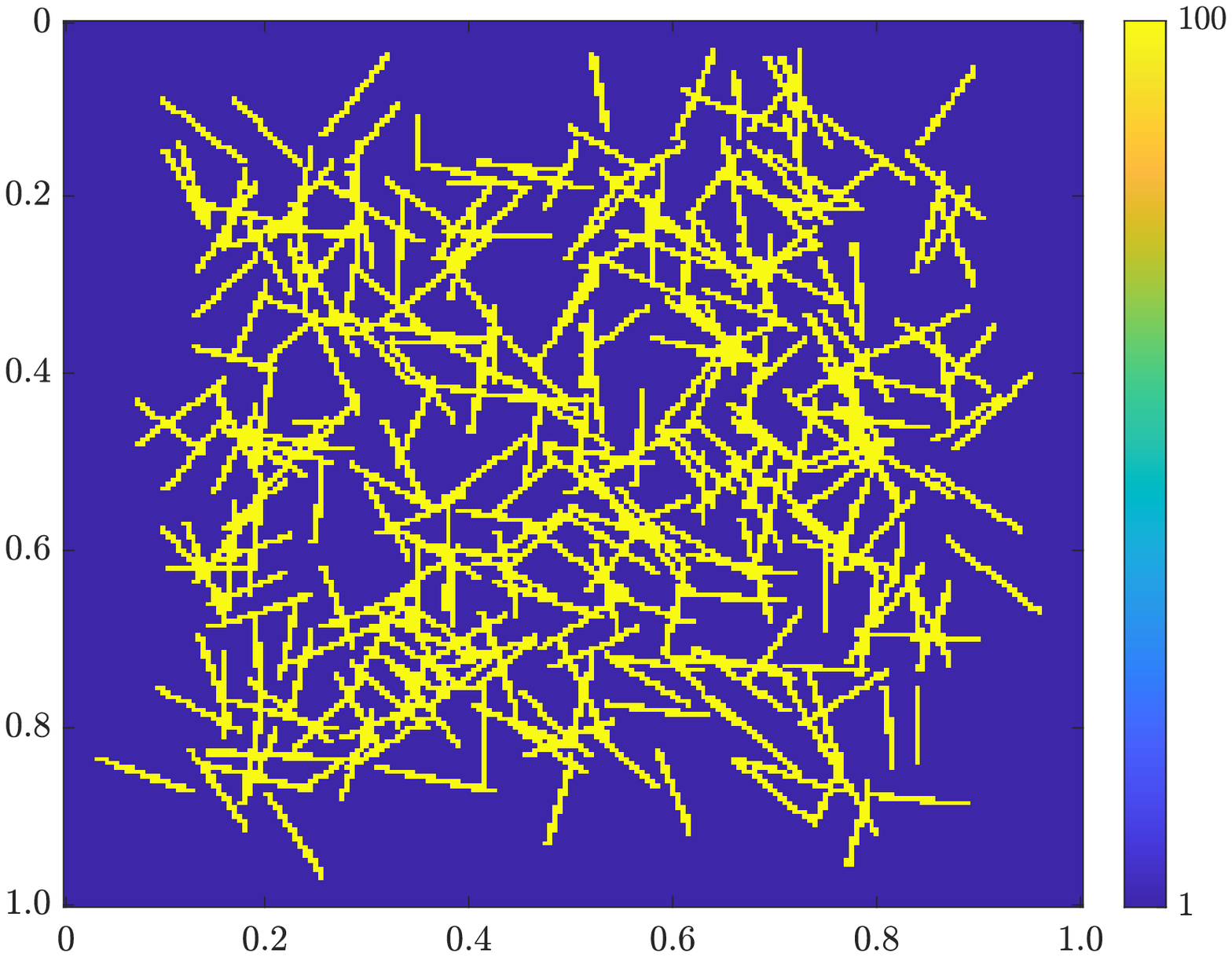}
\caption{Young's Modulus of Example \ref{exp:1} and \ref{exp:2}.}
\label{fig:ym_exp1}
\end{figure}

\begin{example}\label{exp:1}
The problem setting in this example reads as follows: 
\begin{enumerate}
\item We set $\Omega = (0,1)^2$, $\alpha = 0.9$, $M = 1$, $\nu_p = 0.2$ (Poisson ratio), and $\nu = 1$. 
\item The terminal time is $T = 1$. The time step is $\tau = 0.05$, i.e., $N = 20$. 
\item The Young's modulus $E$ is depicted in Figure \ref{fig:ym_exp1}. We set the permeability to be $\kappa = E$. 
\item The source function is $f \equiv 1$; the initial condition is $p(x_1, x_2) = 100x_1(1-x_1)x_2(1-x_2)$. 
\item The number of (local) offline (pressure and displacement) basis functions is $J = J_i^1 = J_i^2 = 2$ and the number of oversampling layers is $\ell = 2$ for both the offline and online stages. 
\item The coarse mesh is $10 \times 10$ and the overall fine mesh is $100 \times 100$. 
\item In this example, we perform the online enrichment only at the final time level $t = T = 1$ with online tolerance $\theta = \gamma=30\%$ and $\theta = \gamma=70\%$. 
\end{enumerate}
\end{example}
\noindent
In this example, we test and compare two different strategies for the online basis enrichment. The first approach is to enrich the multiscale space by traversing the oversampled regions $K_{i,\ell}$ (defined in \eqref{eq:K_il}) related to coarse elements $K_i$. 
We refer to this approach as the element-based strategy. 
The second one, called neighborhood-based strategy (see also Algorithm \ref{algo:online}), is traversing the oversampled regions $\omega_{i,\ell}$ (defined in \eqref{eq:omega_i_neighbor}) related to the coarse neighborhoods $\omega_i$. 
Within the neighborhood-based strategy, the actual computation of the online basis functions consists of several sub-iterations such that the selected coarse neighborhoods are mutually exclusive in each sub-iteration. 

Figure \ref{fig:5_1_soln_profile} depicts the solution profiles at the terminal time. In Tables \ref{exp:5_1_07_K} and \ref{exp:5_1_03_K}, we present the energy errors of the displacement and pressure with the degrees of freedom when using the element-based strategy with $\theta=\gamma=0.7$ and $\theta=\gamma=0.3$. Tables \ref{exp:5_1_07_w} and \ref{exp:5_1_03_w} show the energy errors along with the degrees of freedom using the neighborhood-based strategy. Both strategies lead to moderately fast convergence in terms of energy errors (within a few iterations) with these setting of algorithmic parameters. We observe that for any fixed strategy, choosing a smaller online tolerance parameter can result in a sharper error decay rate. 
Figures \ref{fig:5_1_07} and \ref{fig:5_1_03} show the trends of the energy errors against the number of online iterations. From these two figures we can see that the neighborhood-based strategy outperforms another one as it can achieve a good accuracy with fewer iterations.

We remark that one achieves $e_u=2.59\%$ and $e_p=6.83\%$ when the offline degrees of freedom are $(u_{\text{dof}}^{k}, p_{\text{dof}}^{k})=(2000, 2000)$, which is of full capability of offline multiscale approximation. Compared to this, using the online adaptive method to refine the multiscale approximation can achieve same level of accuracy with much fewer degrees of freedom.

\begin{figure}[htbp!]
\centering
\includegraphics[width = 1.7in]{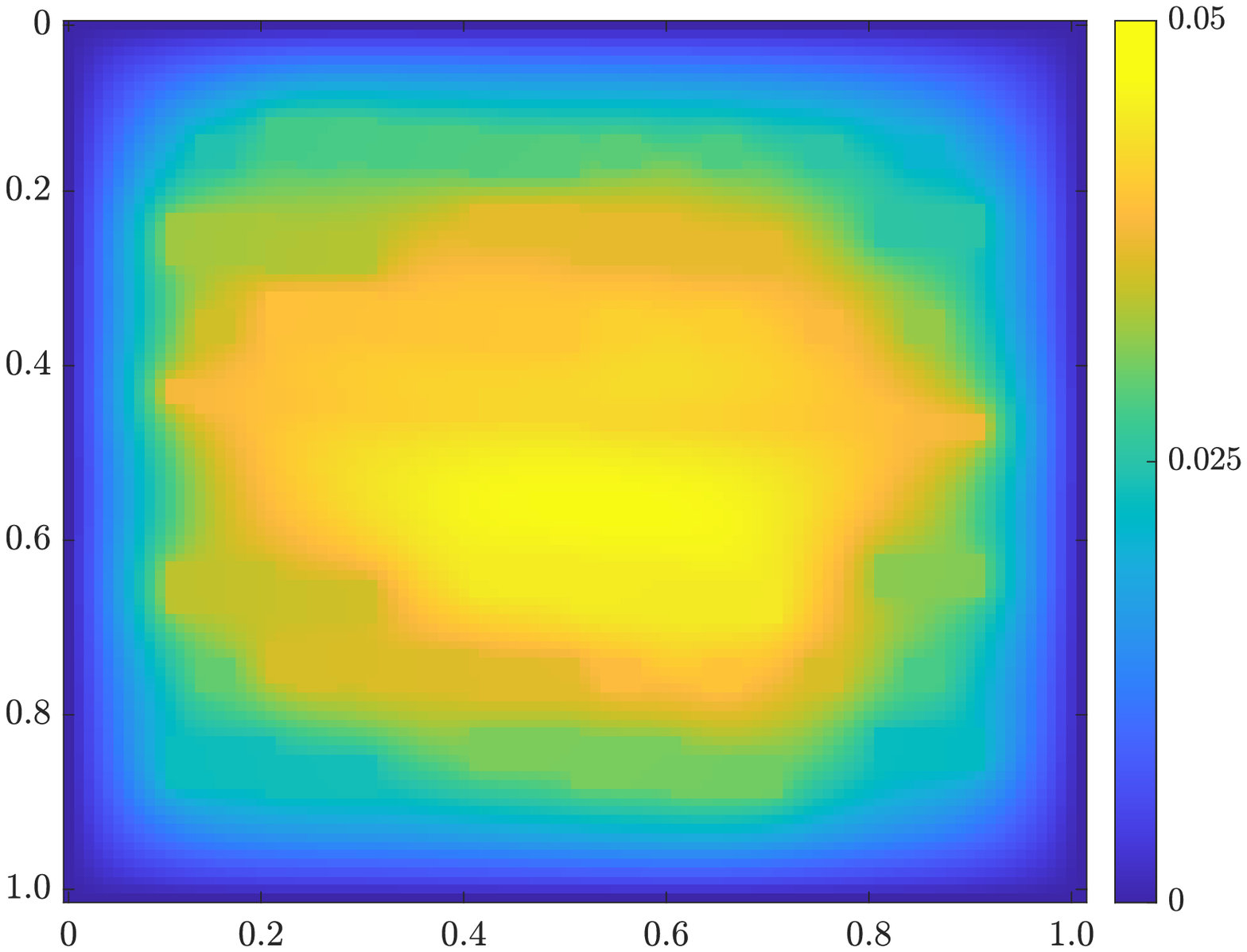} \quad
\includegraphics[width = 1.7in]{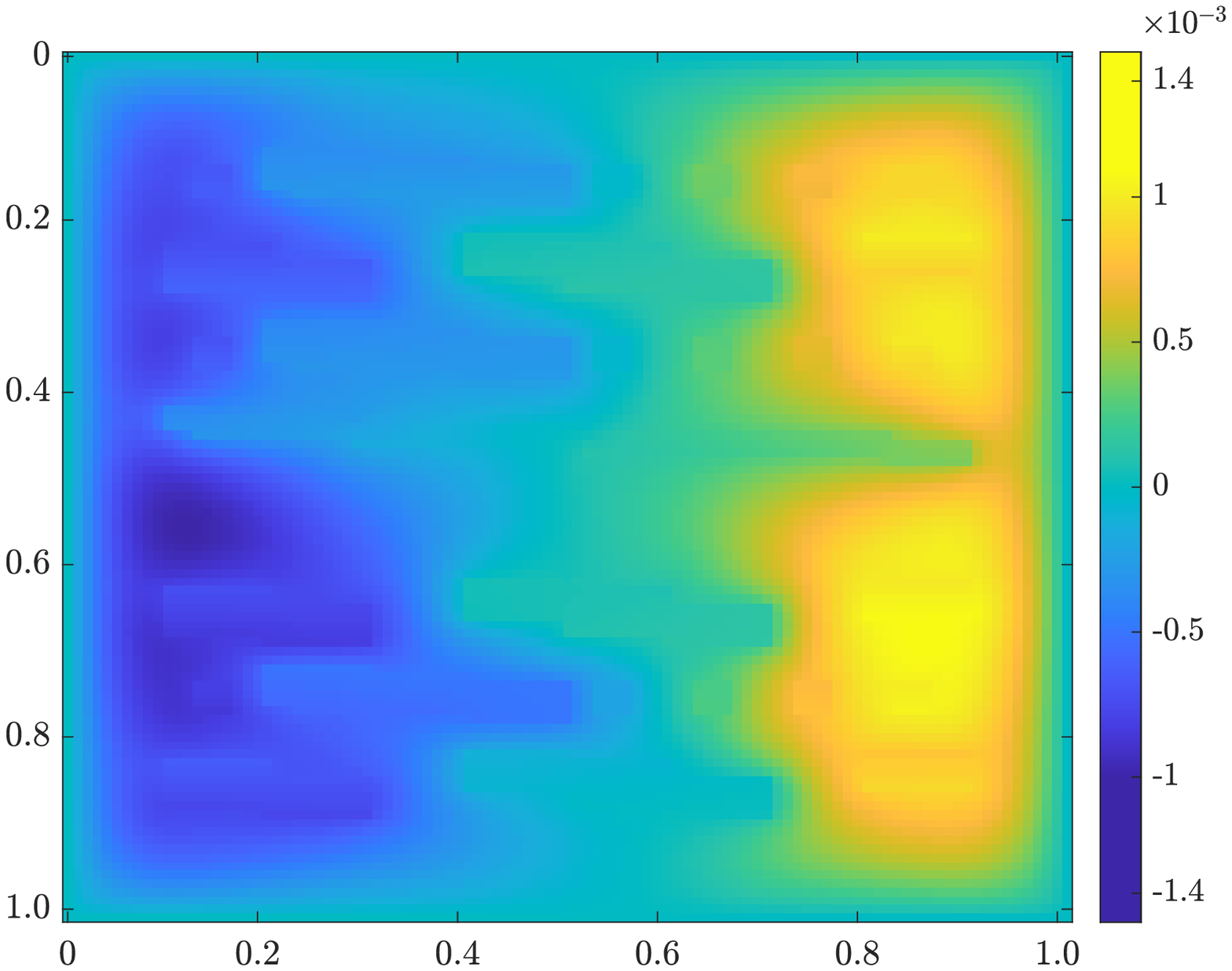}\quad
\includegraphics[width = 1.7in]{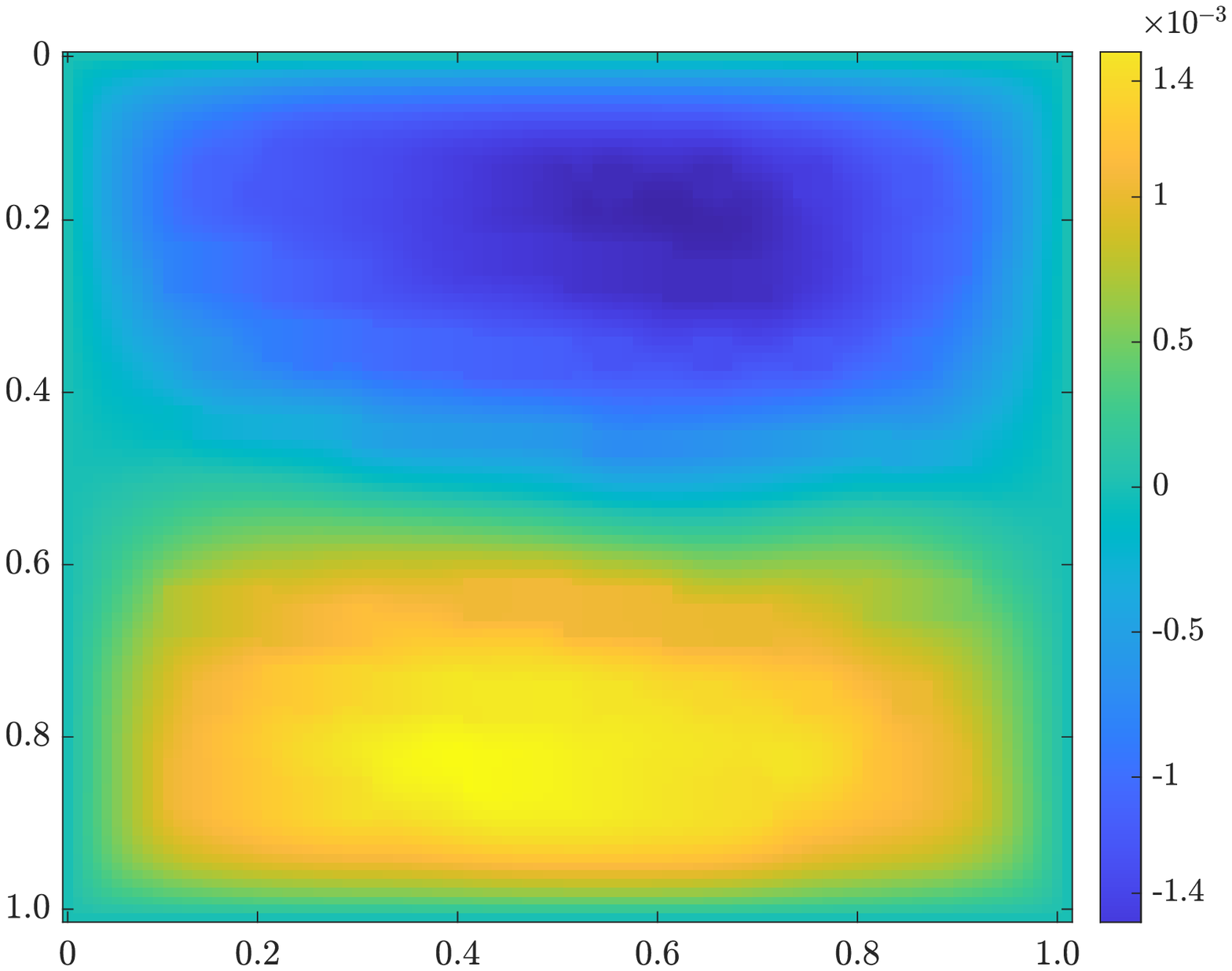}
\includegraphics[width = 1.7in]{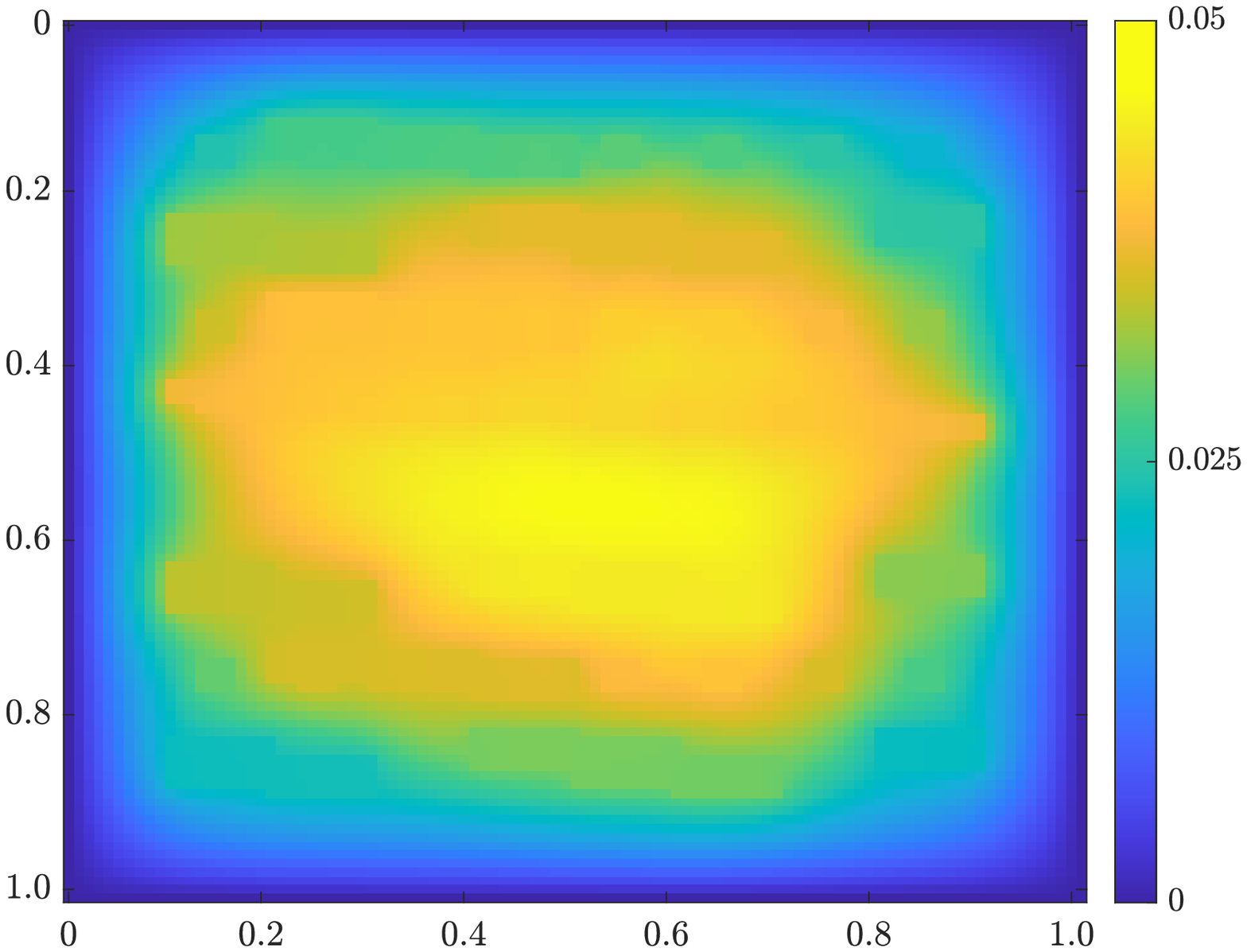}\quad
\includegraphics[width = 1.7in]{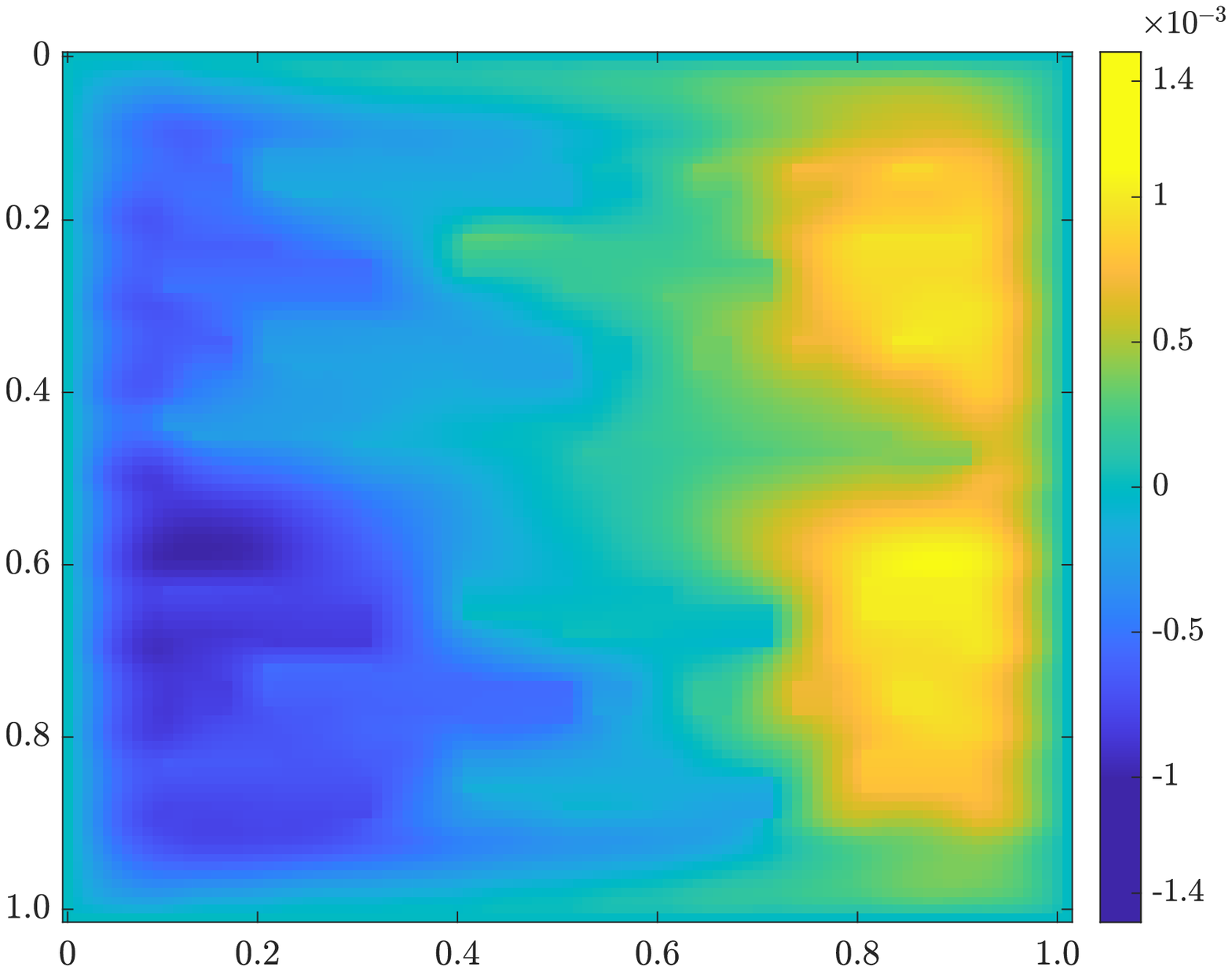}\quad
\includegraphics[width = 1.7in]{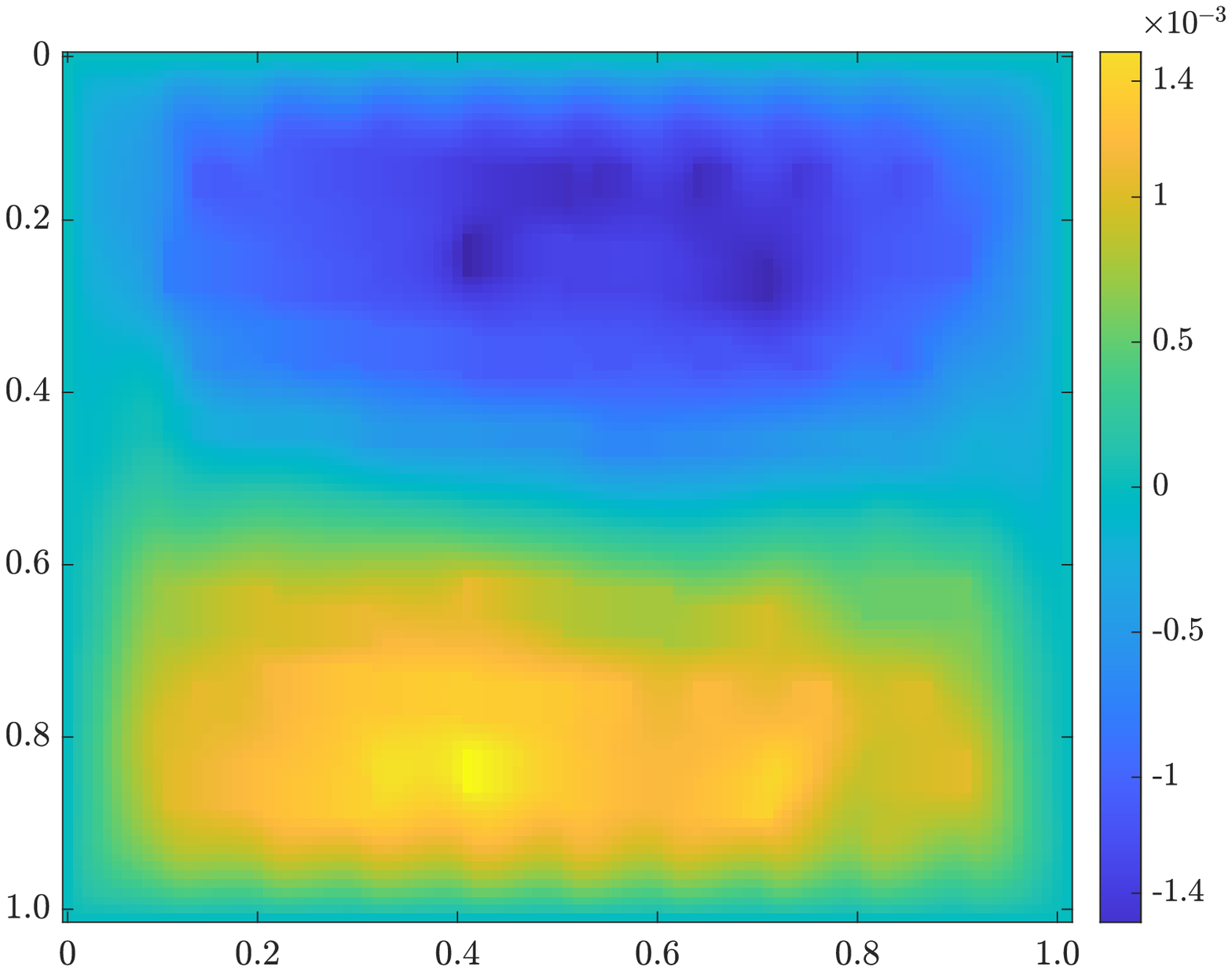}
\includegraphics[width = 1.7in]{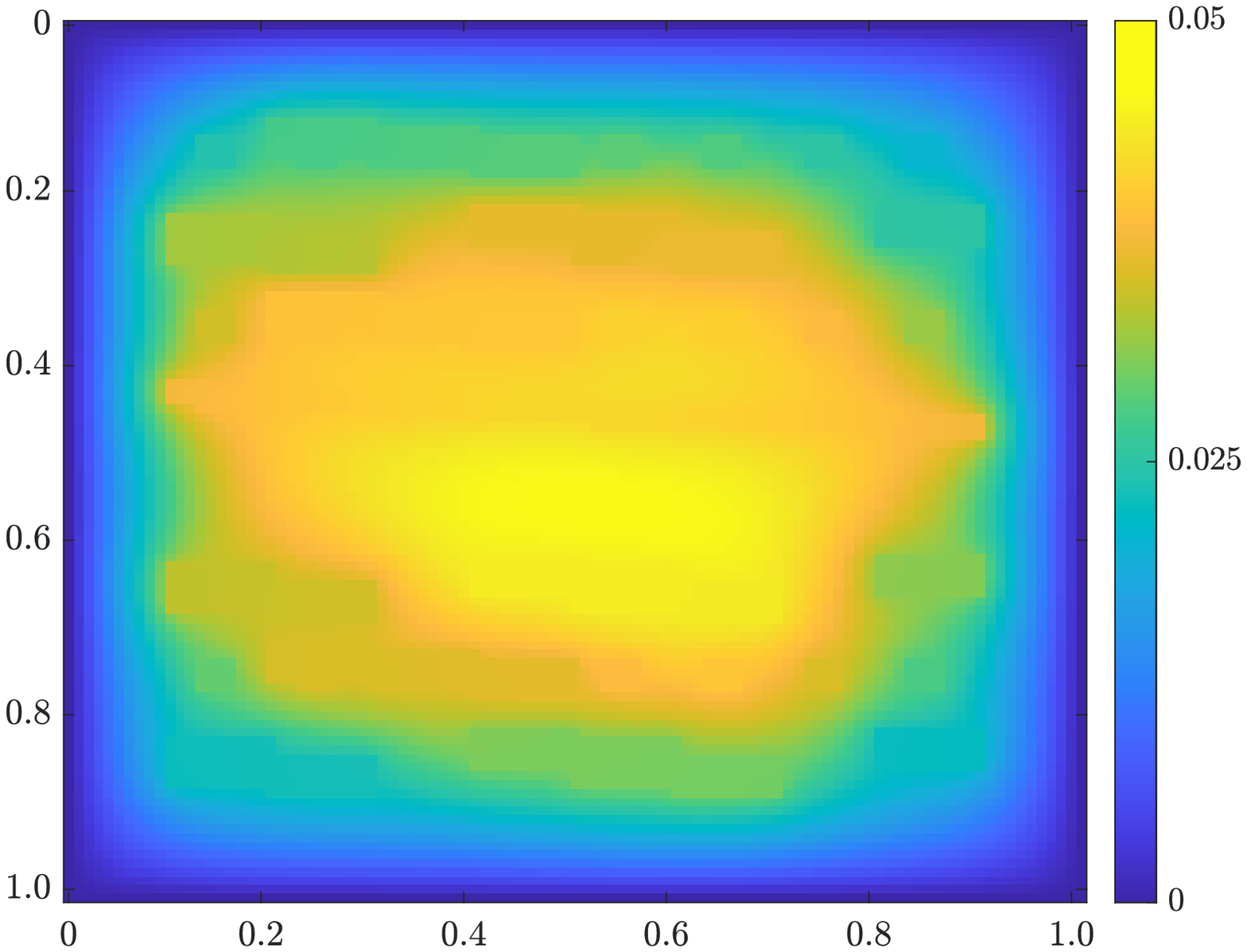}\quad
\includegraphics[width = 1.7in]{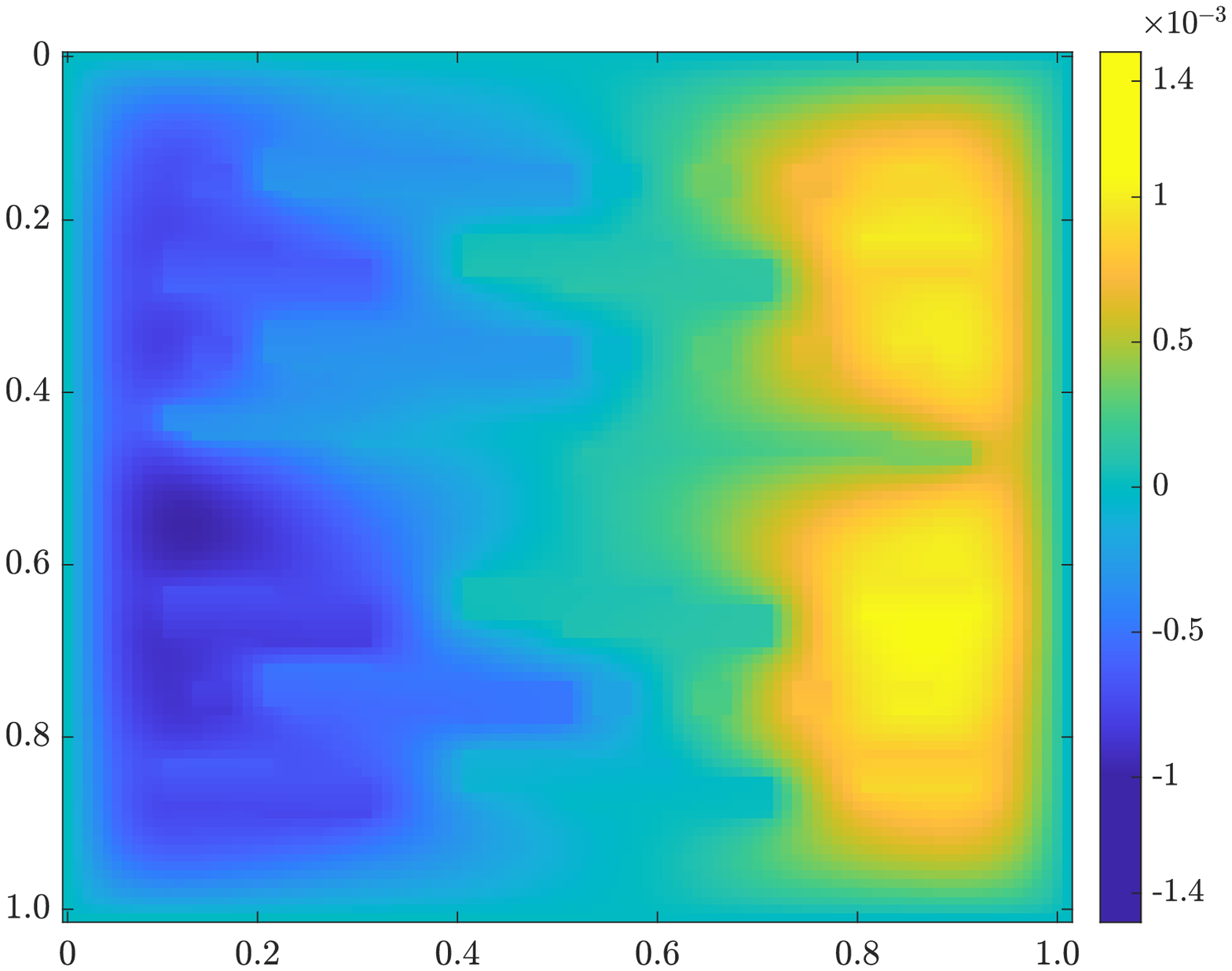}\quad
\includegraphics[width = 1.7in]{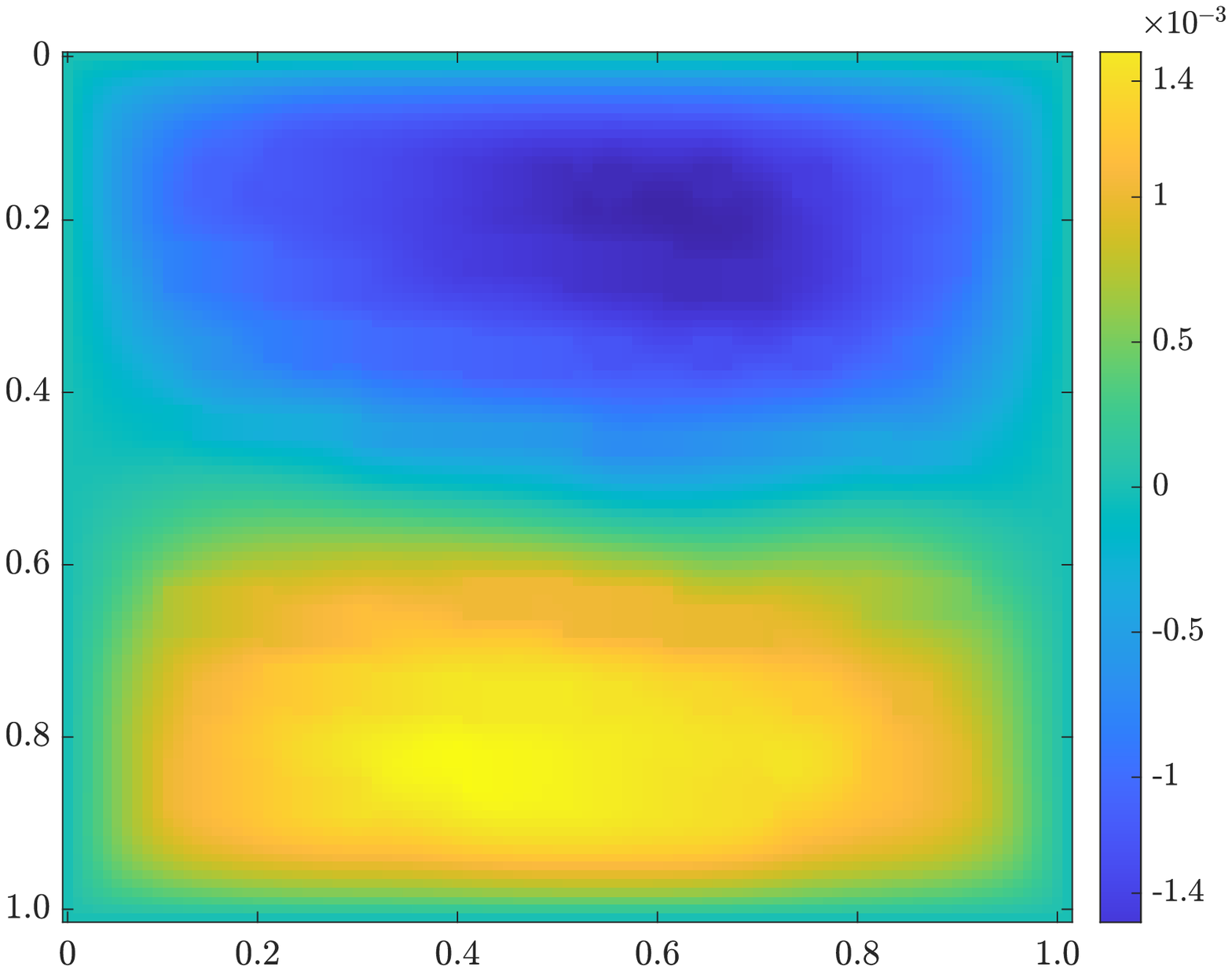}
\caption{Solution profiles for (starting from left to right) $p$, $u_1$, $u_2$ at $T=1$ of Example \ref{exp:1}. Top: reference solutions; middle: offline solutions; bottom: online solutions with iteration 3 and $\theta=\gamma=0.3$. }
\label{fig:5_1_soln_profile}
\end{figure}

\begin{table}[htbp!]
\centering
\begin{tabular}{c|c||c|c}
 $k$ & $(u_{\text{dof}}^{k}, p_{\text{dof}}^{k})$ & $e_u$ & $e_p$ \\ 
\hline
$0$ & $(200, 200)$ & $31.91\%$ & $12.45\%$ \\ 
$1$ & $(210, 206)$ & $25.05\%$ & $9.40\%$ \\ 
$2$ & $(222, 211)$ & $19.26\%$ & $6.60\%$ \\ 
$3$ & $(233, 215)$ & $14.09\%$ & $4.20\%$ \\ 
$4$ & $(244, 217)$ & $11.14\%$  & $3.52\%$ \\
$5$ & $(255, 221)$ & $8.27\%$  & $2.89\%$ \\
$6$ & $(266, 230)$ & $6.29\%$  & $2.38\%$ \\
$7$ & $(276, 238)$ & $4.74\%$  & $2.03\%$ \\
$8$ & $(286, 246)$ & $3.69\%$  & $1.74\%$ \\
$9$ & $(295, 252)$ & $3.01\%$  & $1.56\%$ \\
$10$ & $(305, 257)$ & $2.40\%$  & $1.44\%$ \\
\end{tabular}
\caption{History of online enrichment in Example \ref{exp:1} with $\theta=\gamma = 0.7$ and $K_i^+$ at time level $t=T=1$. }
\label{exp:5_1_07_K}
\end{table}

\begin{table}[htbp!]
\centering
\begin{tabular}{c|c||c|c}
$k$ & $(u_{\text{dof}}^{k}, p_{\text{dof}}^{k})$ & $e_u$ & $e_p$ \\ 
\hline
$0$ & $(200, 200)$ & $31.91\%$ & $12.45\%$ \\ 
$1$ & $(236, 219)$ & $14.66\%$ & $4.88\%$ \\ 
$2$ & $(267, 227)$ & $7.33\%$ & $2.90\%$ \\ 
$3$ & $(308, 258)$ & $3.43\%$ & $1.56\%$ \\ 
$4$ & $(357, 264)$ & $1.95\%$ & $1.47\%$ \\
$5$ & $(388, 264)$ & $1.52\%$  & $1.47\%$ \\
\end{tabular}
\caption{History of online enrichment in Example \ref{exp:1} with $\theta =\gamma= 0.3$ and $K_i^+$ at time level $t=T=1$. }
\label{exp:5_1_03_K}
\end{table}

\begin{table}[htbp!]
\centering
\begin{tabular}{c|c||c|c}
$k$ & $(u_{\text{dof}}^{k}, p_{\text{dof}}^{k})$ & $e_u$ & $e_p$ \\ 
\hline
$0$ & $(200, 200)$ & $31.91\%$ & $12.45\%$ \\ 
$1$ & $(213, 214)$ & $18.23\%$ & $6.82\%$ \\ 
$2$ & $(224, 224)$ & $11.61\%$ & $3.90\%$ \\ 
$3$ & $(233, 231)$ & $6.04\%$ & $2.22\%$ \\ 
$4$ & $(246, 245)$ & $3.55\%$ & $1.61\%$ \\
$5$ & $(258, 259)$ &  $2.29\%$ & $1.36\%$\\
$6$ & $(269, 261)$ &  $1.92\%$ & $1.34\%$\\

\end{tabular}
\caption{History of online enrichment in Example \ref{exp:1} with $\theta = \gamma=0.7$ and $\omega_i^+$ at time level $t=T=1$. }
\label{exp:5_1_07_w}
\end{table}

\begin{table}[htbp!]
\centering
\begin{tabular}{c|c||c|c}
$k$ & $(u_{\text{dof}}^{k}, p_{\text{dof}}^{k})$ & $e_u$ & $e_p$ \\ 
\hline
$0$ & $(200, 200)$ & $31.91\%$ & $12.45\%$ \\ 
$1$ & $(232, 230)$ & $6.73\%$ & $1.96\%$ \\ 
$2$ & $(266, 269)$ & $2.12\%$ & $1.27\%$ \\ 
$3$ & $(296, 302)$ & $1.47\%$ & $1.21\%$ \\ 
$4$ & $(327, 329)$ & $1.42\%$ & $1.21\%$ \\ 
$5$ & $(358, 359)$ & $1.41\%$ & $1.21\%$ \\
\end{tabular}
\caption{History of online enrichment in Example \ref{exp:1} with $\theta = \gamma=0.3$ and $\omega_i^+$ at time level $t=T=1$. }
\label{exp:5_1_03_w}
\end{table}

\begin{figure}[htbp!]
\centering
\includegraphics[width=2.5in]{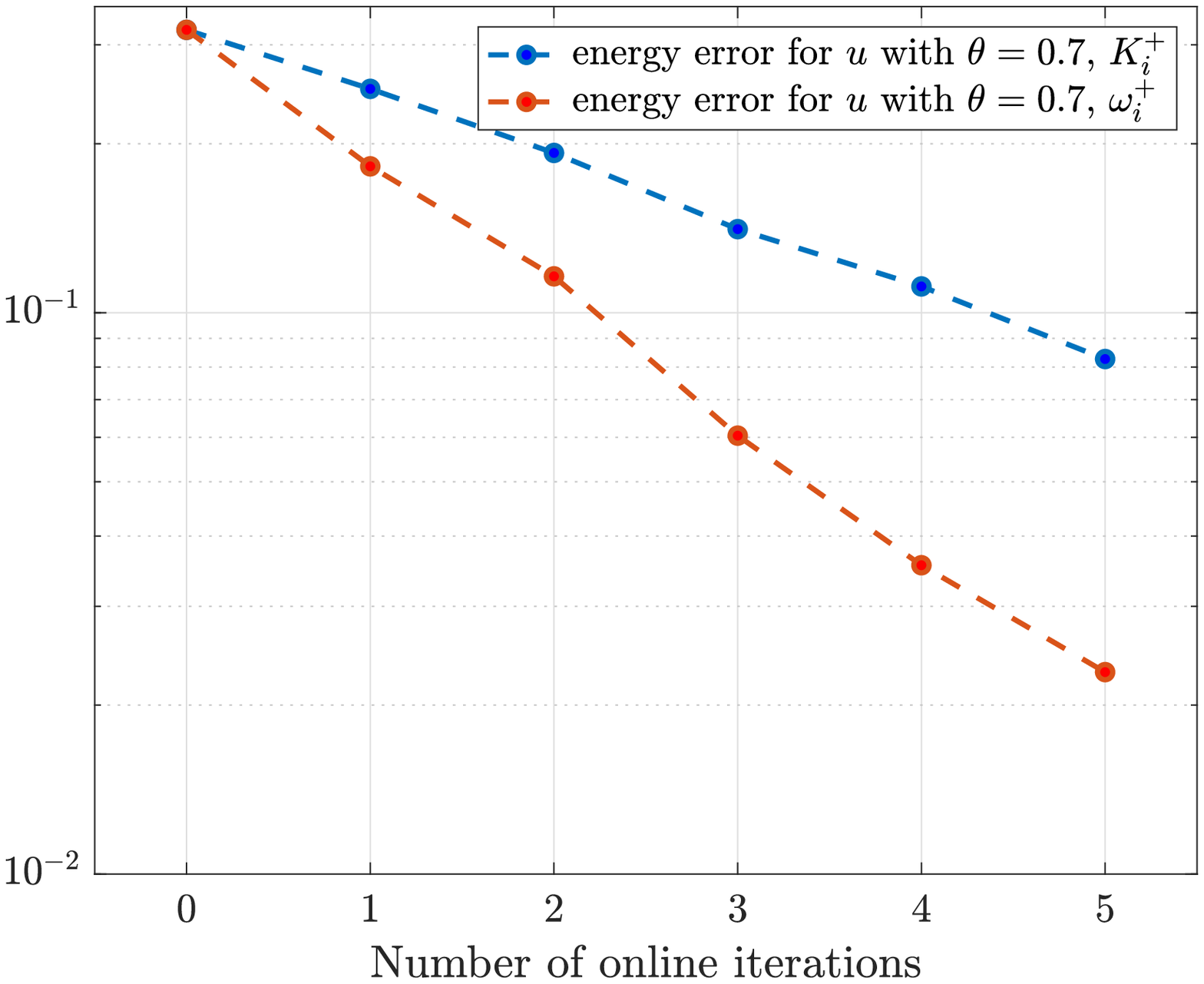} \quad
\includegraphics[width=2.5in]{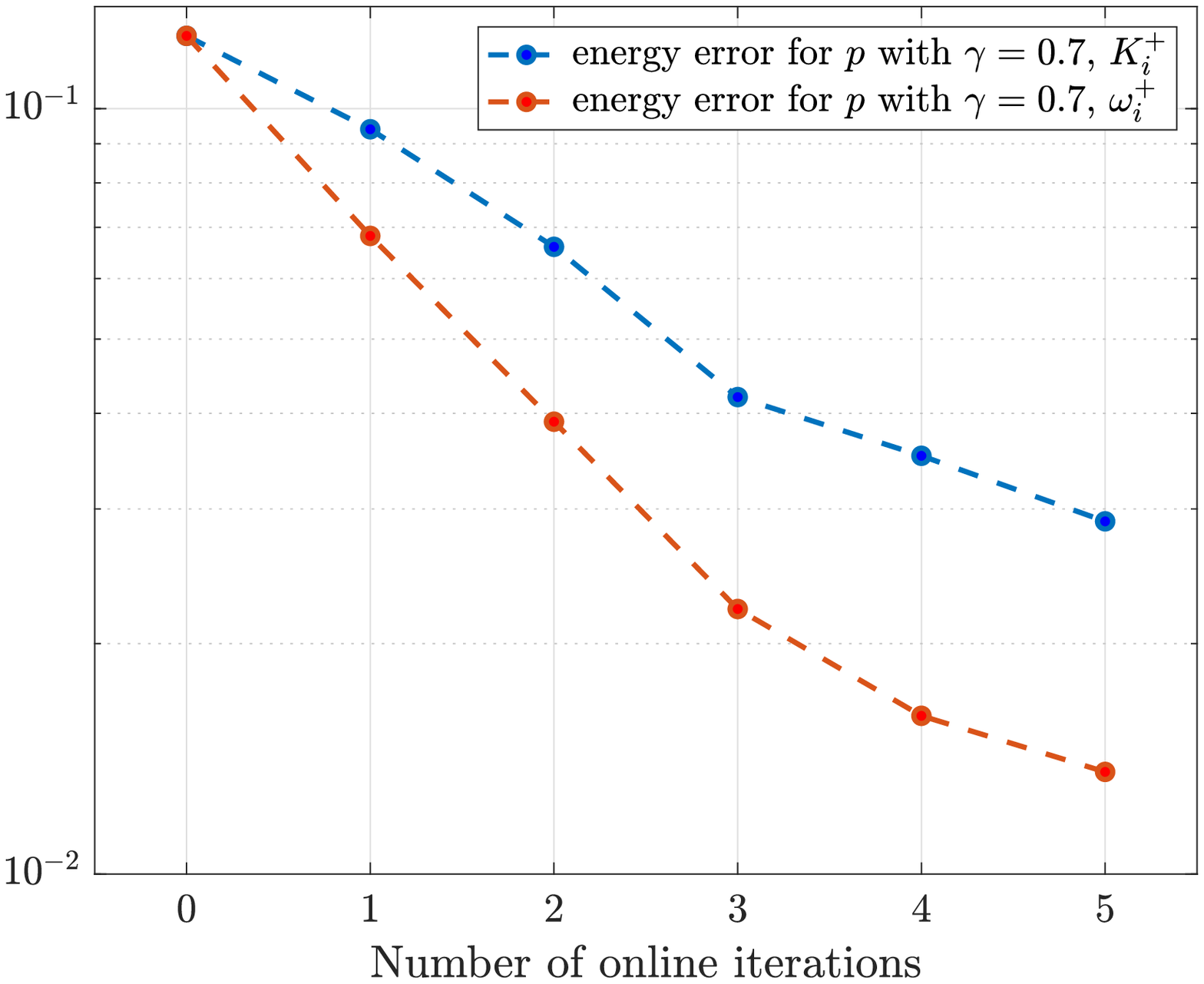}
\caption{Energy errors against the number of online iterations with $\theta=\gamma=0.7$ in Example \ref{exp:1} at time level $t=T=1$.}
\label{fig:5_1_07}
\end{figure}

\begin{figure}[htbp!]
\centering
\includegraphics[width=2.5in]{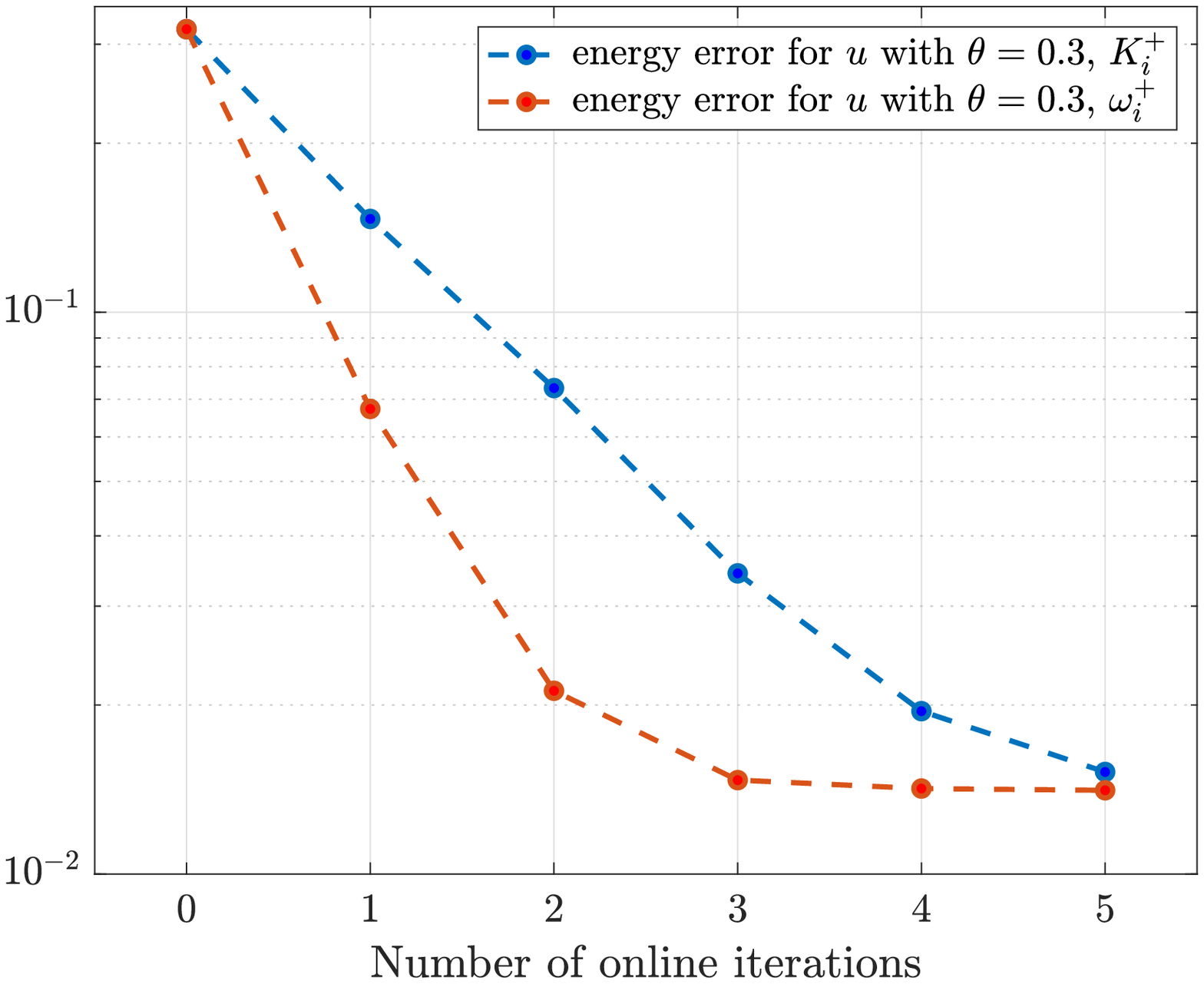} \quad
\includegraphics[width=2.5in]{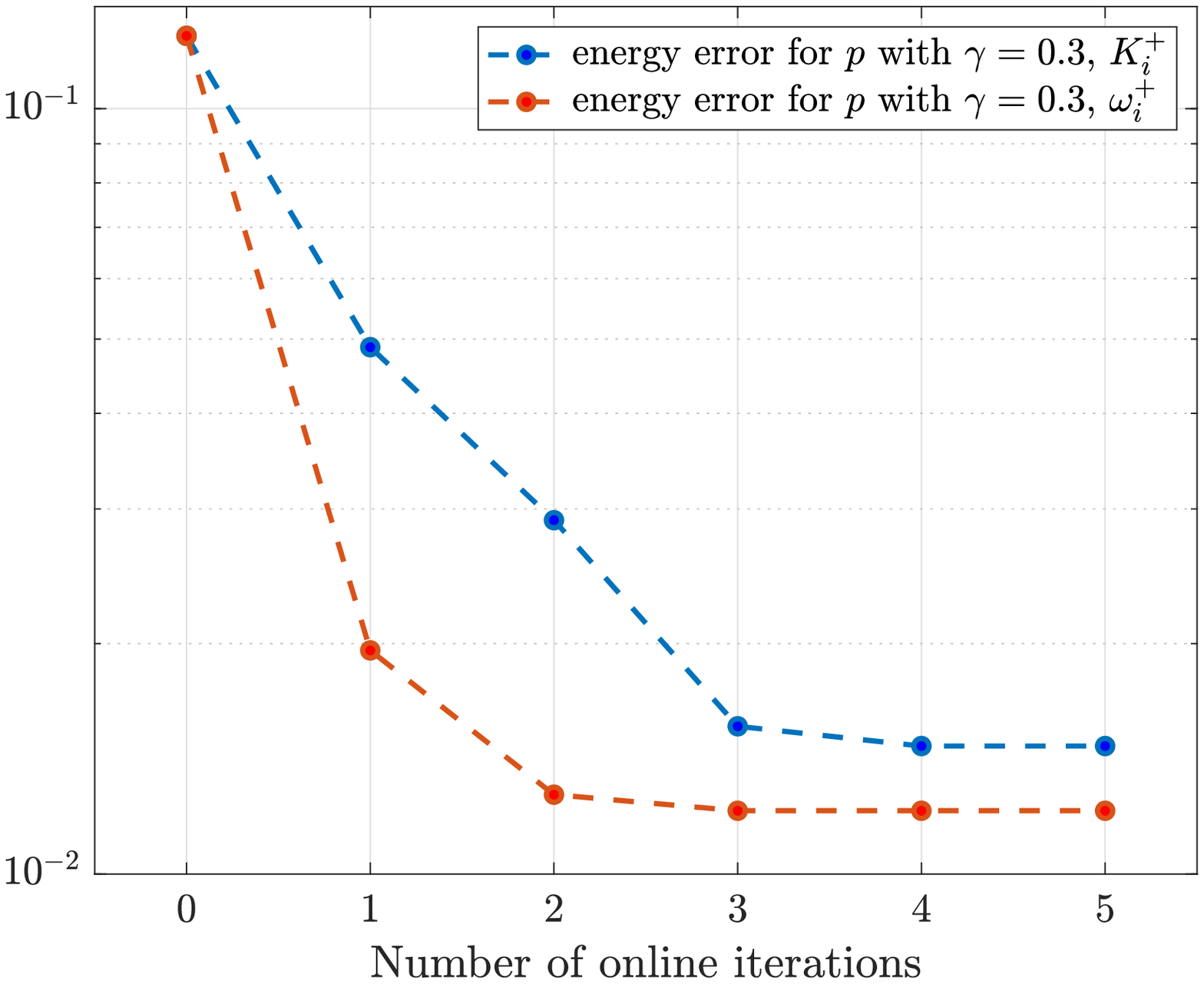}
\caption{Energy errors against the number of online iterations with $\theta=\gamma=0.3$ in Example \ref{exp:1} at time level $t=T=1$.}
\label{fig:5_1_03}
\end{figure}

\begin{example}\label{exp:3}
\revi{We test the online adaptive method on a case with larger Poisson ratio close to $0.5$. 
In particular, we set $\nu_p = 0.49$, $\theta=\gamma=30\%$, and the rest of the setting is the same as that of Example \ref{exp:1}. 
We use the neighborhood-based online strategy for the enrichment of online basis functions. }
\end{example}

\begin{figure}[htbp!]
\centering
\includegraphics[width = 1.7in]{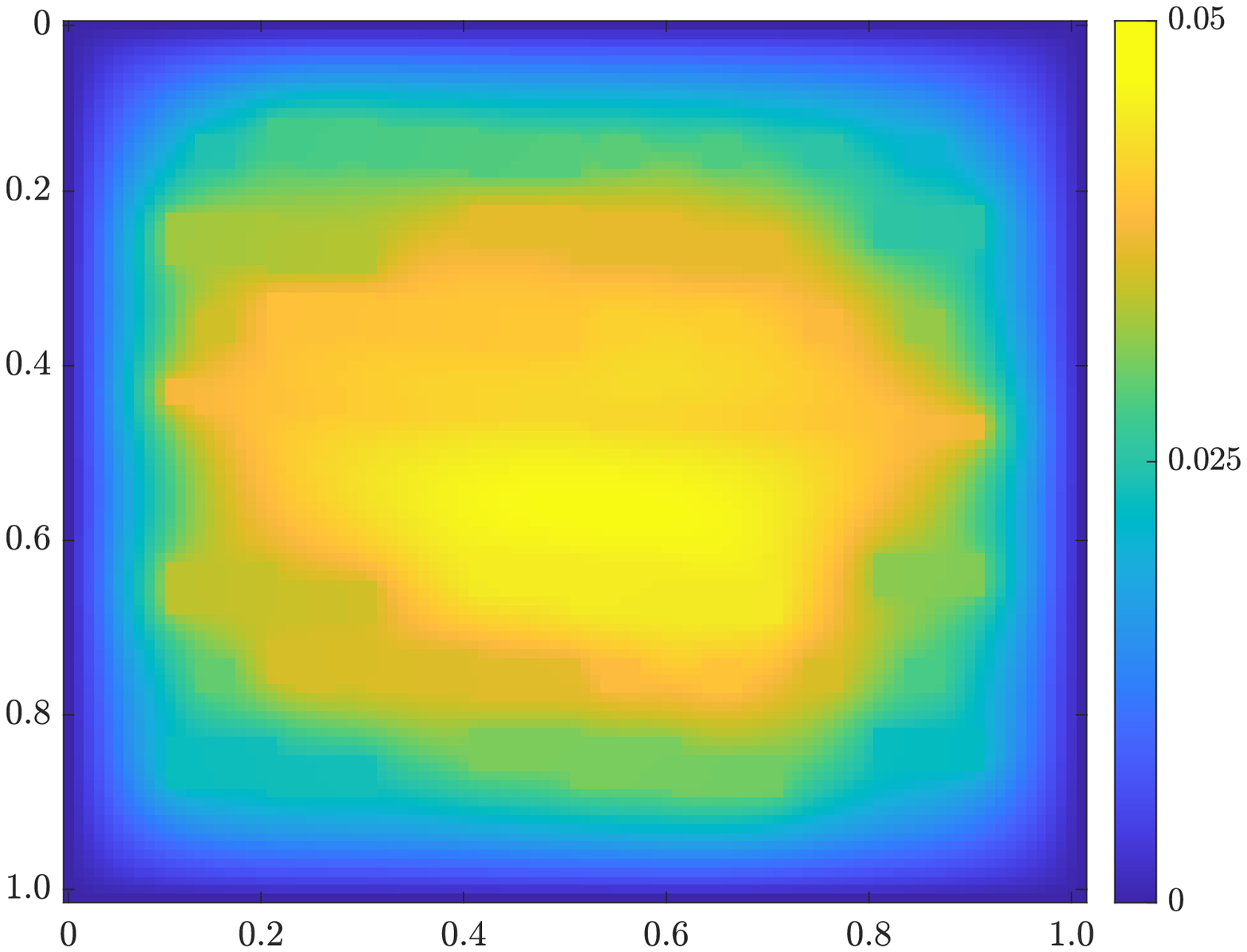} \quad
\includegraphics[width = 1.7in]{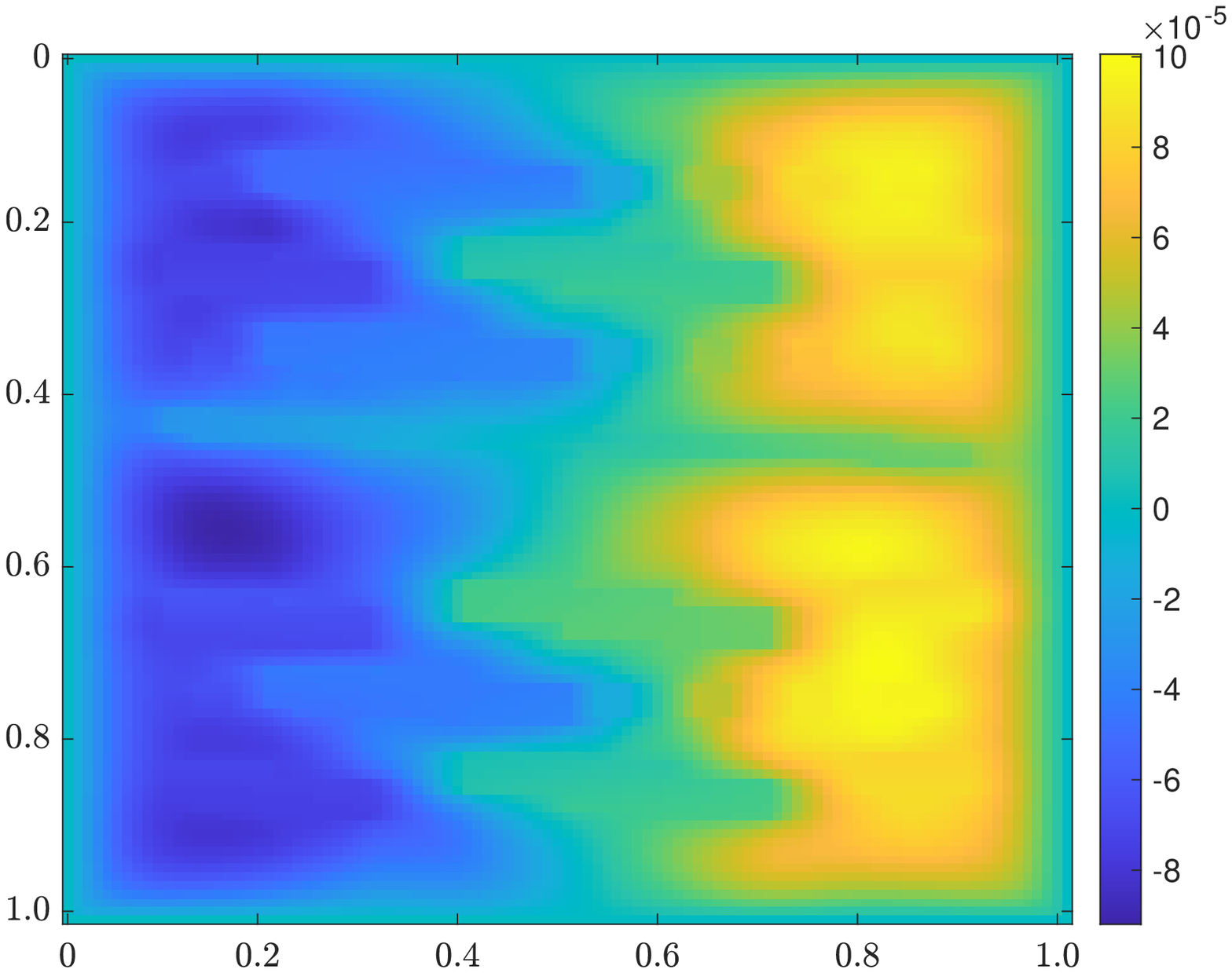}\quad
\includegraphics[width = 1.7in]{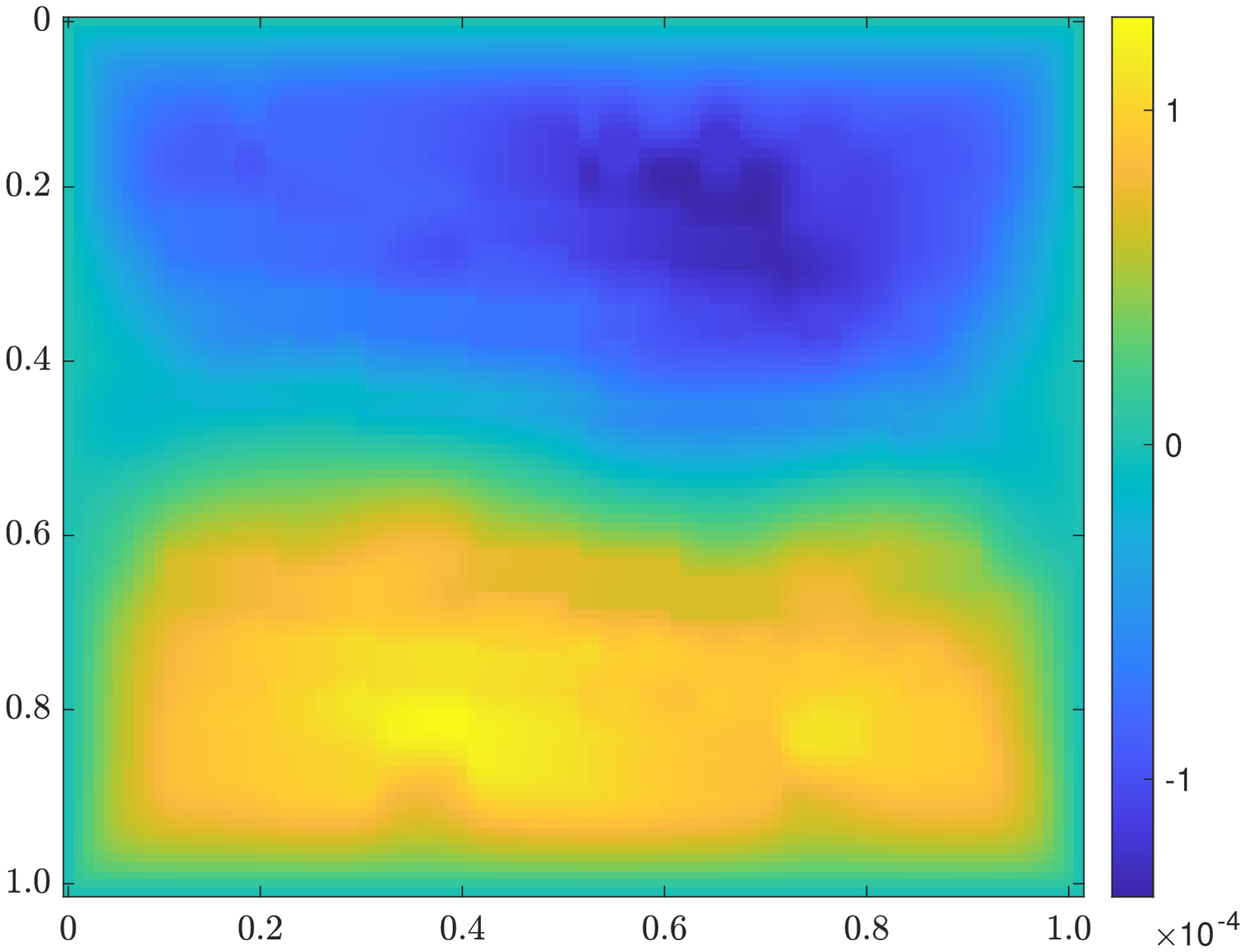}\\
\includegraphics[width = 1.7in]{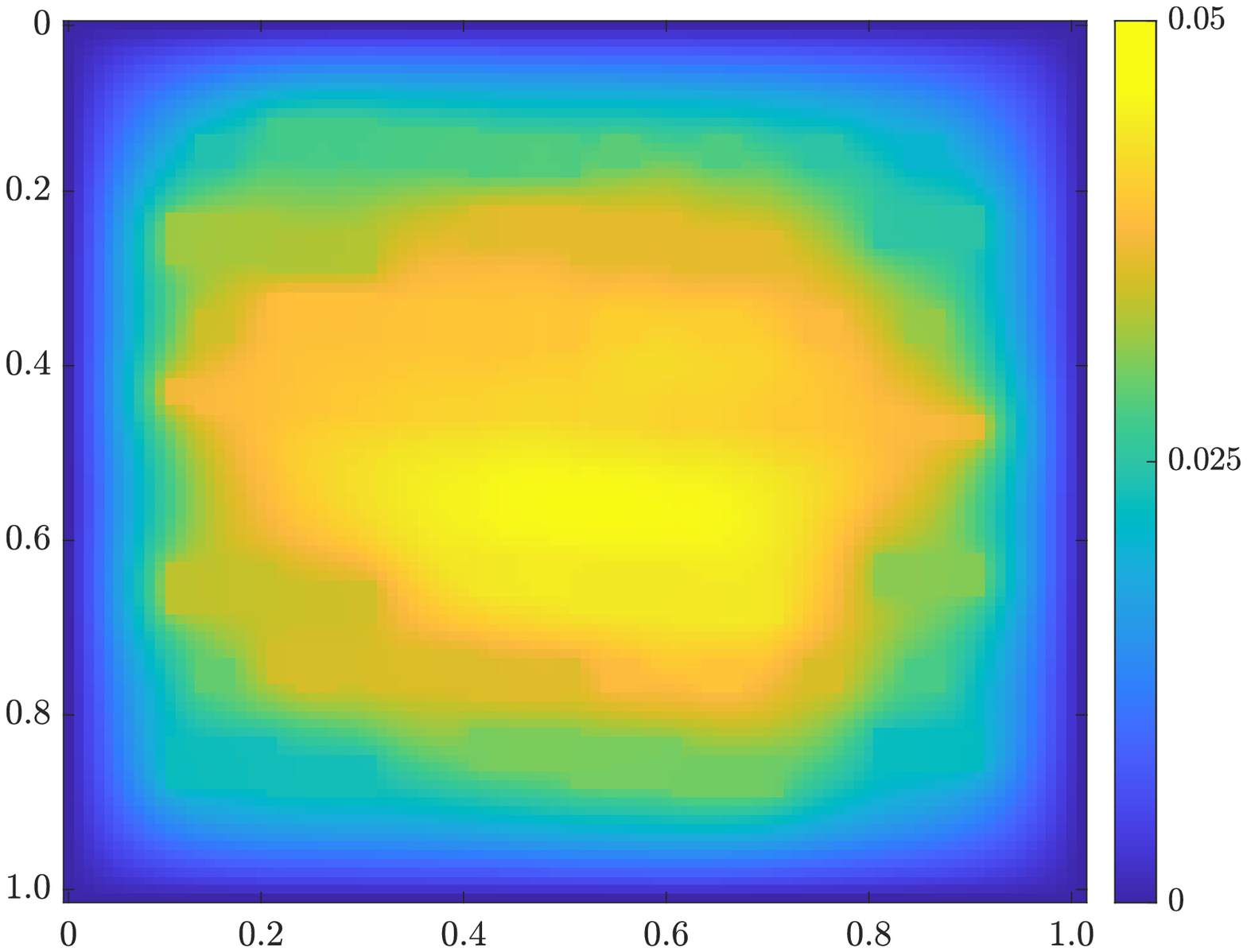}\quad
\includegraphics[width = 1.7in]{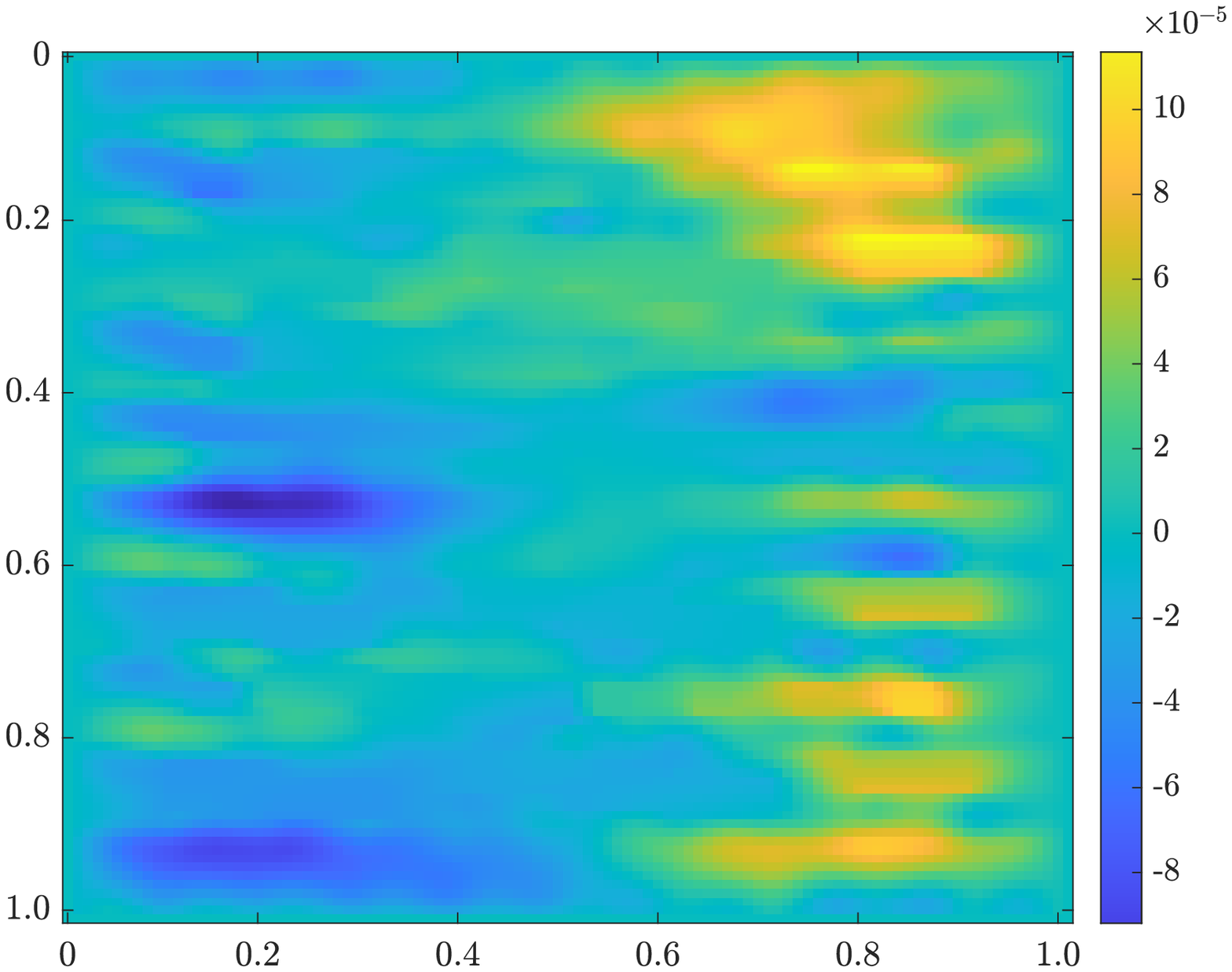}\quad
\includegraphics[width = 1.7in]{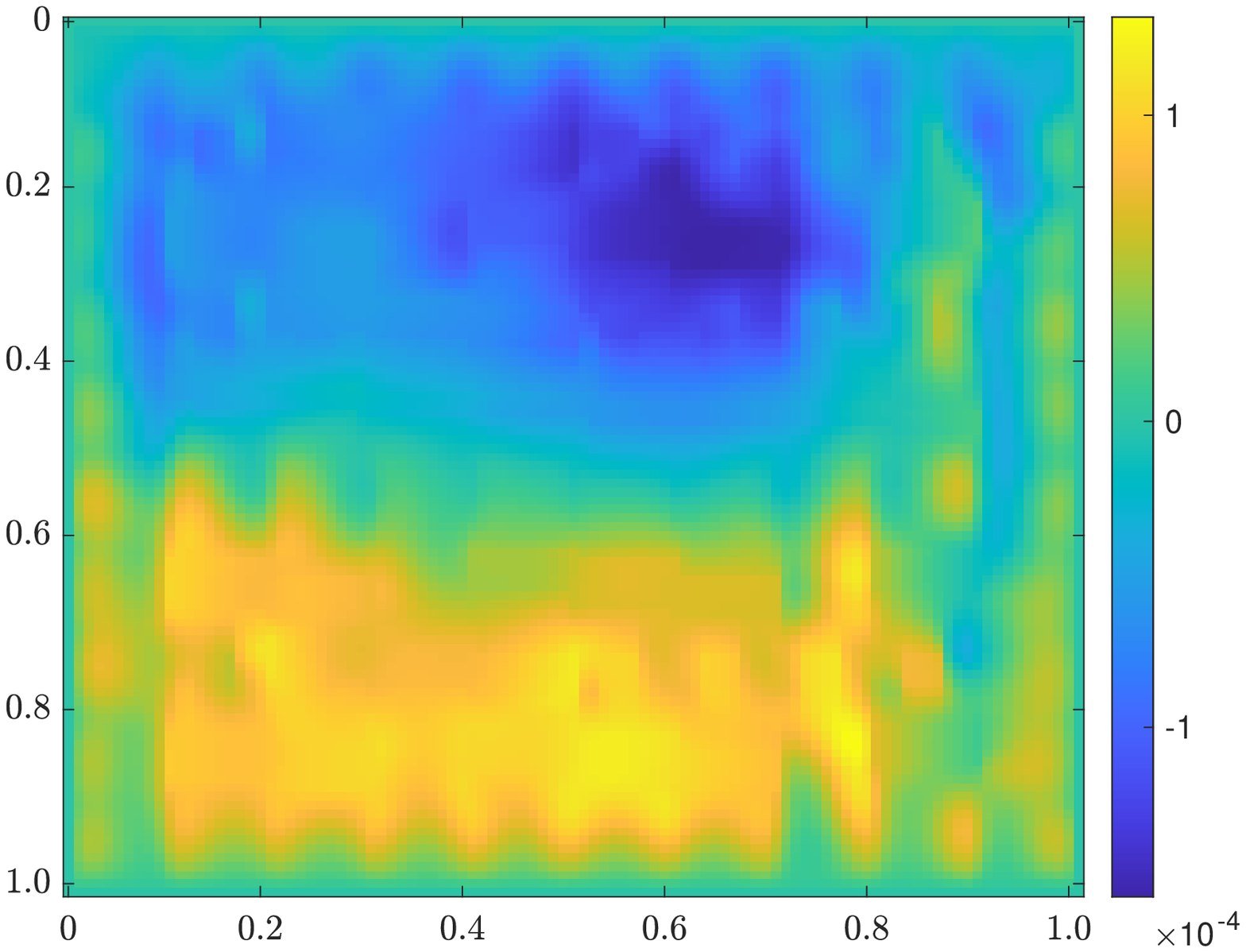}\\
\includegraphics[width = 1.7in]{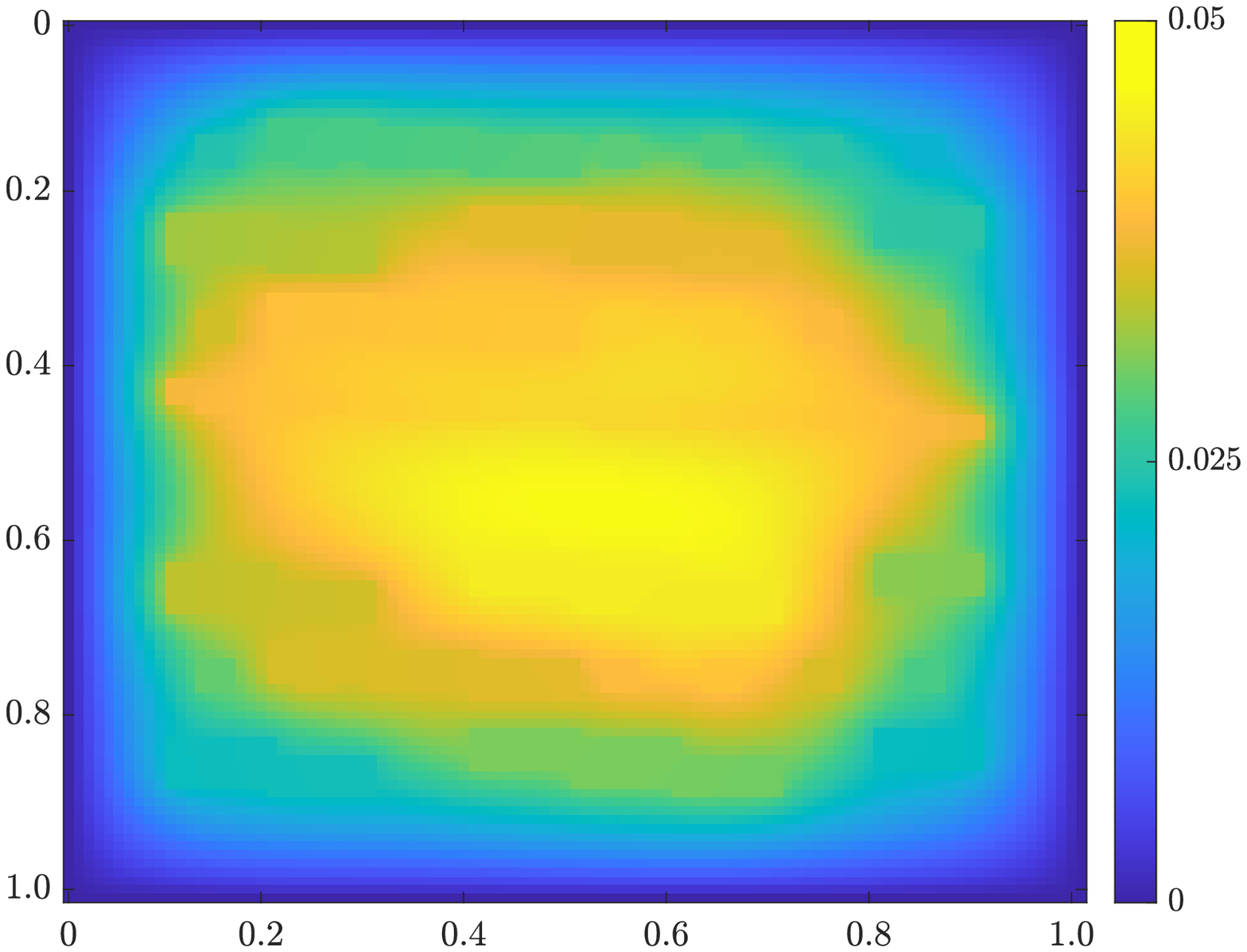}\quad
\includegraphics[width = 1.7in]{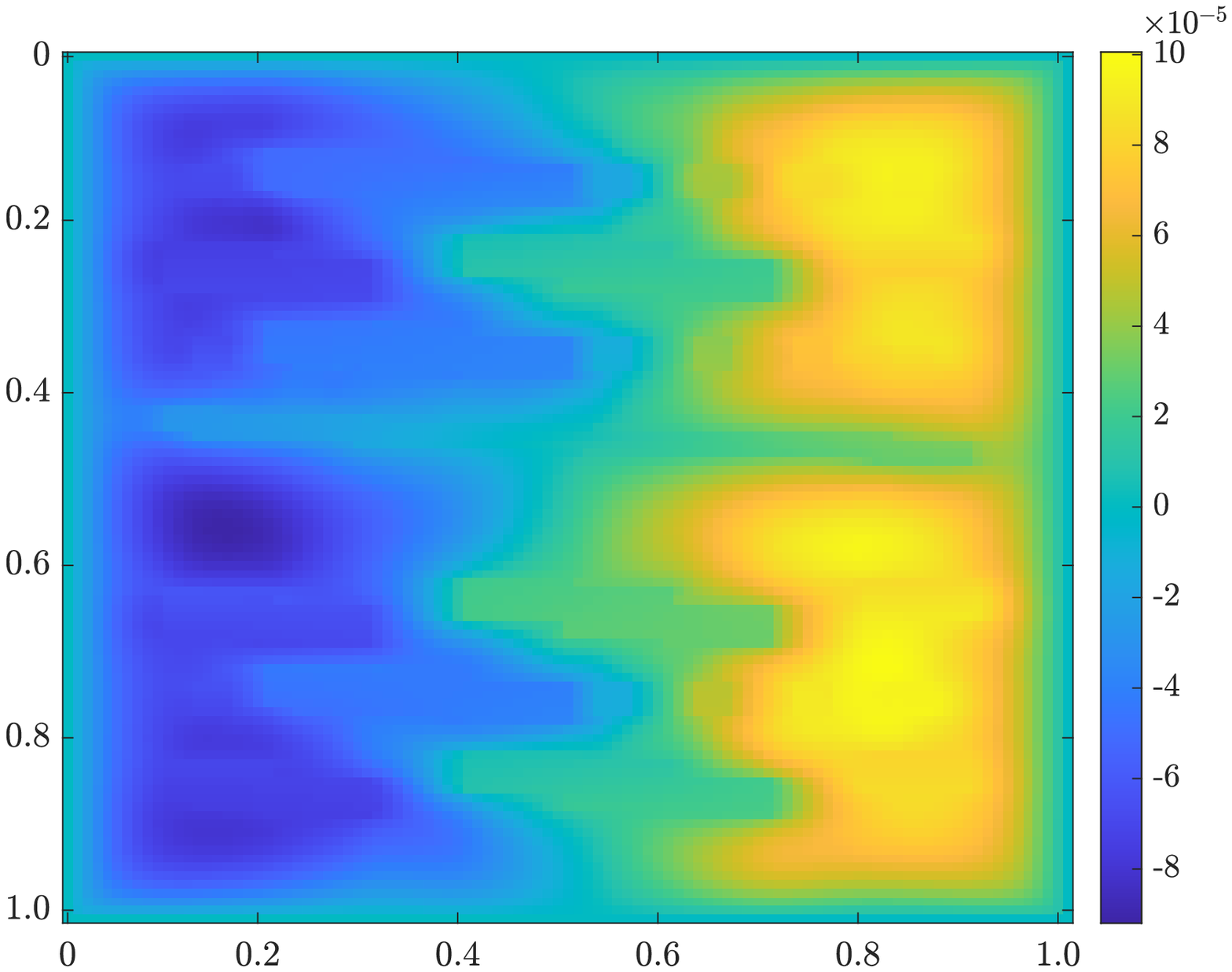}\quad
\includegraphics[width = 1.7in]{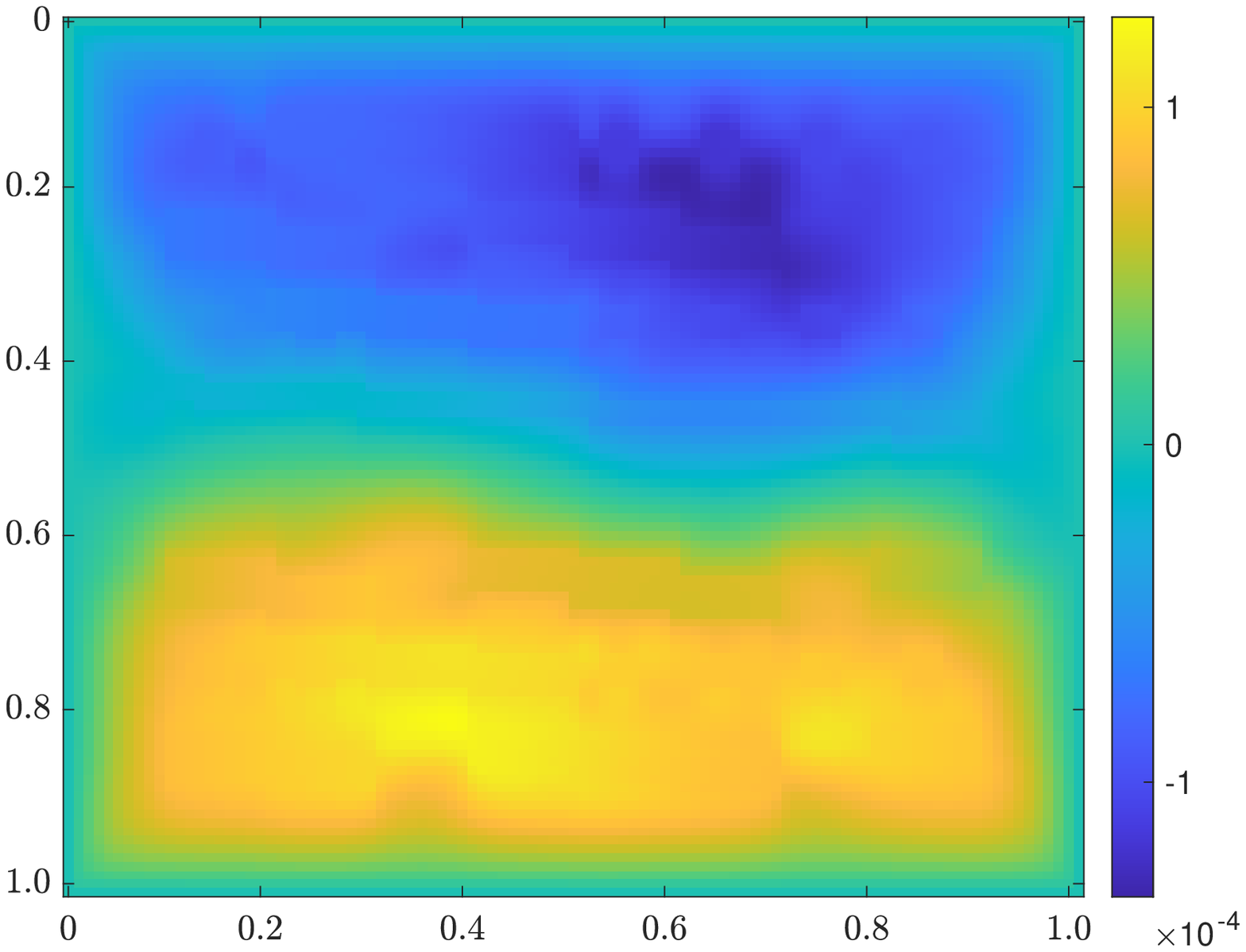}
\caption{Solution profiles (starting from left to right) $p$, $u_1$, $u_2$ of Example \ref{exp:3}. Top: reference solutions; middle: offline solutions; bottom: online solutions with iteration 2 at  time level $t=T=1$.  }
\end{figure}

\begin{table}[htbp!]
\begin{center}
\begin{tabular}{c|c||c|c}
$k$ & $(u_{\text{dof}}^k ,p_{\text{dof}}^k )$& $e_u$ & $e_p$ \\ 
\hline
$0$ & $(200, 200)$ & $57.38\%$ & $12.45\%$ \\ 
$1$ & $(238, 230)$ & $12.72\%$ & $1.75\%$ \\ 
$2$ & $(275, 269)$ & $4.39\%$ & $0.88\%$ \\ 
$3$ & $(316, 301)$ & $1.74\%$ & $0.79\%$ \\ 
$4$ & $(354, 328)$ & $1.04\%$ & $0.78\%$ \\ 
$5$ & $(394, 359)$ & $0.91\%$ & $0.78\%$ \\ 
\end{tabular}
\caption{History of online enrichment in Example \ref{exp:3} at  time level $t=T=1$. }
\end{center}
\end{table}

\begin{figure}[htbp!]
\centering
\includegraphics[width = 2.5in]{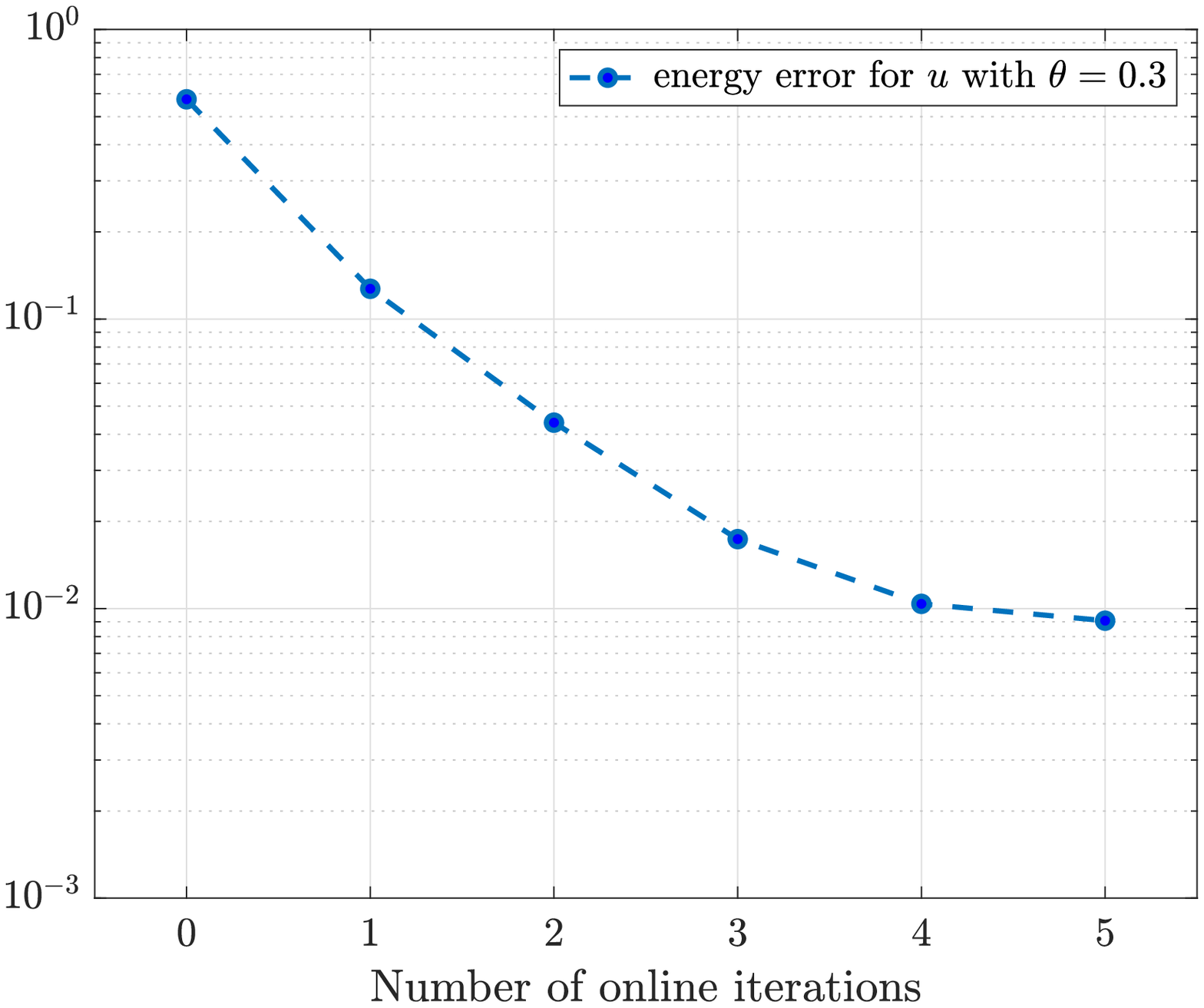}\quad
\includegraphics[width = 2.5in]{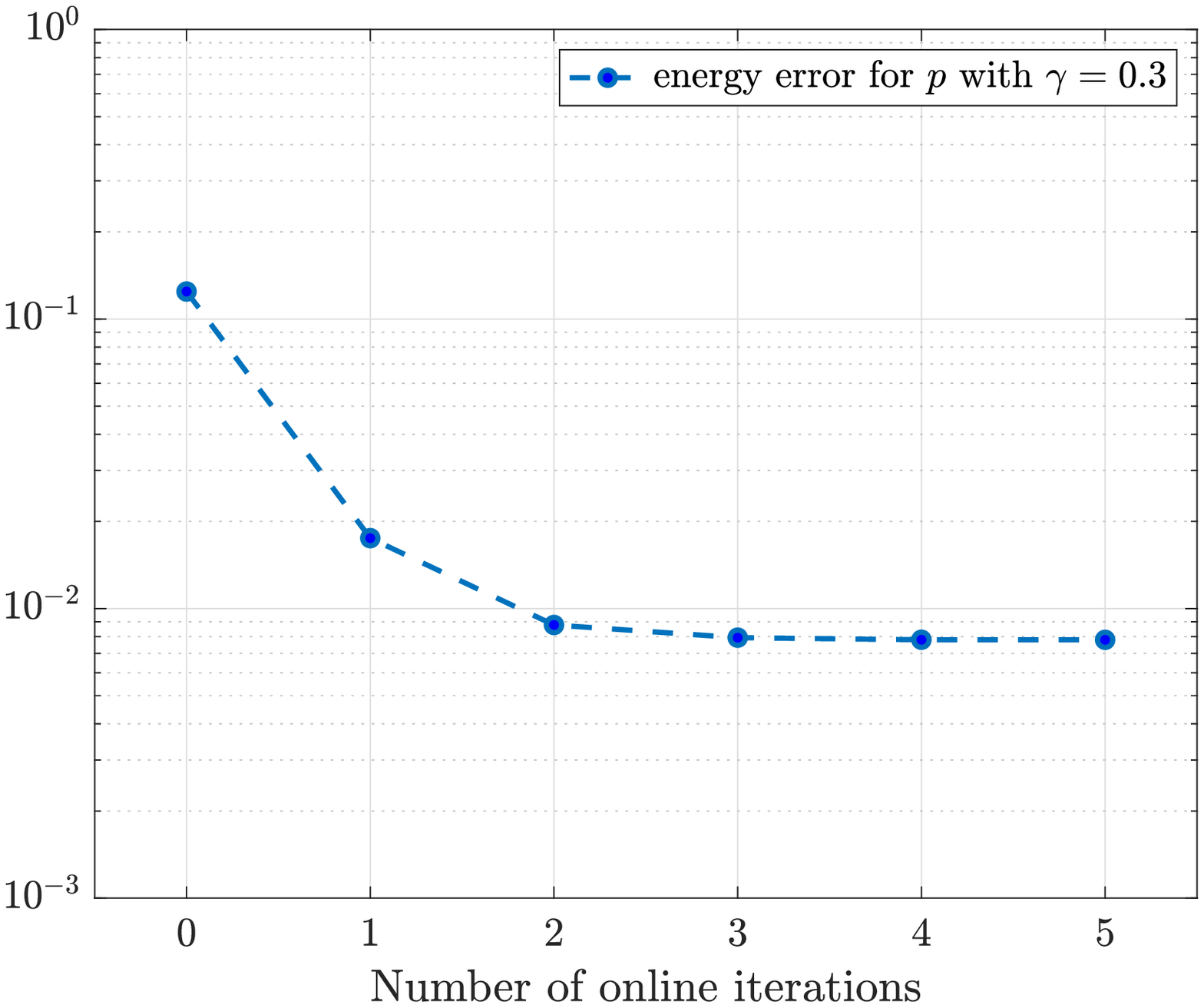}
\caption{Energy errors against the number of online iterations in Example \ref{exp:3} at time level $t=T=1$.}
\end{figure}

\noindent
\revi{We notice that if we keep using the offline degrees of freedom $(u_{\text{dof}}^{k}, p_{\text{dof}}^{k})=(200,200)$ during the whole time propagation for the case with $\nu_p = 0.49$, the energy errors of the offline solutions are significantly large, especially for the error of displacement. 
The result shows that the online enrichment algorithm is necessary to improve the accuracy of multiscale approximation. 
With only five online enrichments, the energy errors drops drastically to below $1\%$ while degrees of freedom $(u_{\text{dof}}^{k}, p_{\text{dof}}^{k})=(394,359)$ are used. Compared to the online enrichment method, one can also achieve the level of accuracy $e_u=5.30\%$ and $e_p=6.83\%$ when the offline degrees of freedom are $(u_{\text{dof}}^{k}, p_{\text{dof}}^{k})=(2000,2000)$. 
These results show that the online enrichment algorithm can greatly improve the accuracy of multiscale approximation. }

In the next two examples, we perform online enrichment every several time steps and they yields high-fidelity simulations.

\begin{example}\label{exp:2}
The problem setting in this example reads as follows: 
\begin{enumerate}
\item We set $\Omega = (0,1)^2$, $\alpha = 0.9$, $M = 1$, $\nu_p = 0.2$, and $\nu = 1$. 
\item The terminal time is $T = 1$ and the time step is $\tau = 0.02$ (with $N = 20$). 
\item The Young's modulus $E$ is depicted in Figure \ref{fig:ym_exp1} (right). We set the permeability to be $\kappa = E$. 
\item We set $f (x_1,x_2) =2\pi^2 \sin(\pi x_1) \sin(\pi x_2)$; the initial condition is $p(x_1, x_2) = 100x_1^2(1-x_1)x_2^2(1-x_2)$. 
\item We set $J = J_i^1 = J_i^2 = 2$; $\ell=2$ for the offline stage and $\ell = 3$ for the online stage. 
\item Every five steps in time we perform an online (neighborhood-based) enrichment with online tolerances to be $\theta =\gamma= 30\%$ and $\theta = \gamma  = 70\%$. 
\item The coarse mesh is $20 \times 20$ and the overall fine mesh is $200 \times 200$. 
\end{enumerate}
\end{example}

\begin{figure}[htbp!]
\centering
\includegraphics[width = 1.7in]{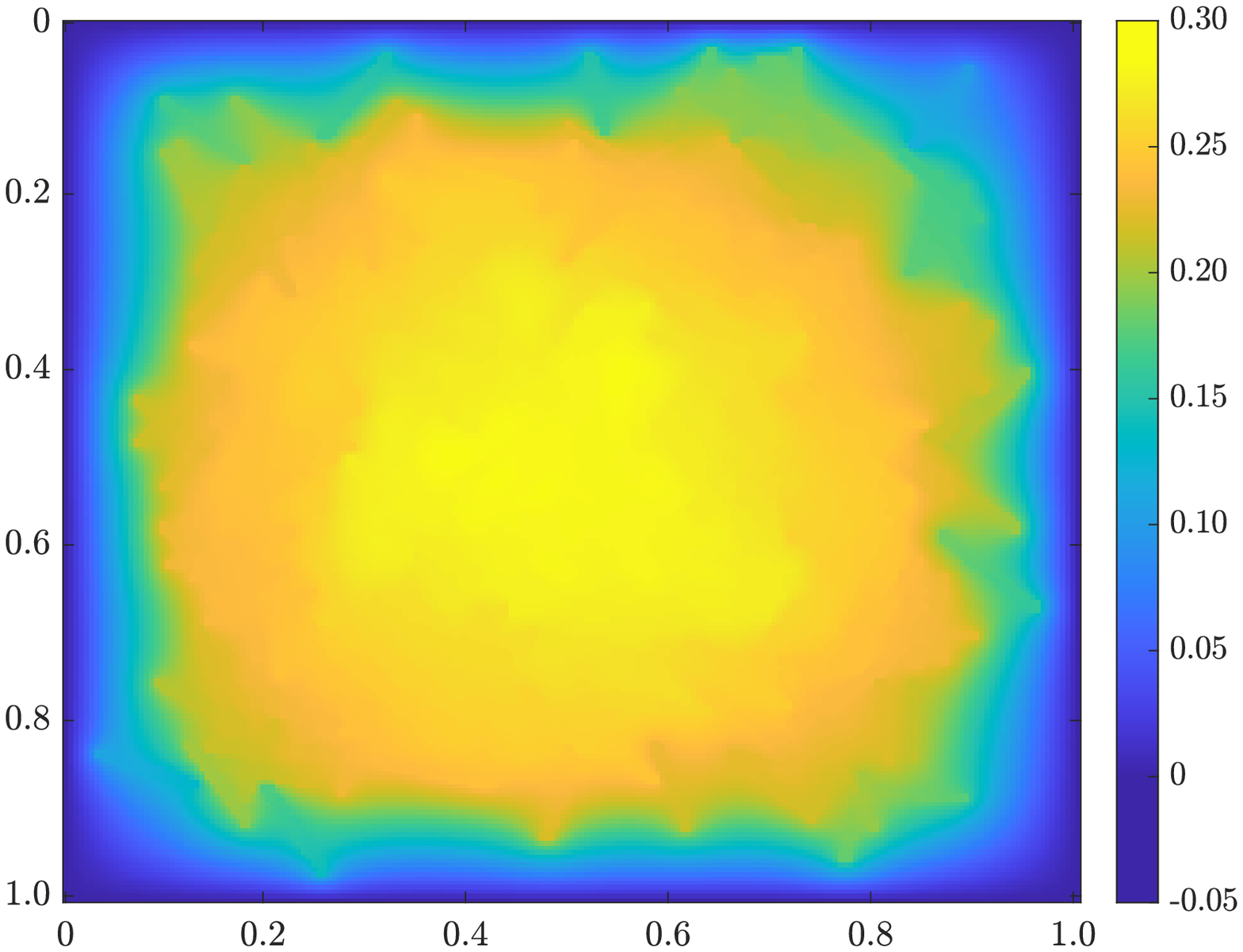} \quad
\includegraphics[width = 1.7in]{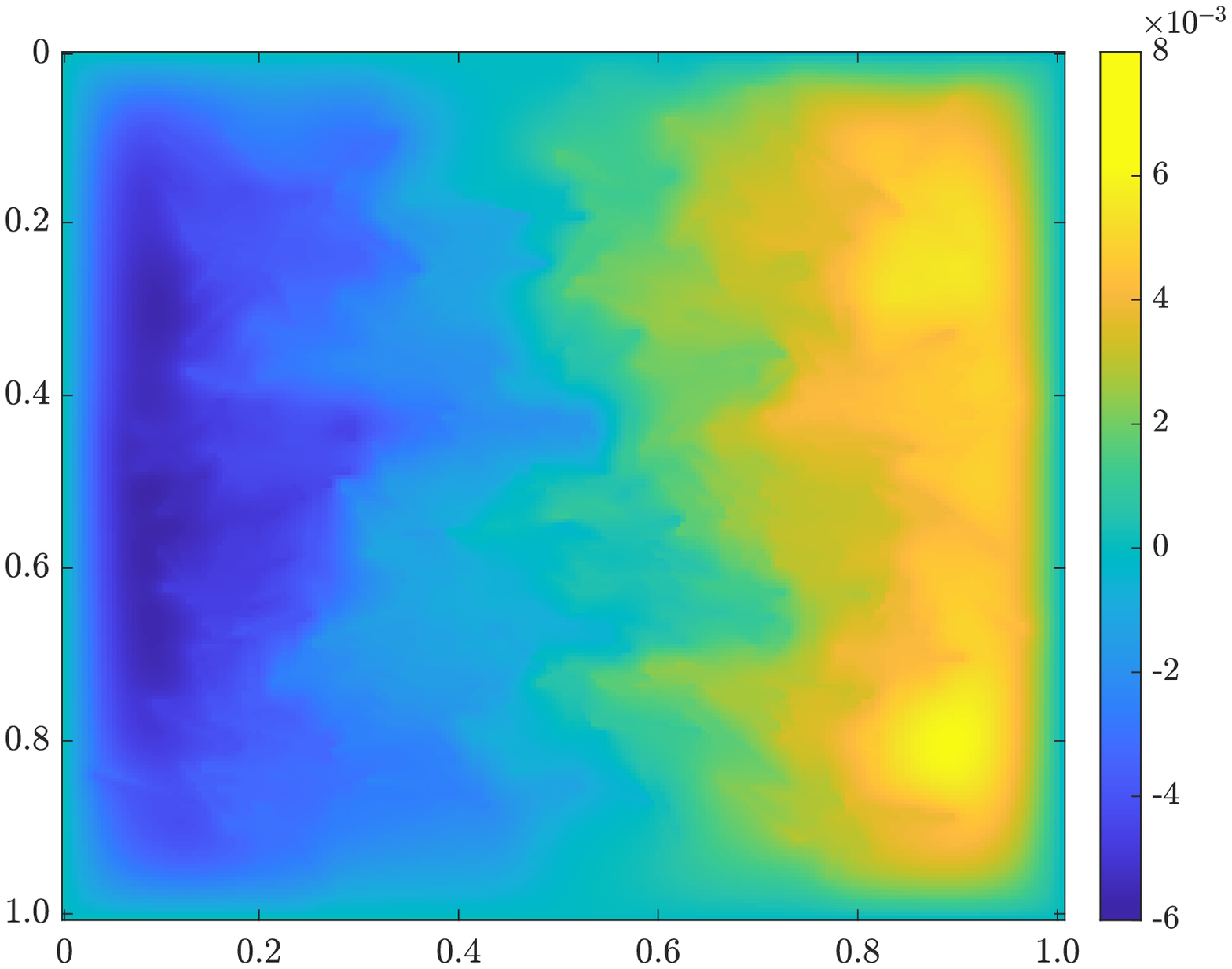}\quad
\includegraphics[width = 1.7in]{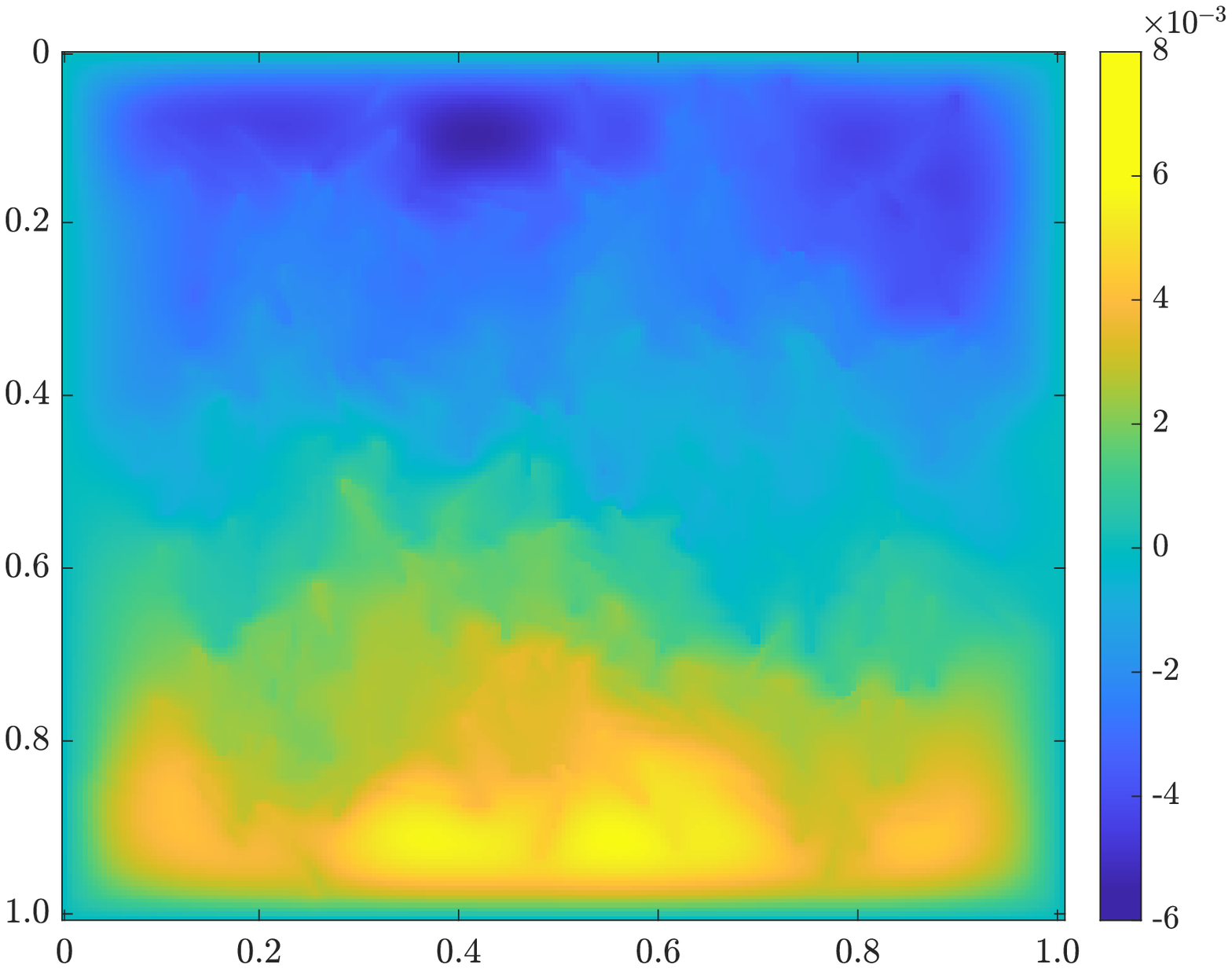}
\includegraphics[width = 1.7in]{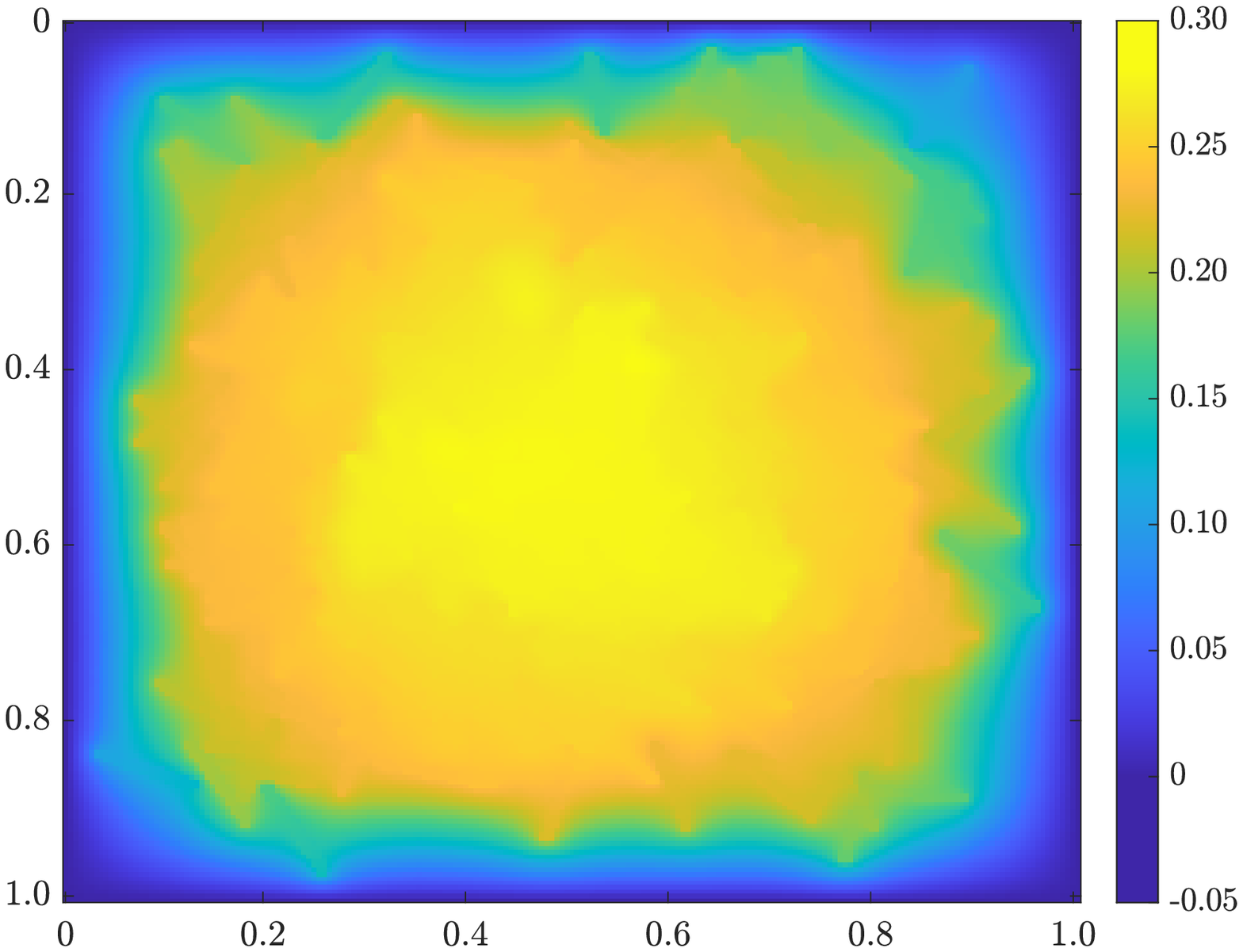}\quad
\includegraphics[width = 1.7in]{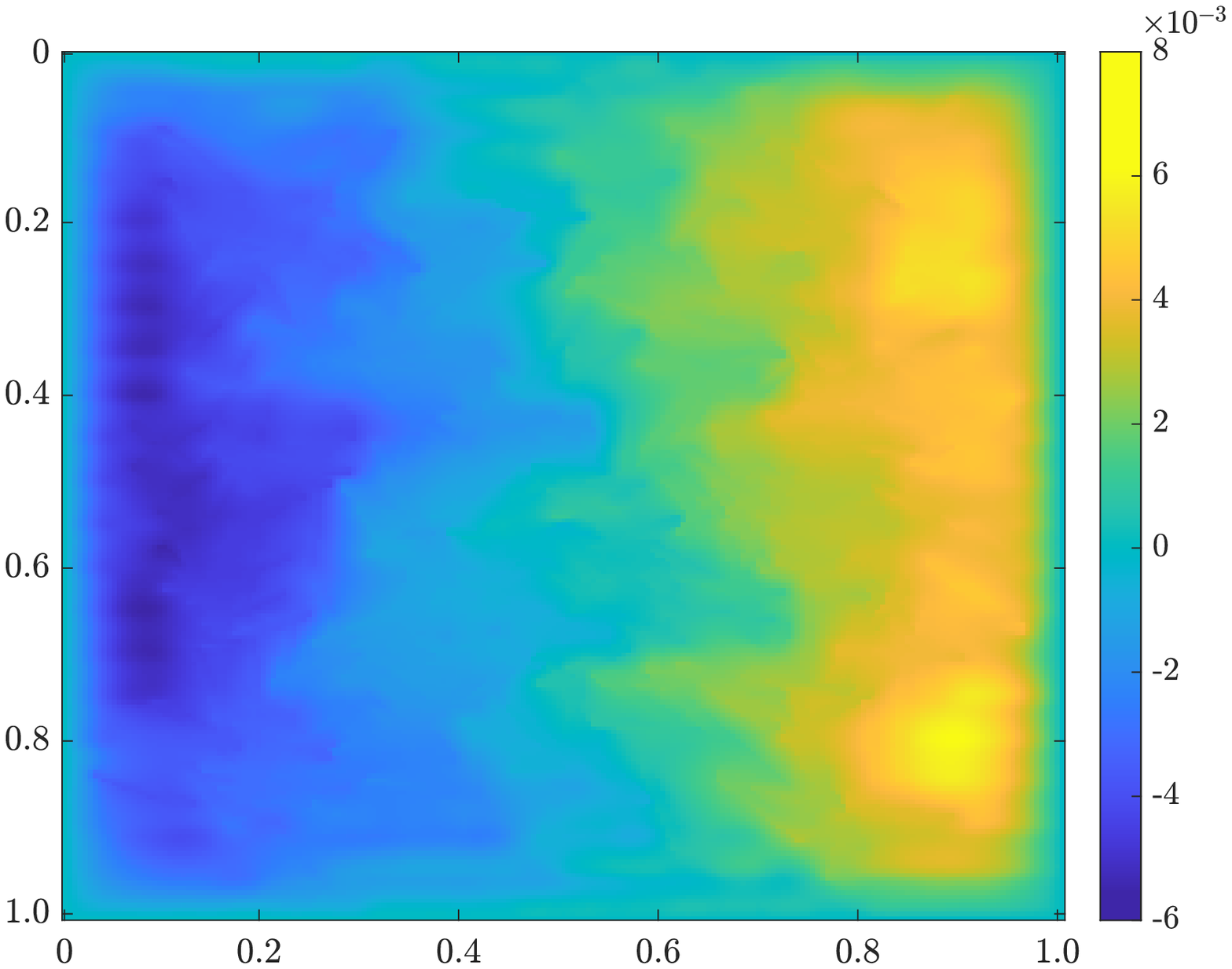}\quad
\includegraphics[width = 1.7in]{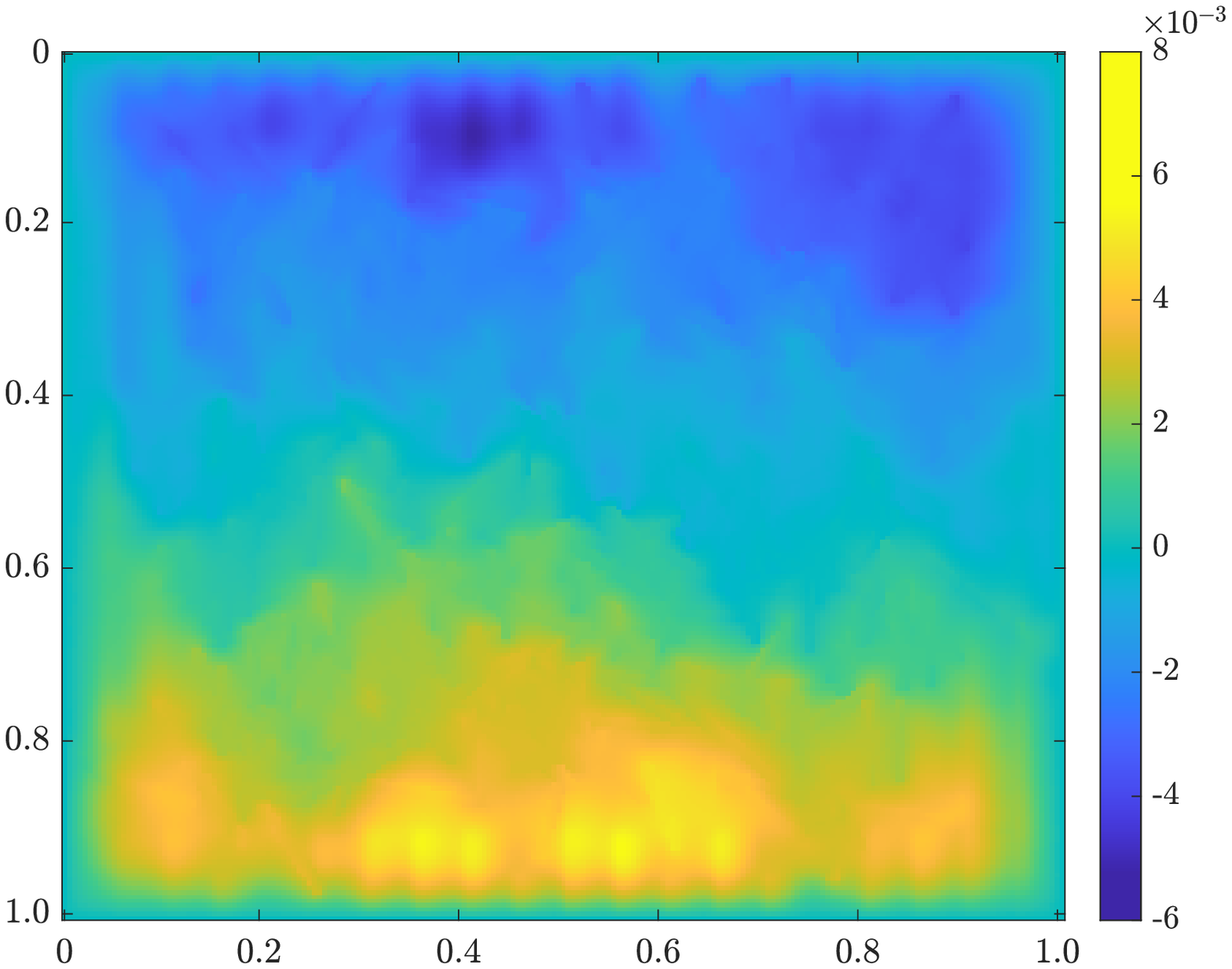}
\includegraphics[width = 1.7in]{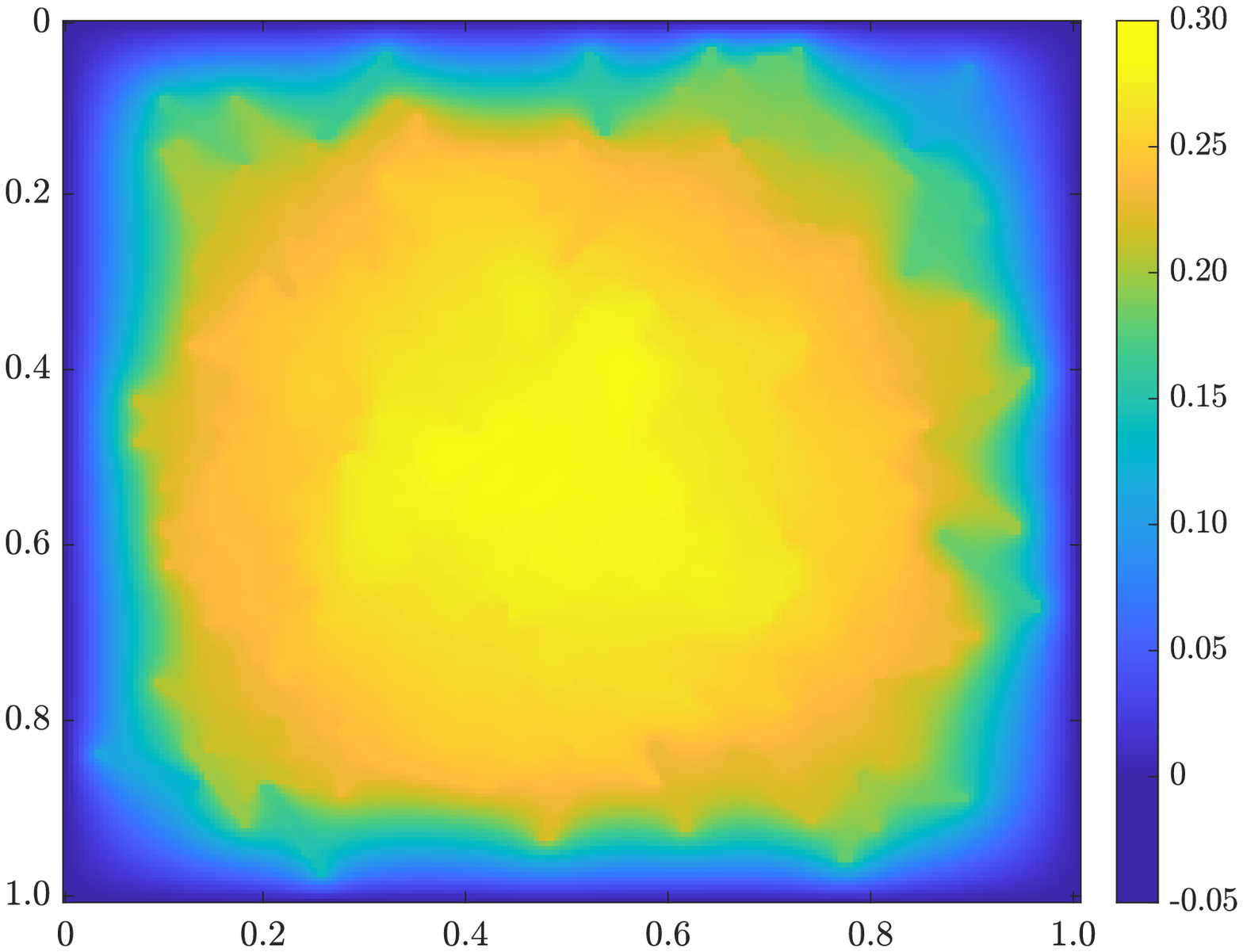}\quad
\includegraphics[width = 1.7in]{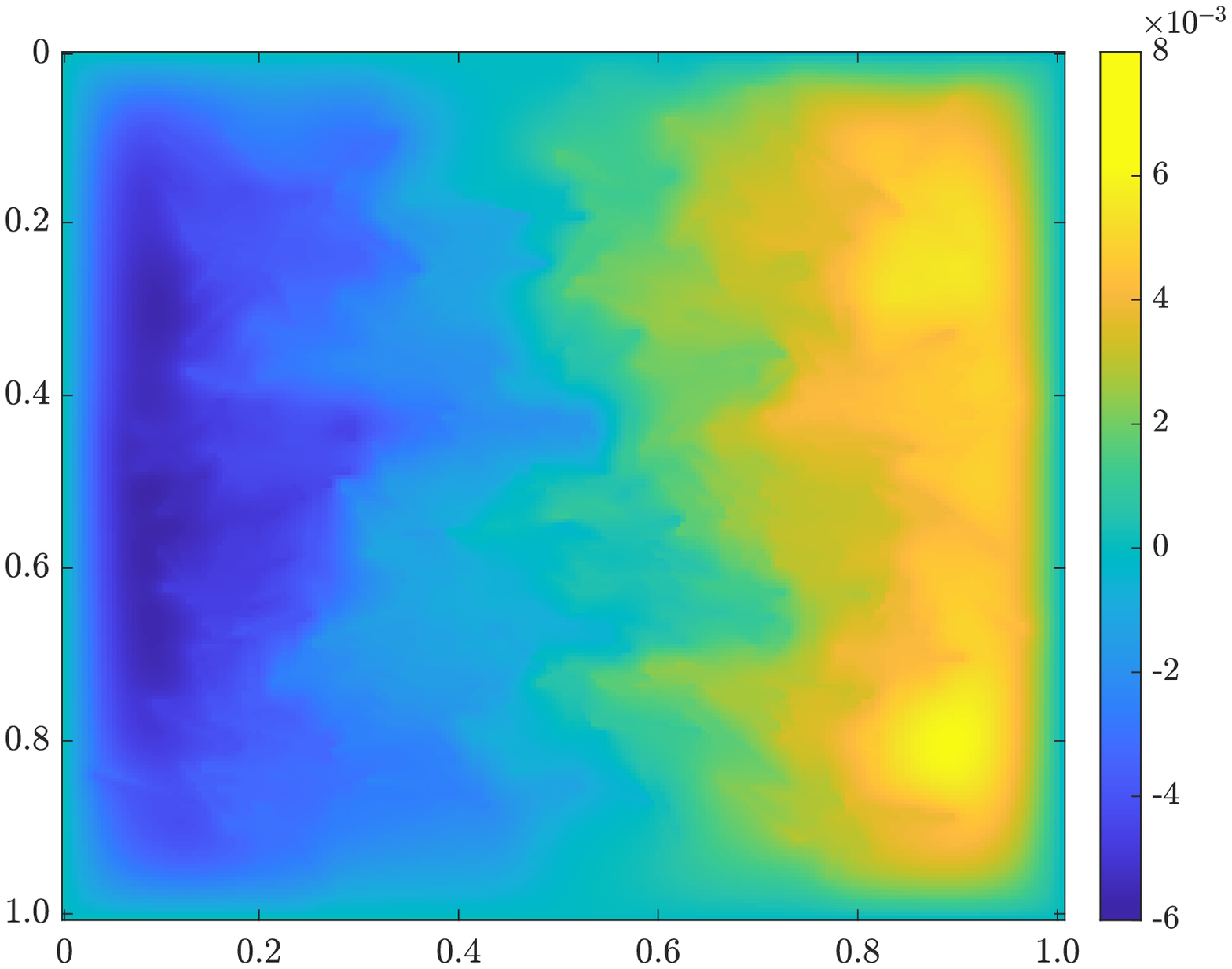}\quad
\includegraphics[width = 1.7in]{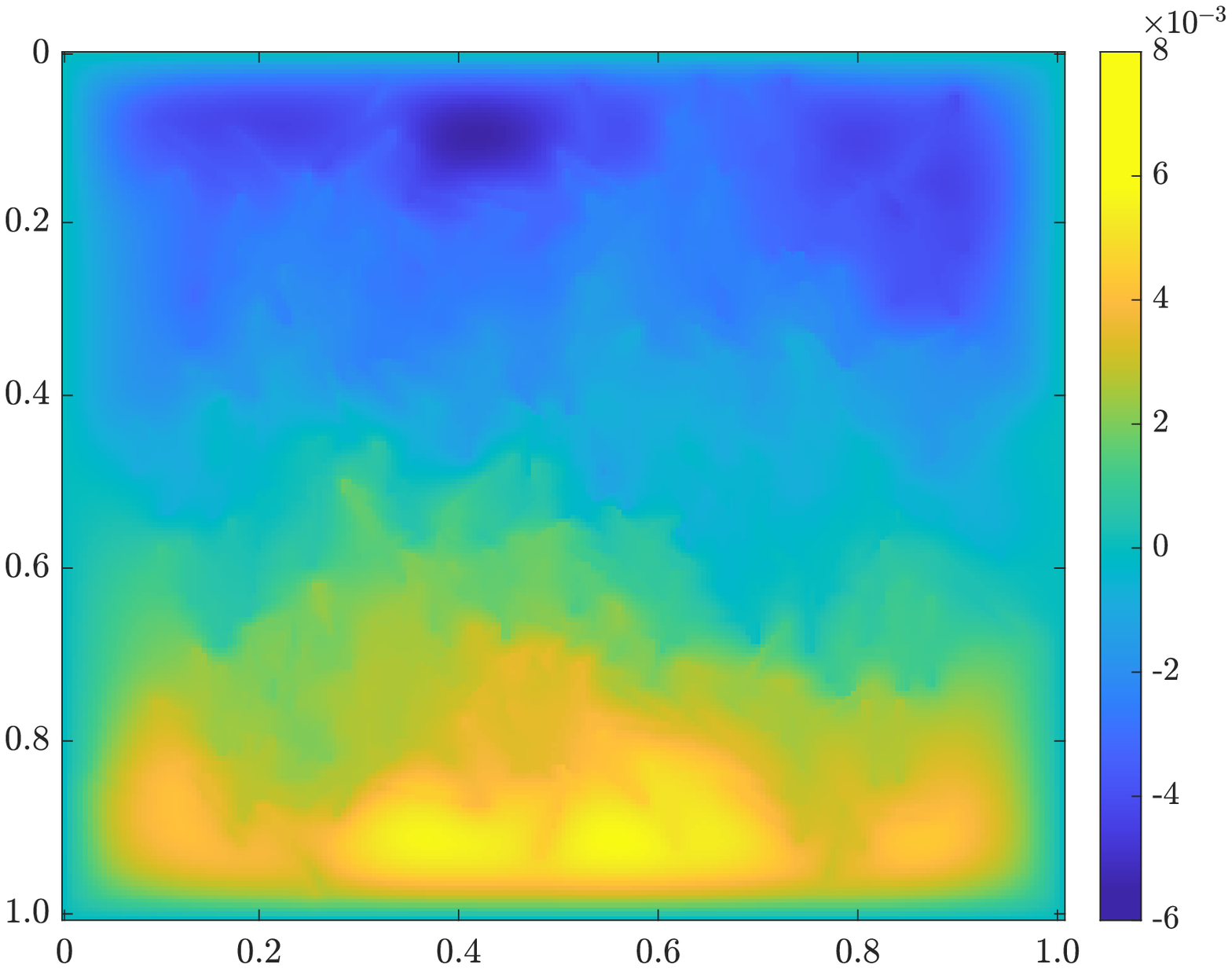}
\caption{Solution profiles for (starting from left to right) $p$, $u_1$, $u_2$ at $t=0.52$ of Example \ref{exp:2}. Top: reference solutions; middle: offline solutions; bottom: online solutions with iteration 3 with $\theta=\gamma=0.3$ at  time level $t=0.52$. }
\label{fig:5_2_soln_profile}
\end{figure}

\begin{table}[htbp!]
\centering
\begin{tabular}{c|c||c|c}
$k$ & $(u_{\text{dof}}^{k}, p_{\text{dof}}^{k})$& $e_u$ & $e_p$ \\ 
\hline
$0$ & $(800, 800)$ & $30.83\%$ & $8.95\%$ \\ 
$1$ & $(1115, 1144)$ & $1.46\%$ & $0.62\%$ \\ 
$2$ & $(1408, 1411)$ & $0.11\%$ & $0.25\%$ \\ 
$3$ & $(1664, 1511)$ & $0.03\%$ & $0.23\%$ \\ 
\end{tabular}
\caption{History of online enrichment in Example \ref{exp:2} with $\theta=\gamma=0.7$ and $\omega_i^+$ at  time level $t=0.52$. }
\end{table}

\begin{table}[htbp!]
\centering
\begin{tabular}{c|c||c|c}
$k$ & $(u_{\text{dof}}^{k}, p_{\text{dof}}^{k})$& $e_u$ & $e_p$ \\ 
\hline
$0$ & $(800, 800)$ & $30.83\%$ & $8.95\%$ \\ 
$1$ & $(1674, 1661)$ & $0.0178\%$ & $0.0115\%$ \\ 
$2$ & $(2581, 2006)$& $7.32\times 10^{-5}\%$ & $8.66\times 10^{-4}\%$ \\ 
$3$ & $(2716, 2128)$ & $6.99\times 10^{-5}\%$ & $7.73\times 10^{-4}\%$ \\ 
\end{tabular}
\caption{History of online enrichment in Example \ref{exp:2} with $\theta=\gamma=0.3$ and $\omega_i^+$ at  time level $t=0.52$. }
\end{table}

\begin{figure}[H]
\centering
\includegraphics[width=2.5in]{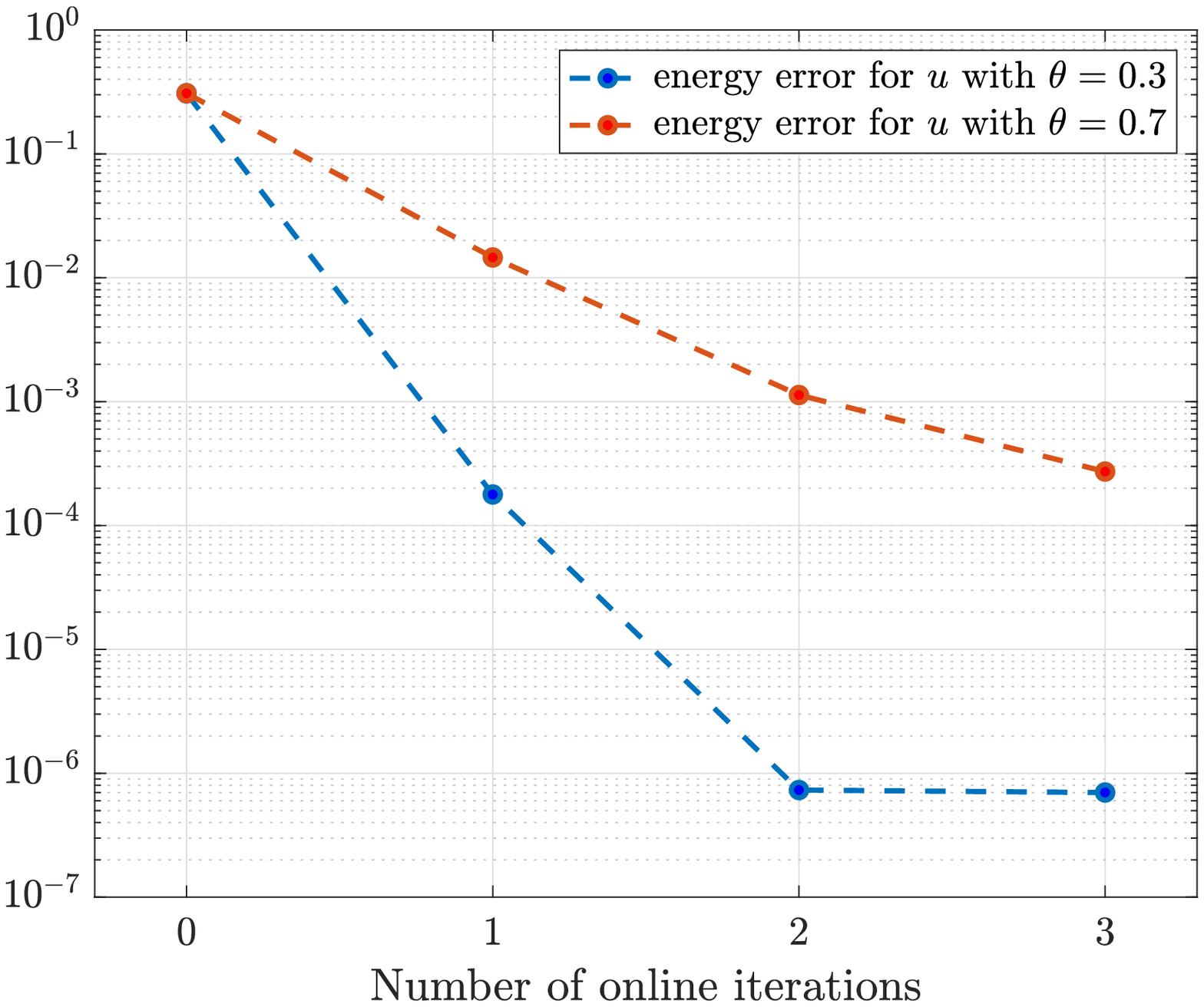} \quad
\includegraphics[width=2.5in]{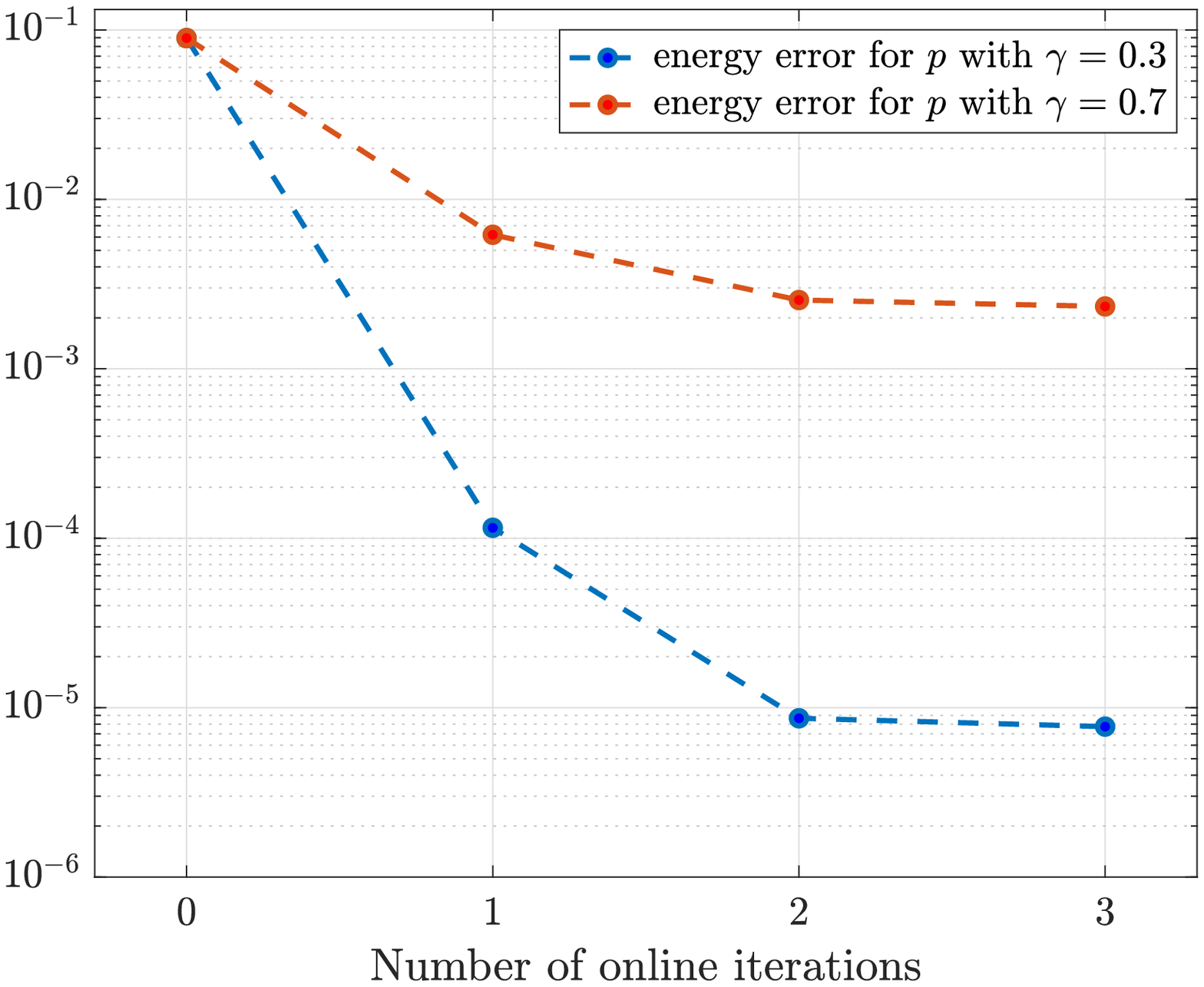}
\caption{Energy errors in Example \ref{exp:2} against the number of online iterations with $\theta=\gamma=0.3$ and $\theta=\gamma=0.7$ using $\omega_i^+$ at  time level $t=0.52$.}
\end{figure}

Figure \ref{fig:5_2_soln_profile} shows the solution profiles  at  time level $t=0.52$. 
We observe a rapid error decay for the first iteration in Example \ref{exp:2}. This remarkable error decay is credited to applying the online process every several steps. The refinement at the previous steps gives us a more accurate approximation of solution at the previous step, which helps getting a better result at the current time step. 

We remark that one can also achieve the level of accuracy  $e_u=13.91\%$ and $e_p=3.86\%$ when the offline degrees of freedom are  $(u_{\text{dof}}^{k}, p_{\text{dof}}^{k})=(1600,1600)$. Compared to this, using the online adaptive method to refine the solution can achieve more accurate results with similar degrees of freedom.

In Example \ref{exp:1}, when the number of online iterations reaches a certain number, there is a bottleneck in the error decay. For instance, in the case with $\theta=\gamma=0.3$ and using the neighborhood-based strategy, the energy error $e_p$ of pressure stalls at $1.21\%$ and the energy error $e_u$ of displacement stalls at around $1.47\%$ in Example \ref{exp:1} after three online iterations. In contrast to this result, Example \ref{exp:2} shows us that performing online enrichment recurrently can resolves this bottleneck. In Example \ref{exp:2}, with online tolerance $\theta=\gamma=0.3$, after two online iterations, the magnitude of the energy errors for both pressure and displacement are reduced to the level of $10^{-4}\%$.

\begin{example}\label{exp:4}
\revi{We consider a time-dependent source function in this example. Specifically, we set the source function to be $f (t, x_1,x_2) =2\pi^2 t \sin(\pi x_1) \sin(\pi x_2)$. The rest of the settings are the same as Example \ref{exp:2}.
It demonstrates the capability of the online enrichment method for the case with time-dependent source term. 
It is worth noting that for time-dependent source term case, the bottleneck caused by online enrichment performed solely at one time step can also be avoided by conducting the online enrichment every several time steps.}
\end{example}

\begin{figure}[htbp!]
\centering
\includegraphics[width = 1.7in]{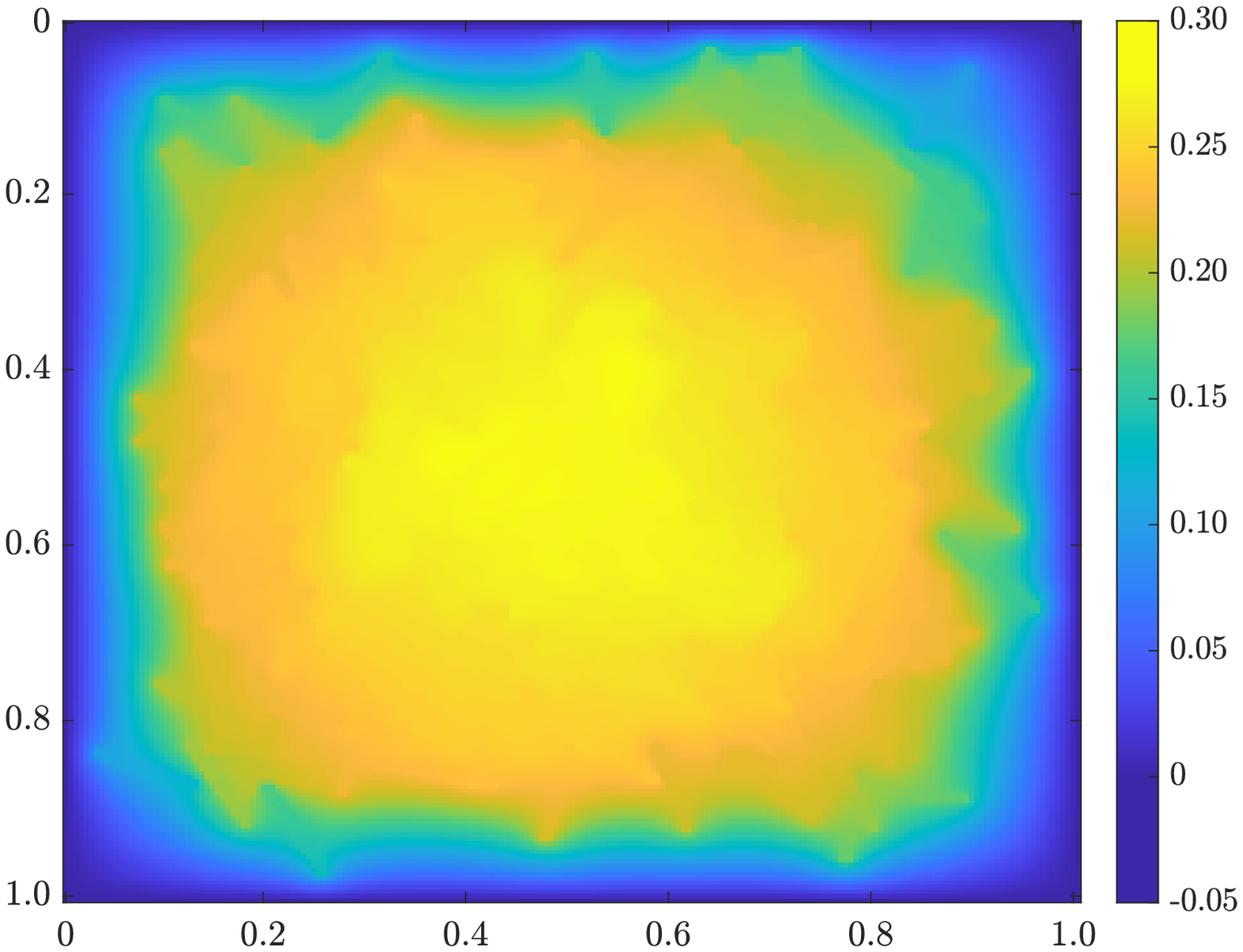} \quad
\includegraphics[width = 1.7in]{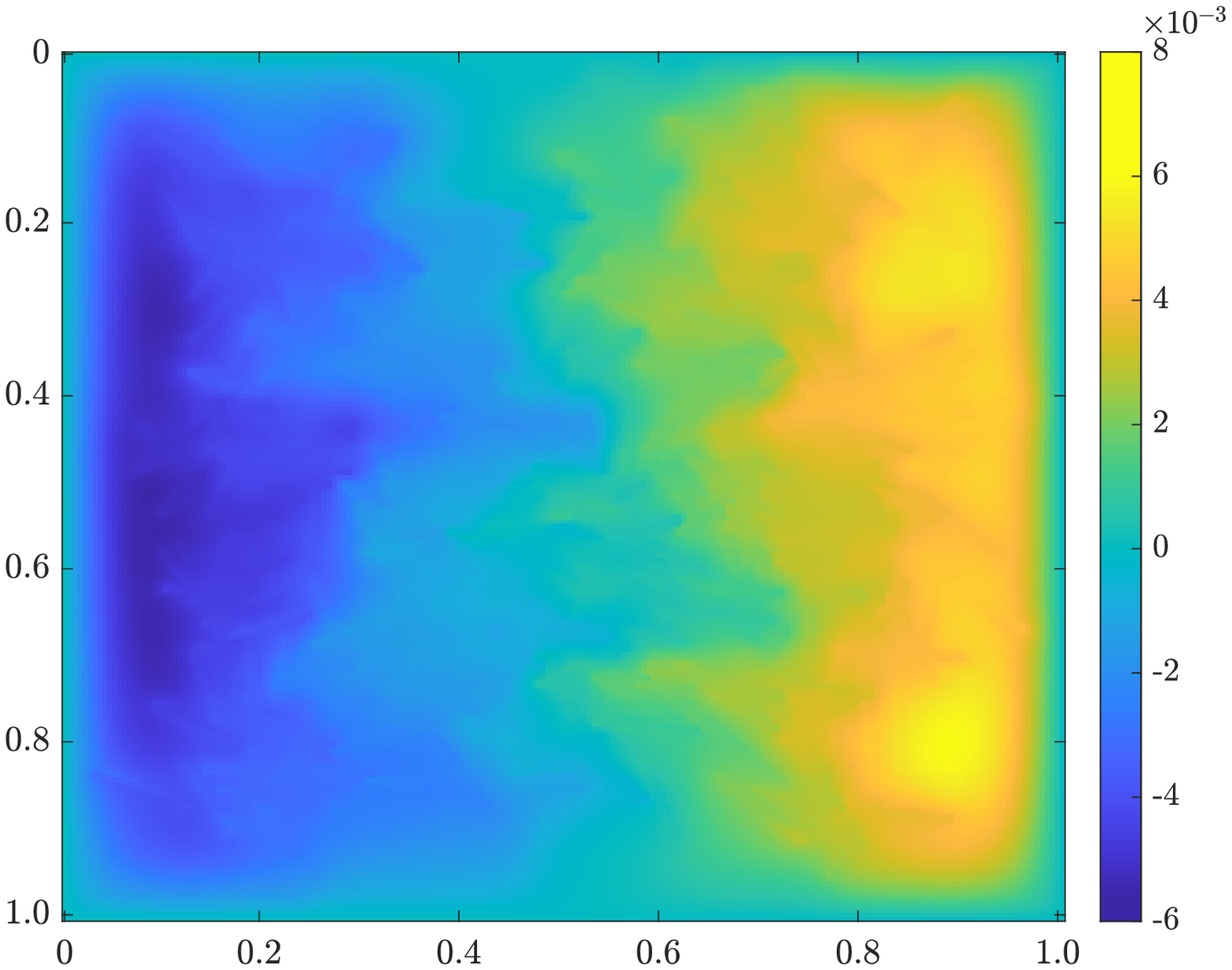}\quad
\includegraphics[width = 1.7in]{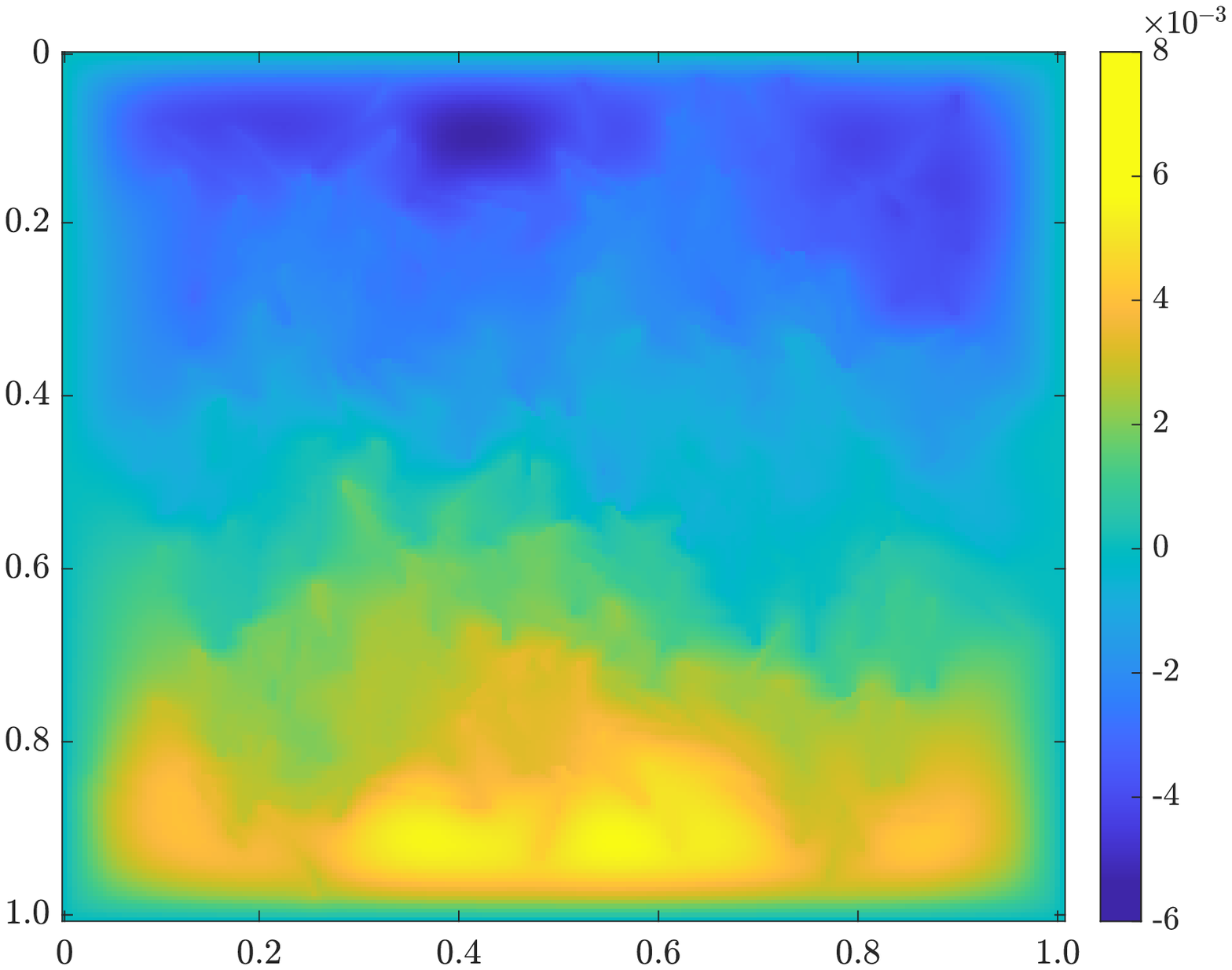}\\
\includegraphics[width = 1.7in]{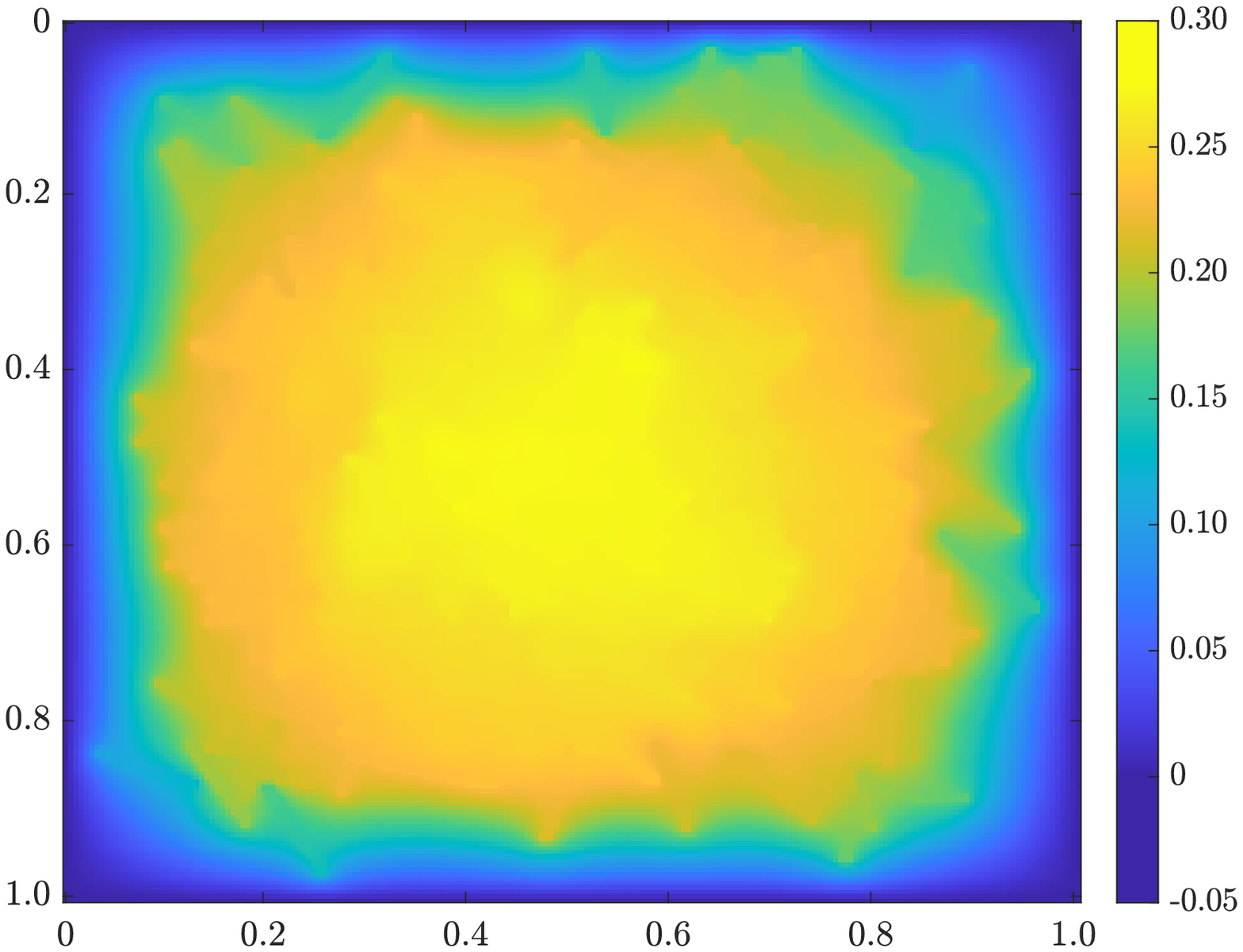}\quad
\includegraphics[width = 1.7in]{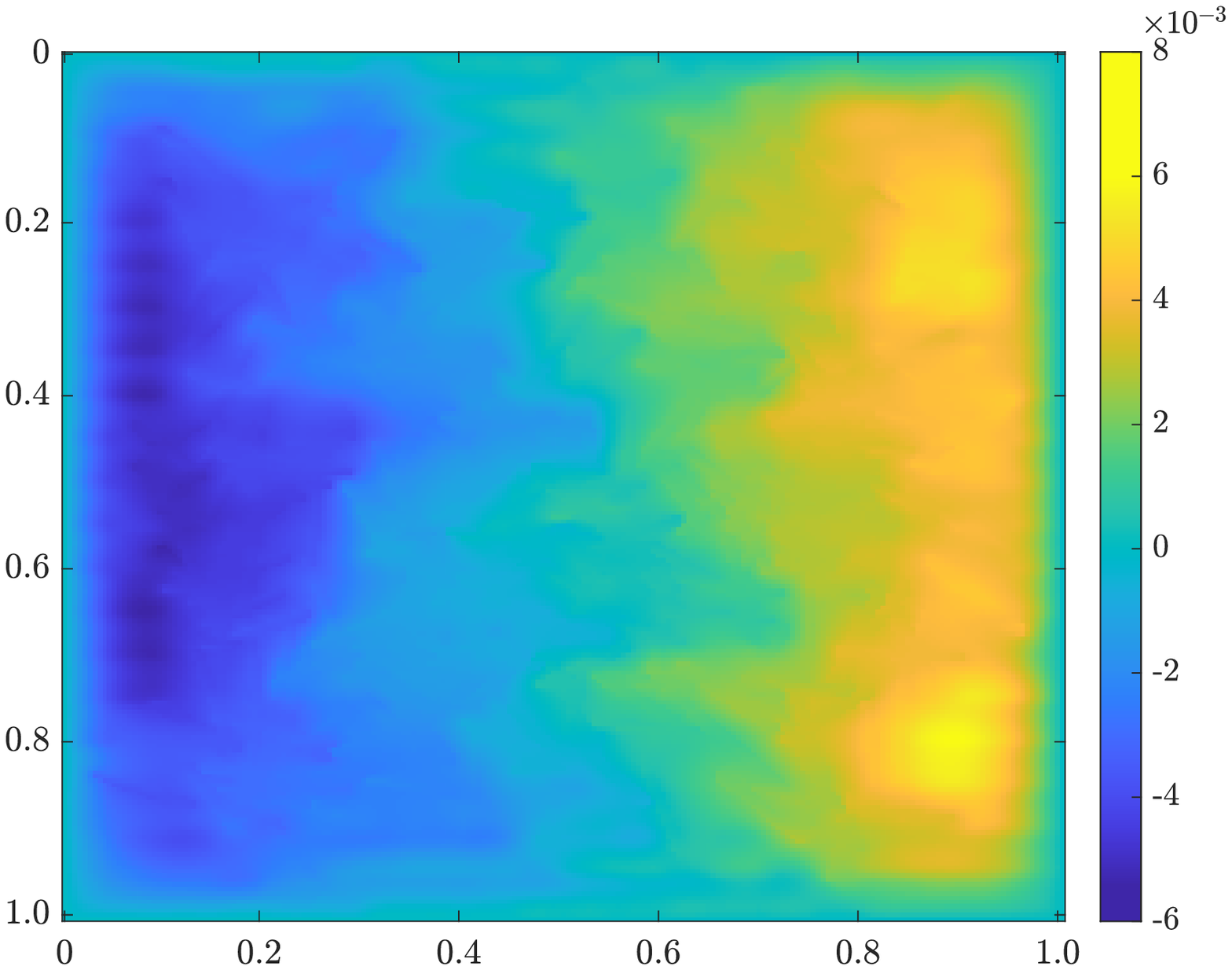}\quad
\includegraphics[width = 1.7in]{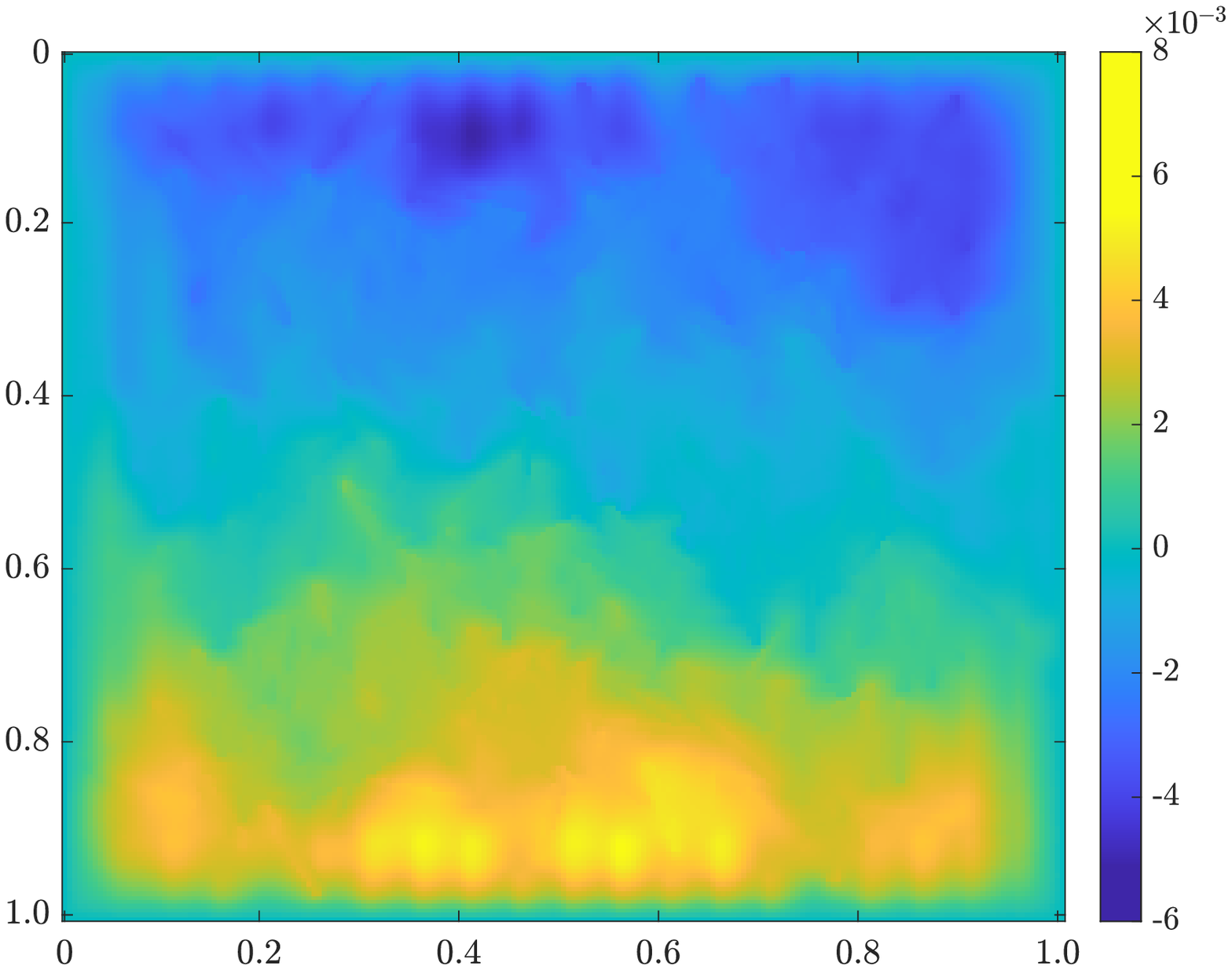}\\
\includegraphics[width = 1.7in]{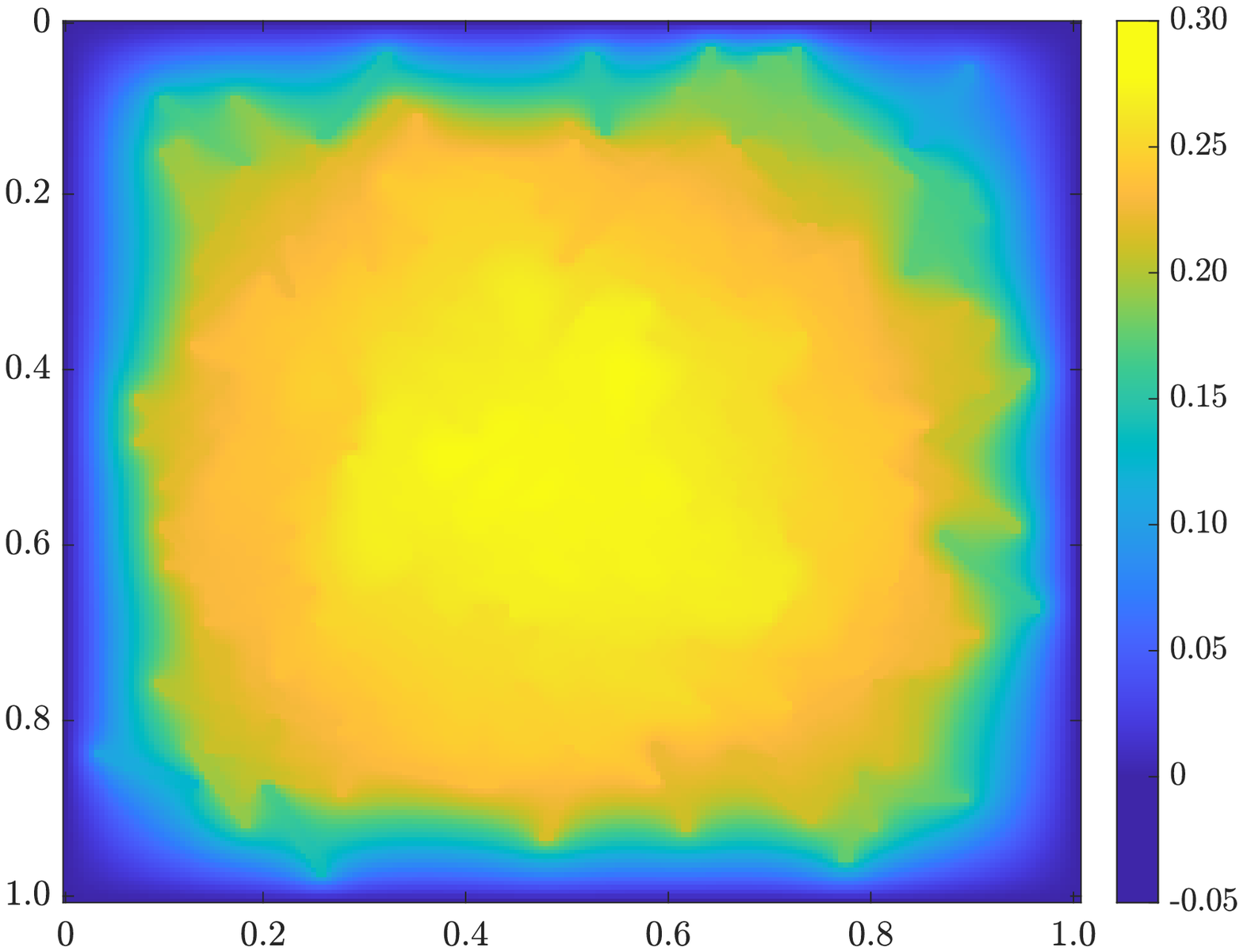}\quad
\includegraphics[width = 1.7in]{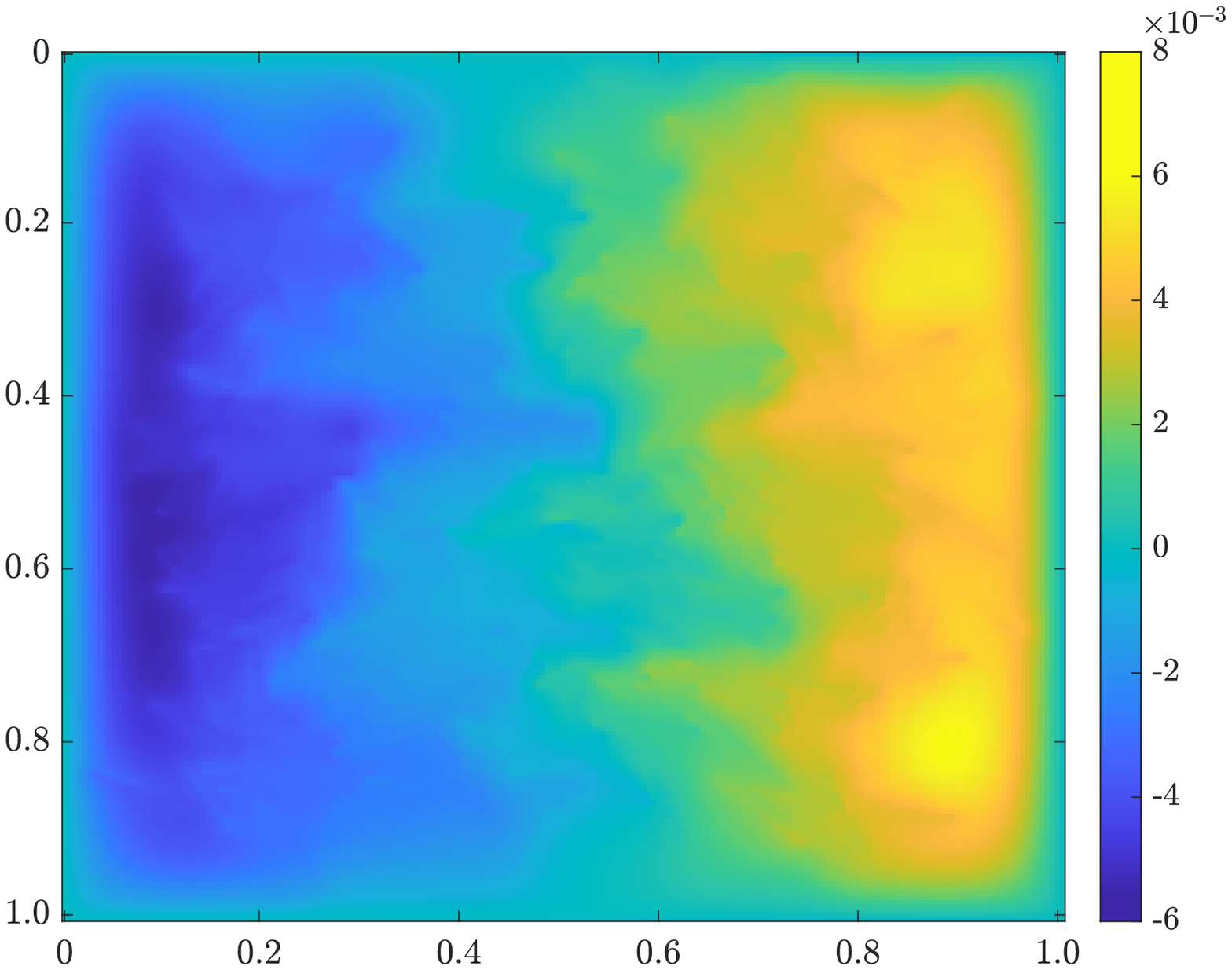}\quad
\includegraphics[width = 1.7in]{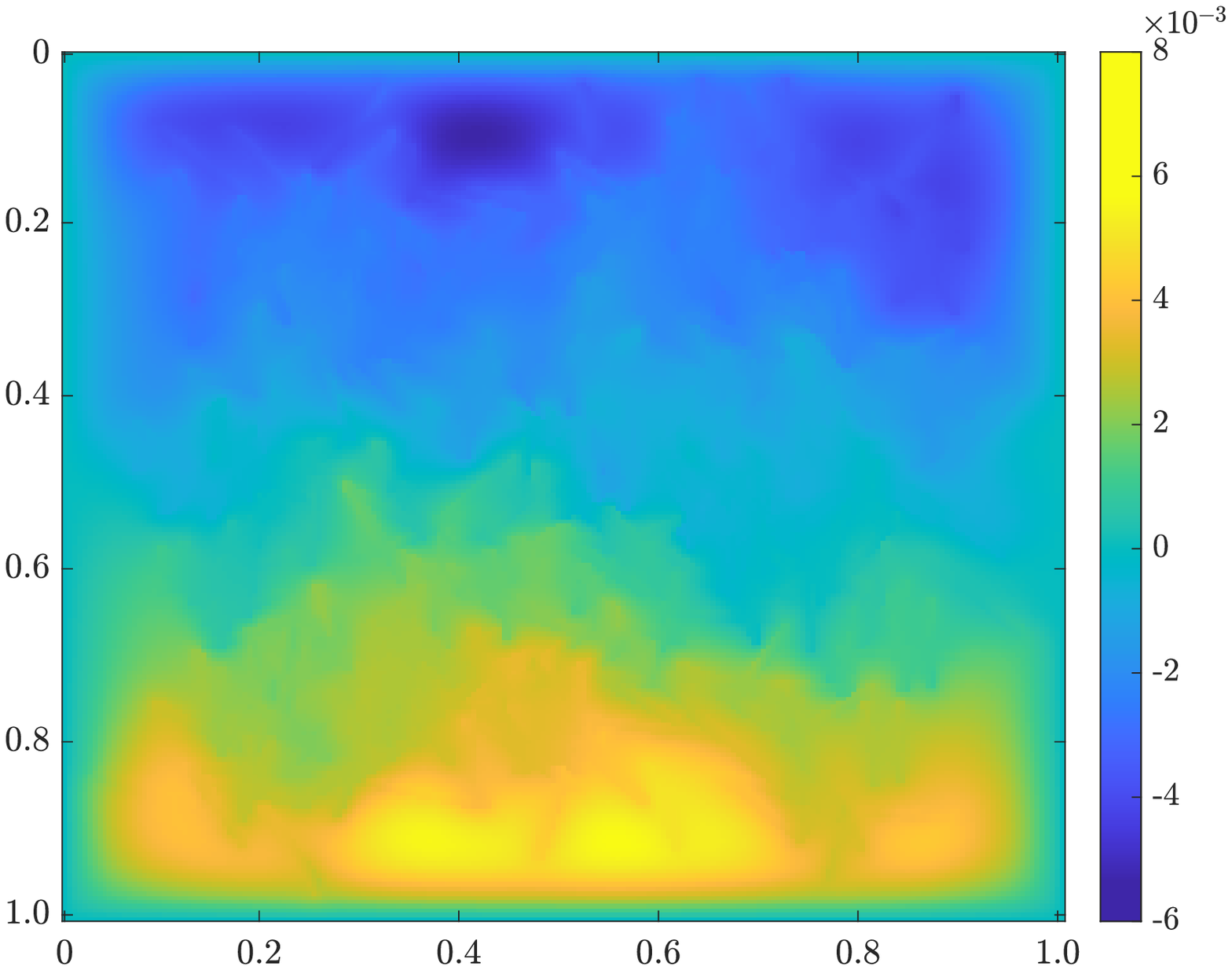}
\caption{Solution profiles (starting from left to right) $p$, $u_1$, $u_2$ of Example \ref{exp:4}. Top: reference solutions; middle: offline solutions; bottom: online solutions with iteration 2 with $\theta=\gamma=0.3$ at  time level $t=T=1$  with  Poisson ratio $0.2$ and the source $f$ is time-dependent. }
\end{figure}

\begin{table}[htbp!]
\centering
\begin{tabular}{c|c||c|c}
$k$ & $(u_{\text{dof}}^{k}, p_{\text{dof}}^{k})$& $e_u$ & $e_p$ \\ 
\hline
$0$ & $(800, 800)$ & $30.83\%$ & $8.95\%$ \\ 
$1$ & $(1065,1094)$ & $2.30\%$ & $1.03\%$ \\ 
$2$ & $(1300,1379)$ & $0.31\%$ & $0.39\%$ \\ 
$3$ & $(1565,1594)$ & $0.05\%$ & $0.04\%$ \\ 
\end{tabular}
\caption{History of online enrichment in Example \ref{exp:4}  with $\theta=\gamma=0.7$, $\omega_i^+$ at  time level $t=T=1$. }
\end{table}

\begin{table}[htbp!]
\centering
\begin{tabular}{c|c||c|c}
$k$ & $(u_{\text{dof}}^{k}, p_{\text{dof}}^{k})$& $e_u$ & $e_p$ \\ 
\hline
$0$ & $(800, 800)$ & $30.83\%$ & $8.95\%$ \\ 
$1$ & $(1523,1518)$ & $0.07\%$ & $0.03\%$ \\ 
$2$ & $(2372,2202)$& $3.14\times 10^{-4}\%$ & $3.49\times 10^{-4}\%$ \\ 
$3$ & $(3070,2410)$ & $3.41\times 10^{-5}\%$ & $3.48\times 10^{-4}\%$ \\ 
\end{tabular}
\caption{History of online enrichment in Example \ref{exp:4}  with $\theta=\gamma=0.3$, $\omega_i^+$ scheme at  time level $t=T=1$. }
\end{table}

\begin{figure}[htbp!]
\centering
\includegraphics[width=2.2in]{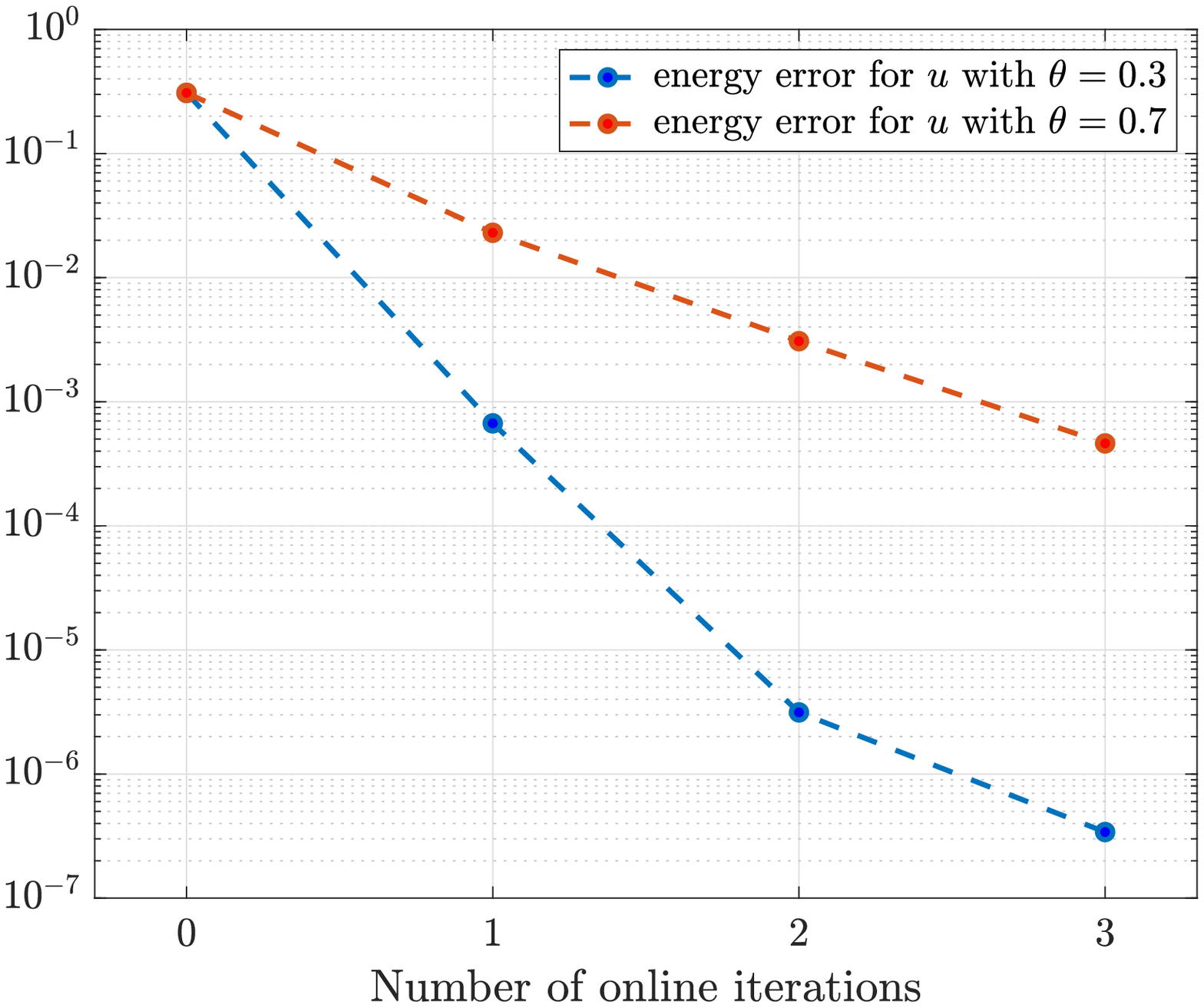} \quad
\includegraphics[width=2.2in]{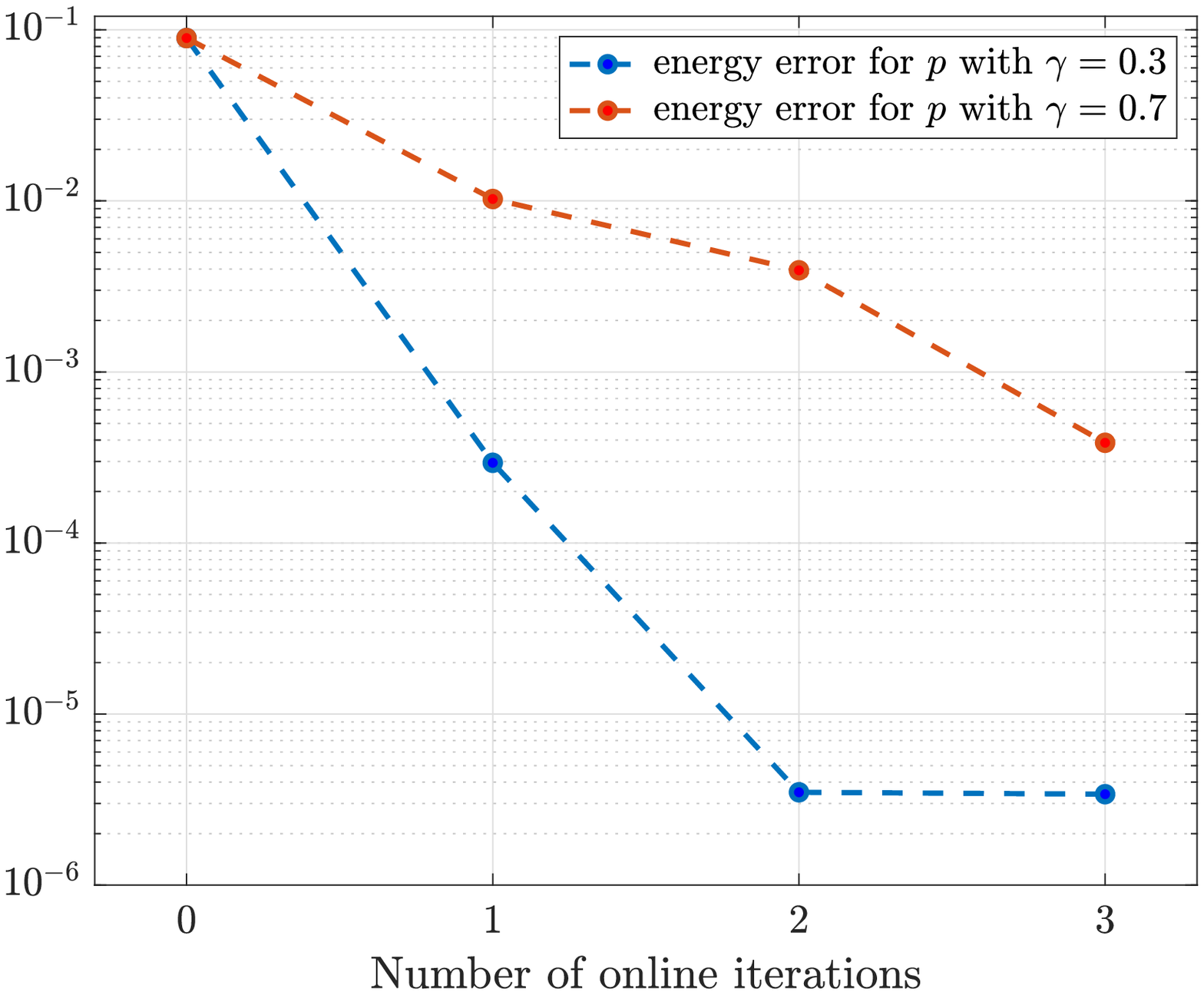}
\caption{Energy errors in Example \ref{exp:4}  against the number of online iterations with $\theta=\gamma=0.3$ and $\theta=\gamma=0.7$ using $\omega_i^+$ scheme at  time level $t=T=1$.}
\end{figure}

\section{Convergence Analysis}\label{sec:anal}
In this section, we analyze the proposed method and provide a theoretical estimate for the online adaptive algorithm. To this aim, we first define two constants $C_0\in  \mathbb{R}$ and $C_1\in  \mathbb{R}$ such that
\begin{equation}\label{eq:constants}
C_0 :=\sup_{0\not=v\in V} \frac{\|\pi_1 (v)\|_{s^1}^2}{\|v\|_a^2} \tand C_1 :=\sup_{0\not =q\in Q} \frac{\|\pi_2 (q)\|_{s^2}^2}{\|q\|_b^2}.\end{equation}
We remark that $C_0\leq \max\{\tilde{\sigma}\} C_p^2$ and $C_1\leq \max\{\tilde{\kappa}\} C_p^2$, where $C_p$ is the Poincar\'{e} constant for the spatial domain $\Omega$, which is the optimal constant in the Poincar\'{e} inequality $\|w\|_{L^2(\Omega)}\leq C_p\|\nabla w\|_{L^2(\Omega)}$ for $w\in H^1_0(\Omega)$.

We remark that the oversampling layers for displacement basis and pressure basis can be different. Also, the oversampling layers at the offline and online stages can be set differently. In the following, for simplicity, we will denote $\ell$ uniformly to be the oversampling parameter for both displacement and pressure basis functions during either offline or online stages. In the convergence analysis, we use the same values for the tolerances $\theta$ and $\gamma$; we shall only write $\theta$ in the analysis. 
Throughout this section, we write $a \lesssim b$ 
where there is a generic constant $C$, independent of spatial discretization parameters and the time step size, such that $a \leq Cb$.

Inspired by Hypothesis 2.2 from \cite{ern2009posteriori} and Lemma 4.4 in \cite{fu2019computational}, we introduce the following assumptions \revii{regarding the multiscale-space construction and the source function and the initial condition of the system for our convergence analysis.} 
\begin{assumption}\label{conv_rate}
\begin{enumerate}
\item
With sufficiently large $\ell$, $J_i^1$, and $J_i^2$, we assume that for all $v\in V$ and $q \in Q$, there exist corresponding elements $v_\ms\in V_\ms$ and $q_{\ms} \in Q_{\ms}$ such that 
\begin{equation}\label{hypothesis2}
\norm{v- v_{\ms}}_a^2 \lesssim  C_e (\ell+1)^d (1+\Lambda^{-1})\norm{v}_a^2, 
\end{equation}
\begin{equation}\label{hypothesis1}
\norm{q-q_\ms}_b^2 \lesssim C_e(\ell+1)^d (1+\Lambda^{-1}) \norm{q}_b^2. 
\end{equation}
Here, the constants $\Lambda$ (see also \eqref{eq:eig1} and \eqref{eq:eig2}) and $C_e$ are defined to be 
\begin{equation}
\displaystyle\Lambda:=\min \left \{ \min_{1\leq i\leq \Ne}\lambda_{J^1_i+1}^i, \min_{1\leq i\leq \Ne}\zeta_{J^2_i+1}^i \right \} \tand C_e := (1+\Lambda^{-1})\left(1+(2(1+\Lambda^{\frac12}))^{-1}\right)^{1-\ell}.
\label{eqn:exp_fact}
\end{equation}
We refer to $C_e$ as the exponential decay factor. 

\item The source function satisfies $f \in {H^1(0,T, H^{-1}(\Omega))}\cap{L^2(0,T, L^2(\Omega))}$.

\item The initial data $p^0$ (see \eqref{eq:init3}) satisfies $p^0 \in H^1_0(\Omega)$. 
\end{enumerate}
\end{assumption}
\noindent The main result of this section is the following convergence theorem.
\begin{theorem}\label{thm:conv}
Suppose that Assumption \ref{conv_rate} holds. 
Let $\{(u_{\mathrm{ms}}^{n,k},p_{\text{ms}}^{n,k}) \in V_{\mathrm{ms}}^{(k)}\times Q_{\mathrm{ms}}^{(k)}\}$ be the sequence of multiscale solutions obtained by our online adaptive enrichment algorithm and $ (u^n,p^n) \in V\times Q$ be the semi-discrete solution of \eqref{eq:semi}. 
Then we have
\begin{equation}\label{eqn:conv}
\begin{aligned}\norm{u^n- u_{\ms}^{n, k+1}}_a^2+\tau\norm{p^n-p_{\ms}^{n,k+1}}_b^2&\lesssim (1+\tau)(1+\Lambda^{-1})[C_r (\ell+1)^d C_e +N_K^2\theta ]\mathcal{E}^k\\
&+\norm{u^{n-1}-u^{n-1}_\ms}_a^2+\norm{p^{n-1}-p^{n-1}_\ms}_b^2
\end{aligned}
\end{equation}
for $n = 1,2,\cdots,N$, where $\mathcal{E}^k := \mathcal{E}^k_u + \mathcal{E}^k_p$, 
$\mathcal{E}_u^k := \norm{u^n-u_\ms^{n,k}}_a^2$, and 
$$
\mathcal{E}_p^k := \norm{D_\tau u^n-D_\tau u_\ms^{n,k}}_a^2+\norm{p^n-p_\ms^{n,k}}_b^2+\norm{D_\tau p^n-D_\tau p_\ms^{n,k}}_c^2.
$$
The constants terms are $C_e$ is the exponential decay factor defined in \eqref{eqn:exp_fact}, $N_K:=\displaystyle{\max_{K \in \mathcal{T}^H} n_K}$ with $n_K$ being the number of coarse nodes of the coarse block $K$, $C_r=\max\{ C_0, C_1\}$, and $\tau$ is the time step size.
\end{theorem}
\begin{remark}
\revi{We note that the convergence rate of the first term depends on two terms $C_e C_r(\ell+1)^d$ and $N_K^2\theta$. By choosing the number of oversampling layers $\ell=O(\log (C_{\text{par}}/H))$, the term $C_e C_r(\ell+1)^d$ tends to zero.  Thus the factor $N_K^2\theta$ dominates the convergence rate. Here, $C_{\text{par}}$ depends on Lam$\acute{e}$ coefficients, Young's modulus $E$, permeability field $\kappa$ and Poisson ration $\nu_p$, see \cite[Sect.\ 5]{chung2017constraint} for more details.}
\end{remark}

Before showing the proof of Theorem \ref{thm:conv}, we will recall a few important results from \cite{chung2017constraint} that are useful in the convergence analysis. The proofs for the elasticity case are similar to that of the pressure case so we omit them here.

\begin{lemma}\label{lm:projectionexist}
There are constants $D_0$ and $D_1$ such that for all $v_{\text{aux}}\in V_{\text{aux}}$ and $q_{\text{aux}} \in Q_{\text{aux}}$, there exist $v\in V$ and $q\in Q$ such that
\begin{eqnarray*}
\begin{split}
\pi_1(v)&=v_{\text{aux}}, \quad \norm{v}_a^2\leq D_0 \norm{v_{\text{aux}}}_{s^1}^2,  \quad\operatorname{supp}(v)\subset \operatorname{supp}(v_{\text{aux}}), \\
\pi_2(q)&=q_{\text{aux}}, \quad \norm{q}_b^2\leq D_1 \norm{q_{\text{aux}}}_{s^2}^2,  \quad\operatorname{supp}(q)\subset \operatorname{supp}(q_{\text{aux}}).
\end{split}
\end{eqnarray*}
\end{lemma}

\begin{lemma} 
\label{lm:basiserror}
We consider the oversampled domain $K_{i}^{+}$ obtained from $K_{i}$ by extending  $\ell$ coarse grid layers with $\ell \geq 2$ .  Let  $v_{j}^{(i)} \in V_{\text {aux }}$ be a given auxiliary multiscale basis function. We let $ \psi_{j, \ms}^{(i)}$ be the multiscale basis functions obtained in \eqref{eq:mineq1_loc} and let $\psi_{j}^{(i)}$ be the global multiscale basis functions obtained in \eqref{eq:minglo1}. Then we have 

$$
\left\|\psi_{j}^{(i)}-\psi_{j,\ms}^{(i)}\right\|_{a}^{2}+\left\|\pi_1\left(\psi_{j}^{(i)}-\psi_{j,\ms}^{(i)}\right)\right\|_{s^1}^{2} \leq C_e\left(\left\|\psi_{j}^{(i)}\right\|_{a}^{2}+\left\|\pi_1\left(\psi_{j}^{(i)}\right)\right\|_{s^1}^{2}\right)
$$
where $C_e$ is the exponential decay factor defined in \eqref{eqn:exp_fact}. 

Similarly, let  $q_{j}^{(i)} \in Q_{\text {aux }}$ be a given auxiliary multiscale basis function, $ \phi_{j, \ms}^{(i)}$ be the multiscale basis functions obtained in \eqref{eq:mineq2_loc} and let $\phi_{j}^{(i)}$ be the global multiscale basis functions obtained in \eqref{eq:minglo2}. Then we have 
$$
\left\|\phi_{j}^{(i)}- \phi_{j, \ms}^{(i)}\right\|_{b}^{2}+\left\|\pi_2\left(\phi_{j}^{(i)}- \phi_{j, \ms}^{(i)}\right)\right\|_{s^2}^{2} \leq C_e\left(\left\|\phi_{j}^{(i)}\right\|_{b}^{2}+\left\|\pi_2\left(\phi_{j}^{(i)}\right)\right\|_{s^2}^{2}\right).
$$
\end{lemma}

\begin{lemma} \label{lm:basiserror_sum}
Assume the same conditions in Lemma \ref{lm:basiserror}, we have
$$
\begin{aligned}
&~~\norm{\sum_{i=1}^{\Ne} \sum_{j=1}^{J_{i}^1} d_{j}^{(i)}(\psi_{j}^{(i)}-\psi_{j, \ms}^{(i)})}_{a}^{2}+\norm{\pi_1(\sum_{i=1}^{\Ne} \sum_{j=1}^{J_{i}^1} d_{j}^{(i)}(\psi_{j}^{(i)}-\psi_{j, \ms}^{(i)}) )}_{s^1}^2 \\
& \lesssim \left(1+\Lambda^{-1}\right)(\ell+1)^{d} \sum_{i=1}^{\Ne} \left (\norm{\sum_{j=1}^{J_{i}^1} d_{j}^{(i)}(\psi_{j}^{(i)}-\psi_{j,\ms}^{(i)})}_{a}^{2}+\norm{\pi_1(\sum_{j=1}^{J_{i}^1}d_{j}^{(i)}(\psi_{j}^{(i)}-\psi_{j,\ms}^{(i)}))}_{s^1}^{2}\right),\end{aligned}
$$
and
$$
\begin{aligned}
&~~\norm{\sum_{i=1}^{\Ne} \sum_{j=1}^{J_{i}^2} c_{j}^{(i)}(\phi_{j}^{(i)}-\phi_{j, \text{ms}}^{(i)})}_{b}^{2}+\norm{\pi_2(\sum_{i=1}^{\Ne} \sum_{j=1}^{J_{i}^2} c_{j}^{(i)}(\phi_{j}^{(i)}-\phi_{j, \text{ms}}^{(i)}) )}_{s^2}^{2} \\
& \lesssim \left(1+\Lambda^{-1}\right)(\ell+1)^{d} \sum_{i=1}^{\Ne}\left(\norm{\sum_{j=1}^{J_{i}^2} c_{j}^{(i)}(\phi_{j}^{(i)}-\phi_{j, \text{ms}}^{(i)})}_{b}^{2}+\norm{\pi_2(\sum_{j=1}^{J_{i}^2} c_{j}^{(i)}(\phi_{j}^{(i)}-\phi_{j, \text{ms}}^{(i)}))}_{s^2}^{2}\right).\end{aligned}
$$
\end{lemma}

\begin{lemma} 
\label{lm:onlinebasiserror}
We consider the oversampled domain $\omega_{i}^{+}$ obtained from $\omega_{i}$ by extending  $\ell$ coarse grid layers with $\ell \geq 2$ .  Let  $\delta_{\text{ms}}^{(i)}$  be the online multiscale basis functions obtained in \eqref{eqn:loc_online_basis} and let $\delta_{\text{glo}}^{(i)}$ be the global multiscale basis functions obtained in \eqref{eqn:glo_online_basis}. Then we have 

$$
\norm{\delta_{\text{glo}}^{(i)}-\delta_{\text{ms}}^{(i)}}_{a}^{2}+\norm{\pi_1(\delta_{\text{glo}}^{(i)}-\delta_{\text{ms}}^{(i)})}_{s^1}^{2} \leq 
C_e\left(\norm{\delta_{\text{glo}}^{(i)}}_{a}^{2}+\norm{\pi_1(\delta_{\text{glo}}^{(i)})}_{s^1}^{2}\right)
$$
where $C_e$ is the exponential decay factor defined in \eqref{eqn:exp_fact}
Furthermore, we have
$$
\norm{\sum_{i=1}^{\Nv}(\delta_{\text{glo}}^{(i)}-\delta_{\text{ms}}^{(i)})}_{a}^{2}+\norm{\sum_{i=1}^{\Nv}\pi_1(\delta_{\text{glo}}^{(i)}-\delta_{\text{ms}}^{(i)})}_{s^1}^{2} \leq (1+\Lambda^{-1})(\ell+1)^d\sum_{i=1}^{\Nv}\left(\norm{\delta_{\text{glo}}^{(i)}-\delta_{\text{ms}}^{(i)}}_{a}^{2}+\norm{\pi_1(\delta_{\text{glo}}^{(i)}-\delta_{\text{ms}}^{(i)})}_{s^1}^{2}\right).
$$
Similarly, let $\rho_{\ms}^{(i)}$'s  be the online multiscale basis functions obtained in \eqref{eqn:loc_online_basis} and $\rho_{\text{glo}}^{(i)}$'s be the global multiscale basis functions obtained in \eqref{eqn:glo_online_basis}. Then we have 
$$
\norm{\rho_{\text{glo}}^{(i)}-\rho_{\text{ms}}^{(i)}}_{b}^{2}+\norm{\pi_2(\rho_{\text{glo}}^{(i)}-\rho_{\text{ms}}^{(i)})}_{s^2}^{2} \leq 
C_e\left(\norm{\rho_{\text{glo}}^{(i)}}_{b}^{2}+\norm{\pi_2(\rho_{\text{glo}}^{(i)})}_{s^2}^{2}\right). 
$$
Furthermore, we have
$$
\norm{\sum_{i=1}^{\Nv}(\rho_{\text{glo}}^{(i)}-\rho_{\text{ms}}^{(i)})}_{b}^{2}+\norm{\sum_{i=1}^{\Nv}\pi_2 (\rho_{\text{glo}}^{(i)}-\rho_{\text{ms}}^{(i)})}_{s^2}^{2}
\leq (1+\Lambda^{-1})(\ell+1)^d\sum_{i=1}^{\Nv}\left(\norm{\rho_{\text{glo}}^{(i)}-\rho_{\text{ms}}^{(i)}}_{b}^{2}+\norm{\pi_2(\rho_{\text{glo}}^{(i)}-\rho_{\text{ms}}^{(i)})}_{s^2}^{2}\right).
$$
\end{lemma}

Next, we consider a pair of Riesz projection operators $\mathfrak{R}_{a, \text{ms}}^{k+1}: V\to V_{\text{ms}}^{(k+1)}$ and $\mathfrak{R}_{b, \ms}^{k+1}: Q \to Q_{\text{ms}}^{(k+1)}$ such that for all $(v,q)\in V\times Q$ 
\begin{equation}
\begin{aligned}
a(v-\mathfrak{R}_{a,{\text{ms}}}^{k+1}(v),v_H)&=0, \\
b(q-\mathfrak{R}_{b,{\text{ms}}}^{k+1}(q),q_H)&=0,
\end{aligned}
\end{equation}
for any $v_H\in V_{\text{ms}}^{(k+1)}$ and $q_H \in Q_{\ms}^{(k+1)}$. 
The projection operators $\mathfrak{R}_{a,\ms}^{k+1}$ and $\mathfrak{R}_{b,\ms}^{k+1}$ have the following properties. 
\begin{lemma}\label{lm:infcontrl}
The following holds for all $(v,q)\in V\times Q$, 
\begin{equation}\label{eq:project_a}
\norm{v-\mathfrak{R}_{a,{\text{ms}}}^{k+1}(v)}_a\leq \inf_{v_H\in V_{\text{ms}}^{(k+1)}} \norm{v-v_H}_a,
\end{equation}
\begin{equation}\label{eq:project_b}
\norm{q-\mathfrak{R}_{b,{\text{ms}}}^{k+1}(q)}_b\leq \inf_{q_H\in  Q_{\text{ms}}^{(k+1)}} \norm{q-q_H}_b.
\end{equation}
\end{lemma}

Now, we are ready to prove Theorem \ref{thm:conv}.
\begin{proof}[Proof of Theorem \ref{thm:conv}]
We denote $(u^{n-1}_\ms,p_{\ms}^{n-1})$ as the multiscale solution at the time step  $t_{n-1}=(n-1)\tau$. We perform the online adaptive procedure at the time step $t_n=n\tau$. Let $\{(u_{\mathrm{ms}}^{n,k},p_{\text{ms}}^{n,k}) \in V_{\mathrm{ms}}^{(k)}\times Q_{\mathrm{ms}}^{(k)}\}$ be the sequence of multiscale solutions obtained by our online adaptive enrichment algorithm.  
In the following, we simply denote $V_{\ms} = V_{\ms}^{(k)}$ and $Q_{\ms} = Q_{\ms}^{(k)}$ for the current online enrichment level $k$. We also write $V_{\ms}^{(k+1)} = V_{\ms}^{(k)} \bigoplus \spa \{ \delta_{\ms}^{(i)} : i \in \mathcal{I}_1 \}$ and $Q_{\ms}^{(k+1)} = Q_{\ms}^{(k)} \bigoplus \spa \{ \rho_{\ms}^{(i)} : i \in \mathcal{I}_2 \}$. 
For the online procedure, in the equation \eqref{eq:weak2}, we denote $D_\tau u_{\text{ms}}^{n,k}=(u_{\mathrm{ms}}^{n,k}-u^{n-1}_\ms)/\tau $ (resp. $D_\tau p_{\text{ms}}^{n,k}=(p_{\mathrm{ms}}^{n,k}-p^{n-1}_\ms)/\tau $). We divide the proof into four steps. 

\noindent\textbf{Step 1.}
Recall the orthogonal decompositions
$V=V_{\text{ms}}\bigoplus \widetilde{V}$ with respect to $a(\cdot,\cdot)$ and $Q=Q_{\text{ms}}\bigoplus \widetilde{Q}$ with respect to $b(\cdot,\cdot)$. The proof starts with a representation of the error term $u^n-u^{n,k}_{\text{ms}}$. Summing the first equation in \eqref{eqn:glo_online_basis} over the indices $i \in \{ 1, \cdots, \Nv \}$ and denoting $\delta_{\glo} := \sum_{i=1}^{\Nv} \delta_{\glo}^{(i)}$, we obtain 
$$
\begin{aligned}
a(\delta_{\glo}, v) + s^1\left (\pi_1 (\delta_{\glo}), \pi_1(v) \right )= \sum_{i=1}^{\Nv}r_{n,i}^{1,k} (v)
=a(u^n-u^{n,k}_{\text{ms}} ,v)-d(v,p^n-p^{n,k}_{\text{ms}})
\end{aligned}$$
for any $v \in V$. 
Then, we have 
$$
\begin{aligned}
a (u^n-u^{n,k}_{\text{ms}}-\delta_{\glo}, v) -d(v,p^n-p^{n,k}_{\text{ms}} )= s^1\left (\pi_1 (\delta_{\glo}), \pi_1(v) \right )
\end{aligned}$$
for any $v \in V$. Let $\widetilde{u}^d$ be the solution of the following equation
\begin{equation}
a(\widetilde{u}^d,v)=d(v,p^n-p^{n,k}_{\text{ms}} ) \quad \tforall v\in V.
\end{equation}
This gives us 
\begin{equation}\label{eq:orthogonality_u}
\begin{aligned}
a (u^n-u^{n,k}_{\text{ms}}-\delta_{\glo}-\widetilde{u}^d, v) = s^1\left (\pi_1 (\delta_{\glo}), \pi_1(v) \right ) \quad \tforall v \in V
\end{aligned}\end{equation}
and thus 
$$
 a (u^n-u^{n,k}_{\text{ms}}-\delta_{\glo}-\widetilde{u}^d, v) = 0 \quad \text{ for all } v \in \widetilde{V}. 
$$
Using the decomposition $V=V_{\text{ms}}\bigoplus \widetilde{V}$, we have
$u^n-u^{n,k}_{\text{ms}}-\delta_{\glo}-\widetilde{u}^d\in \widetilde{V}^\perp=V_{\ms}$ and there exists a set of coefficients $\left \{ d_j^{(i)} \right \}_{1 \leq i \leq \Ne, 1 \leq j \leq J_i^1}$ such that 
\begin{equation}\label{eq:newrepres_u}
u^n-u^{n,k}_{\text{ms}}-\delta_{\glo}-\widetilde{u}^d=\sum_{i=1}^{\Ne} \sum_{j=1}^{J_i^1} d_j^{(i)}\psi_{j}^{(i)}. 
\end{equation}
On the other hand, 
summing the second equation in \eqref{eqn:glo_online_basis} over the indices $i \in \{ 1, \cdots, \Nv \}$ and denoting $\rho_{\glo} := \sum_{i=1}^{\Nv} \rho_{\glo}^{(i)}$, we obtain 
$$\begin{aligned}
b(\rho_{\glo} , q) + s^2 \left ( \pi_2 ( \rho_{\glo} ), \pi_2 (q) \right ) &= \sum_{i=1}^{\Nv}r_{n,i}^{2,k}(q)=(f^n,q)-b(p_{\text{ms}}^{n,k},q)-c(D_\tau p_{\text{ms}}^{n,k},q)-d(D_\tau u^{n,k}_{\text{ms}} ,q) \\
&= d(D_\tau u^n-D_\tau u^{n,k}_{\text{ms}} ,q)+ b(p^n-p_{\text{ms}}^{n,k},q)+c(D_\tau p^n-D_\tau p_{\text{ms}}^{n,k},q) \\
\end{aligned}$$
for any $q \in Q$. This leads to
$$
 b(p^n-p_{\text{ms}}^{n,k}-\rho_{\glo},q)+d(D_\tau u^n-D_\tau u^{n}_{\text{ms}} ,q)+c(D_\tau p^n-D_\tau p_{\text{ms}}^{n,k},q)= s^2 \left ( \pi_2 ( \rho_{\glo} ), \pi_2 (q) \right )
$$
for any $q \in Q$. 
Let $\widetilde{p}^c$ and  $ \widetilde{p}^{d}$ be the solutions of the following equations: 
\begin{equation}
\label{p^c}
b(\widetilde{p}^c,q)=-c(D_\tau p^n-D_\tau p_{\text{ms}}^{n,k},q)
\end{equation}
and 
\begin{equation}
\label{p^d}
b(\widetilde{p}^d,q)=-d(D_\tau u^n-D_\tau u^{n,k}_{\text{ms}},q). 
\end{equation}
Then, we have 
\begin{equation}\label{eq:orthogonality_q}
 b(p^n-p_{\text{ms}}^{n,k}-\widetilde{p}^c-\widetilde{p}^d-\rho_{\glo} ,q)
 = s^2 \left ( \pi_2 (\rho_{\glo}), \pi_2 (q) \right )
\end{equation}
for any $q \in Q$. 
Similar to the case of the displacement variable, using the decomposition $Q=Q_{\text{ms}}\bigoplus \widetilde{Q}$, we have
$p^n-p_{\text{ms}}^{n,k}-\widetilde{p}^c-\widetilde{p}^d-\rho_{\glo} \in \widetilde{Q}^\perp=Q_{\ms}$ and there exists a set of coefficients $\left \{ c_j^{(i)} \right \}_{1 \leq i \leq \Ne, 1 \leq j \leq J_i^1}$ such that 
\begin{equation}\label{eq:newrepres}
p^n-p_{\text{ms}}^{n,k}-\widetilde{p}^c-\widetilde{p}^d-\rho_{\glo}=\sum_{i=1}^{\Ne} \sum_{j=1}^{J_i^2} c_j^{(i)}\phi_{j}^{(i)}. 
\end{equation}

\noindent\textbf{Step 2.} 
We localize each $\phi_j^{(i)}$ in \eqref{eq:newrepres}, $\psi_{j}^{(i)}$ in \eqref{eq:newrepres_u} and estimate the error. In particular,  we estimate the localized terms as follows: 
$$ \mathcal{E}_{\text{loc}}^p:= \norm{\sum_{i=1}^{\Ne} \sum_{j=1}^{J_{i}^2} c_{j}^{(i)}(\phi_{j}^{(i)}-\phi_{j, \text{ms}}^{(i)})}_{b}^{2}+\norm{\sum_{i=1}^{\Ne} 
\sum_{j=1}^{J_{i}^2} c_{j}^{(i)}\pi_2(\phi_{j}^{(i)}-\phi_{j, \text{ms}}^{(i)})}_{s^2}^{2}$$
and 
$$\mathcal{E}_{\text{loc}}^u:= \norm{\sum_{i=1}^{\Ne} \sum_{j=1}^{J_{i}^1} d_{j}^{(i)}(\psi_{j}^{(i)}-\psi_{j, \text{ms}}^{(i)})}_{a}^{2}+\norm{\sum_{i=1}^{\Ne} 
\sum_{j=1}^{J_{i}^1} d_{j}^{(i)}\pi_1(\psi_{j}^{(i)}-\psi_{j, \text{ms}}^{(i)})}_{s^1}^{2}.$$
Define $\xi_p:= p^n-p_{\text{ms}}^{n,k}-\widetilde{p}^c-\widetilde{p}^d-\rho_{\glo}$ and $q_{\text{aux}}^{(i)}:=\sum_{j=1}^{J_i^2}c_j^{(i)}q_j^i\in Q_{\text{aux}} (K_i)$. Using  \eqref{eq:newrepres} and the linearities of $b(\cdot, \cdot)$ and $s^2(\cdot, \cdot)$, we have
\begin{equation}\label{eq:q_j^i}
\begin{aligned}
b(\xi_p,q)+s^2(\pi_2(\xi_p),\pi_2(q))&= \sum_{i=1}^{\Ne} \sum_{j=1}^{J_i^2} c_j^{(i)} \left [ b(\phi_{j}^{(i)},q)+s^2(\pi_2(\phi_{j}^{(i)}),\pi_2(q)) \right ] =\sum_{i=1}^{\Ne}  s^2(q_{\aux}^{(i)}, \pi_2(q))
\end{aligned}
\end{equation}
for any $q \in Q$, 
where the last equality follows from \eqref{eq:minglo2eqn}. For each $q_{\aux}^{(i)} \in Q_{\aux}(K_i)$, by Lemma \ref{lm:projectionexist}, there exists a function $\tilde{q}^{(i)}\in Q(K_i)$ such that $\pi_2(\tilde{q}^{(i)})=q_{\text{aux}}^{(i)}$ and 
$$
\norm{\tilde{q}^{(i)}}_{b(K_i)}^2\leq D_1\norm{q_{\text{aux}}^{(i)}}_{s^2(K_i)}^2.
$$
Taking $q=\tilde{q}^{(i)}$ in \eqref{eq:q_j^i} and using Cauchy-Schwarz inequality, we have 
$$
\begin{aligned}
\norm{q_{\text{aux}}^{(i)}}_{s^2(K_i)}^2&=b(\xi_p,\tilde{q}^{(i)})+s^2(\pi_2(\xi_p),\pi_2(\tilde{q}^{(i)}))\\
&\leq (\norm{\xi_p}_{b(K_i)}^2+\norm{\pi_2(\xi_p)}_{s^2(K_i)}^2)^{\frac{1}{2}}(\norm{\tilde{q}^{(i)}}_{b(K_i)}^2+\norm{\pi_2(\tilde{q}^{(i)})}_{s^2(K_i)}^2)^{\frac{1}{2}}\\
&\leq (\norm{\xi_p}_{b(K_i)}^2+\norm{\pi_2(\xi_p)}_{s^2(K_i)}^2)^{\frac{1}{2}}(1+D_1)^{\frac{1}{2}} \norm{q_{\text{aux}}^{(i)}}_{s^2(K_i)}. 
\end{aligned}
$$
Applying the orthogonality of the eigenfunctions $q_j^{(i)}$'s and the normalization condition (i.e., $s^2_i(q_j^{(i)},q_j^{(i)})=1$), we have
$$\begin{aligned}
\sum_{i=1}^{\Ne}\sum_{j=1}^{J_i^2} \left ( c_j^{(i)} \right )^2
&=\sum_{i=1}^{\Ne}\norm{q_{\text{aux}}^{(i)}}_{s^2(K_i)}^2 \leq (1+D_1)(\|\xi_p\|_{b}^2+\|\pi_2(\xi_p)\|_{s^2}^2)\leq (1+D_1)(1+C_1)\norm{\xi_p}_b^2, 
\end{aligned}
$$
where the last inequality follows from the definition of the constant $C_1$. 
We denote $\mathcal{C}_1 := (1+D_1)(1+C_1)$. 
Recalling the definition of $\xi_p$, we have 
$$\begin{aligned}
\sum_{i=1}^{\Ne}\sum_{j=1}^{J_i^2}\left ( c_j^{(i)} \right )^2
&\leq \mathcal{C}_1\norm{p^n-p_{\text{ms}}^{n,k}-\widetilde{p}^c-\widetilde{p}^d-\rho_{\glo}}_b^2
\leq  2\mathcal{C}_1 \left [ \norm{p^n-p_{\text{ms}}^{n,k}-\widetilde{p}^c-\widetilde{p}^d}_b^2+
\norm{\rho_{\glo}}_b^2 \right ]\\
&\leq 4\mathcal{C}_1\norm{p^n-p_{\text{ms}}^{n,k}+\widetilde{p}^c+\widetilde{p}^d}_b^2\\
&\lesssim \mathcal{C}_1 \left [\norm{p^n-p_{\text{ms}}^{n,k}}^2_b+\norm{D_\tau p^n-D_\tau p_{\text{ms}}^{n,k}}_c^2+\norm{D_\tau u^n-D_\tau u^{n,k}_{\text{ms}}}_a^2 \right ],
\end{aligned}
$$
where the penultimate inequality follows from \eqref{eq:orthogonality_q}. 
Consequently, following from Lemmas \ref{lm:basiserror} and \ref{lm:basiserror_sum} as well as the fact that $\norm{\phi_j^{(i)}}_b^2+\norm{\pi_2(\phi_j^{(i)})}^2_{s^2}\leq \norm{q_j^i}^2_{s^2}=1$, we have 
\begin{equation}
\begin{aligned}
\mathcal{E}_{\text{loc}}^p &\lesssim \mathcal{C}_1 C_e (\ell+1)^d (1+\Lambda^{-1})  \left [ \| p^n-p_{\text{ms}}^{n,k}\|^2_b+\|D_\tau p^n-D_\tau p_{\text{ms}}^{n,k}\|_c^2+\|D_\tau u^n-D_\tau u^{n,k}_{\text{ms}}\|_a^2 \right].
\end{aligned}\end{equation}
On the other hand, 
employing the same technique, one can obtain the estimate for $\mathcal{E}_{\text{loc}}^u$. 
We only sketch the crucial steps of the analysis. Denote 
$\xi_u:= u^n-u^{n,k}_{\text{ms}}-\delta_{\glo}-\widetilde{u}^d$. 
One can show that 
$$\begin{aligned}
\sum_{i=1}^{\Ne}\sum_{j=1}^{J_i^1} \left (d_j^{(i)} \right )^2 & \leq 
\mathcal{C}_0 \norm{\xi_u}_a^2,
 \end{aligned}
$$
where $\mathcal{C}_0 := (1+D_0)(1+C_0)$. 
Recalling the definitions of $\xi_u$ and $\widetilde{u}^d$, we obtain 
$$\begin{aligned}
\sum_{i=1}^{\Ne}\sum_{j=1}^{J_i^1} \left ( d_j^{(i)} \right )^2
&\lesssim \mathcal{C}_0 \left [\| u^n-u^{n,k}_{\text{ms}}\|_a^2+\|p^n-p^{n,k}_{\text{ms}}\|_b^2 \right ].
\end{aligned}
$$
As a result, we have 
\begin{equation}
\begin{aligned}
\mathcal{E}_{\text{loc}}^u &\lesssim \mathcal{C}_0 C_e (\ell+1)^d (1+\Lambda^{-1}) \left [ \norm{u^n-u^{n,k}_{\text{ms}}}_a^2+\norm{p^n-p^{n,k}_{\text{ms}}}_b^2 \right ].
\end{aligned}\end{equation}

\noindent\textbf{Step 3.} 
Next, we derive an estimate for $\rho_{\glo}^{(i)}-\rho_{\ms}^{(i)}$ and $\delta_{\glo}^{(i)} - \delta_{\ms}^{(i)}$. By Lemma \ref{lm:onlinebasiserror}, we have 

$$\|\rho_{\text{glo}}^{(i)}-\rho_{\text{ms}}^{(i)}\|_{b}^{2}+\|\pi_2(\rho_{\text{glo}}^{(i)}-\rho_{\text{ms}}^{(i)})\|_{s^2}^{2} \leq 
C_e\left(\norm{\rho_{\text{glo}}^{(i)}}_{b}^{2}+\norm{\pi_2(\rho_{\text{glo}}^{(i)})}_{s^2}^{2}\right).$$
Taking $q = \rho_{\glo}^{(i)}$ in \eqref{eqn:glo_online_basis} and making use of the definition of the local residual operator, we obtain
$$
\begin{aligned}
&\quad \norm{\rho_{\text{glo}}^{(i)}}_{b}^{2}+\norm{\pi_2(\rho_{\text{glo}}^{(i)})}_{s^2}^{2}= 
(f^n, \chi_{\omega_i}^2  \rho_{\text{glo}}^{(i)}) - b(p_{\text{ms}}^{n,k},  \chi_{\omega_i}^2  \rho_{\text{glo}}^{(i)}) - c(D_\tau p_{\text{ms}}^{n,k}, \chi_{\omega_i}^2  \rho_{\text{glo}}^{(i)}) - d(D_\tau u_{\text{ms}}^{n,k}, \chi_{\omega_i}^2  \rho_{\text{glo}}^{(i)})\\
&=d(D_\tau u^n-D_\tau u^{n,k}_{\text{ms}} ,\chi_{\omega_i}^2  \rho_{\text{glo}}^{(i)})+ b(p^n-p_{\text{ms}}^{n,k},\chi_{\omega_i}^2  \rho_{\text{glo}}^{(i)})+c(D_\tau p^n-D_\tau p_{\text{ms}}^{n,k},\chi_{\omega_i}^2  \rho_{\text{glo}}^{(i)})\\
&\lesssim \norm{\chi_{\omega_i}^2  \rho_{\text{glo}}^{(i)}}_{b(\omega_i)}
\left (\beta_1\beta_2\|D_\tau u^n-D_\tau u^{n,k}_{\text{ms}}\|_{a(\omega_i)}+\|p^n-p_{\text{ms}}^{n,k}\|_{b(\omega_i)}+\beta_2\|D_\tau p^n-D_\tau p_{\text{ms}}^{n,k}\|_{c(\omega_i)} \right ) \\
&\lesssim \sqrt{2}\left (\norm{\rho_{\text{glo}}^{(i)}}_{b(\omega_i)}^{2}+\norm{\rho_{\text{glo}}^{(i)}}_{s^2(\omega_i)}^{2}\right)^{\frac{1}{2}}
\left (\beta_1\beta_2\norm{D_\tau u^n-D_\tau u^{n,k}_{\text{ms}}}_{a(\omega_i)}+\norm{p^n-p_{\text{ms}}^{n,k}}_{b(\omega_i)}+\beta_2\norm{D_\tau p^n-D_\tau p_{\text{ms}}^{n,k}}_{c(\omega_i)} \right ).
\end{aligned}$$
Since $\pi_2$ is an orthogonal projection onto the auxiliary space $Q_{\aux}$, we have
$$
\begin{aligned}
\norm{\rho_{\text{glo}}^{(i)}}_{s^2(\omega_i)}^{2}&\leq \norm{\pi_2(\rho_{\text{glo}}^{(i)})}_{s^2(\omega_i)}^{2}+\norm{(I-\pi_2)(\rho_{\text{glo}}^{(i)})}_{s^2(\omega_i)}^{2}\leq 
\norm{\pi_2(\rho_{\text{glo}}^{(i)})}_{s^2(\omega_i)}^{2}+\Lambda^{-1}\norm{\rho_{\text{glo}}^{(i)}}_{b(\omega_i)}^{2} \\
& \leq \left ( 1+ \Lambda^{-1} \right ) \left (\norm{\rho_{\text{glo}}^{(i)}}_{b(\omega_i)}^{2}+\norm{\pi(\rho_{\text{glo}}^{(i)})}_{s^2(\omega_i)}^{2}\right), 
\end{aligned}
$$
where the penultimate inequality follows from the fact that $(I-\pi_2)(\rho_{\text{glo}}^{(i)})$ is spanned by the eigenfunctions $q_j^i$'s with $j\geq J_i^2+1$ of the spectral problem \eqref{eq:eig2}. Therefore, we have
\begin{equation}
\begin{aligned}
&~~~\sum_{i=1}^{\Nv} \left [ \norm{\rho_{\text{glo}}^{(i)}-\rho_{\text{ms}}^{(i)}}_{b}^{2}+\norm{\pi_2(\rho_{\text{glo}}^{(i)}-\rho_{\text{ms}}^{(i)})}_{s^2}^{2} \right ]\\
&\leq 2C_e(1+\Lambda^{-1}) \sum_{i=1}^{\Nv} \left [ \beta \norm{D_\tau u^n-D_\tau u^{n,k}_{\text{ms}}}_{a(\omega_i)}+\norm{p^n-p_{\text{ms}}^{n,k}}_{b(\omega_i)}+\beta_2\norm{D_\tau p^n-D_\tau p_{\text{ms}}^{n,k}}_{c(\omega_i)} \right ] ^2\\
&\lesssim N_KC_e(1+\Lambda^{-1}) \left ( \beta^2 \norm{D_\tau u^n-D_\tau u^{n,k}_{\text{ms}}}_{a}^2
+\norm{p^n-p_{\text{ms}}^{n,k}}_{b}^2
+\beta_2^2\norm{D_\tau p^n-D_\tau p_{\text{ms}}^{n,k}}_{c}^2 \right ),
\end{aligned}
\end{equation}
where $\beta := \beta_1 \beta_2$. Without loss of generality, we assume that $\min \{ \beta_1, \beta_2 \} \geq 1$. 
We can also derive an estimate for $\delta_{\glo}^{(i)}-\delta_{\ms}^{(i)}.$ By Lemma \ref{lm:onlinebasiserror}, we have 
$$\norm{\delta_{\text{glo}}^{(i)}-\delta_{\text{ms}}^{(i)}}_{a}^{2}+\norm{\pi_1(\delta_{\text{glo}}^{(i)}-\delta_{\text{ms}}^{(i)})}_{s^1}^{2} \leq 
C_e\left(\norm{\delta_{\text{glo}}^{(i)}}_{a}^{2}+\norm{|\pi_1(\delta_{\text{glo}}^{(i)})}_{s^1}^{2}\right).$$
Taking $v = \delta_{\glo}^{(i)}$ in \eqref{eqn:loc_online_basis} and making use of the definition of the local residual operator, we obtain
$$
\begin{aligned}
\norm{\delta_{\text{glo}}^{(i)}}_{a}^{2}+\norm{\pi_1(\delta_{\text{glo}}^{(i)})}_{s^1}^{2}&
= a(u_{\text{ms}}^{n,k}, \chi_{\omega_i}^1 \delta_{\text{glo}}^{(i)}) -d(\chi_{\omega_i}^2 \delta_{\text{glo}}^{(i)},p_{\text{ms}}^{n,k}) \\
&=a(u_{\text{ms}}^{n,k}-u^n, \chi_{\omega_i}^2 \delta_{\text{glo}}^{(i)}) -d(\chi_{\omega_i}^2 \delta_{\text{glo}}^{(i)},p_{\text{ms}}^{n,k}-p^n)\\
&\leq \|u^n-u_{\text{ms}}^{n,k}\|_{a(\omega_i)}\|\chi_{\omega_i}^2\delta_{\text{glo}}^{(i)}\|_{a(\omega_i)}+\|p^n-p_{\text{ms}}^{n,k}\|_{b(\omega_i)}\|\chi_{\omega_i}^2\delta_{\text{glo}}^{(i)}\|_{a(\omega_i)}\\
&\leq 
\sqrt{2}\left (\norm{\delta_{\text{glo}}^{(i)}}_{a(\omega_i)}^{2}+\norm{\delta_{\text{glo}}^{(i)}}_{s^1(\omega_i)}^{2} \right )^{\frac{1}{2}}
\left (\norm{u^n-u_{\text{ms}}^{n,k}}_{a(\omega_i)}+\norm{p^n-p_{\text{ms}}^{n,k}}_{b(\omega_i)} \right ) .
\end{aligned}$$
Due to the fact that $\pi_1$ is an orthogonal projection onto the auxiliary space $V_{\aux}$, one can show that 
$$
\norm{\delta_{\text{glo}}^{(i)}}_{s^1(\omega_i)}^{2} \leq  \left ( 1+ \Lambda^{-1} \right )
\left ( \norm{\delta_{\text{glo}}^{(i)}}_{a(\omega_i)}^{2} + \norm{\delta_{\text{glo}}^{(i)})}_{s^1(\omega_i)}^{2}\right ).
$$
Therefore, we have
\begin{equation}
\sum_{i=1}^{\Nv} \left [ \norm{\delta_{\text{glo}}^{(i)}-\delta_{\text{ms}}^{(i)}}_{a}^{2}+\norm{\pi_1(\delta_{\text{glo}}^{(i)}-\delta_{\text{ms}}^{(i)})}_{s^1}^{2} \right ] \leq 4N_KC_e(1+\Lambda^{-1}) \left (\norm{u^n - u_{\text{ms}}^{n,k}}_{a}^2+\norm{p^n-p_{\text{ms}}^{n,k}}_{b} \right )^2.
\end{equation}

\noindent\textbf{Step 4.} 
We prove the desired convergence in this final step. 
As in the proof of a priori estimate of the Euler-Galerkin scheme in  \cite{ern2009posteriori}, we split the errors in the displacement and pressure into two parts each, namely
$$
\begin{aligned}
\rho_u^n&:= u^n-\mathfrak{R}_{a,{\text{ms}}}^{k+1}(u^n),\qquad&& \eta_u^n:=\mathfrak{R}_{a,{\text{ms}}}^{k+1}(u^n)-u_{\text{ms}}^{n,k+1},\\
\rho_p^n&:= p^n-\mathfrak{R}_{b,{\text{ms}}}^{k+1}(p^n),\qquad &&
\eta_p^n:=\mathfrak{R}_{b,{\text{ms}}}^{k+1}(p^n)-p_{\text{ms}}^{n,k+1}.\\
\end{aligned}
$$
Applying Lemma \ref{lm:infcontrl} to estimate the terms $\rho_u^n$ and $\rho_p^n$ yields
$$ \norm{u^n - u_{\ms}^{n,k+1}}_a^2 + \tau \norm{p^n - p_{\ms}^{n,k+1}}_b^2 \lesssim 
\mathcal{U}_{\inf} +\tau \mathcal{P}_{\inf} +\norm{\eta_u^n}_a^2+\tau\norm{\eta_p^n}_b^2,$$
where 
$$ \mathcal{U}_{\inf} := \inf_{v_H\in V_{\text{ms}}^{(k+1)}}\norm{u^n-v_H}_a^2 \tand \mathcal{P}_{\inf} := \inf_{p_H\in Q_{\text{ms}}^{(k+1)}}\norm{p^n-p_H}_b^2.$$
We go through the analysis in three stages. 
\begin{itemize}
\item[(i)] First, we bound the term $\mathcal{P}_{\inf}$.
Let $\tilde{p}^c_\ms\in Q_\ms$ and $\tilde{p}^d_\ms\in Q_\ms$ be the discrete elements  corresponding to $\tilde{p}^c$ and $\tilde{p}^d$ in Assumption \ref{conv_rate}, respectively.
Then, we take $$p_H=p_{\text{ms}}^{n,k}+\sum_{i\in\mathcal{I}_2}\rho_{\ms}^{(i)}+\sum_{i=1}^{\Ne} \sum_{j=1}^{J_i} c_j^{(i)}\phi_{j,\ms}^{(i)} +\tilde{p}^c_\ms+\tilde{p}^d_\ms \in Q_{\text{ms}}^{(k+1)}$$ and we have, using \eqref{eq:newrepres}, the following series of inequalities: 
$$
\begin{aligned}
\mathcal{P}_{\inf} &\leq 
\norm{p^n-p_{\text{ms}}^{n,k}-\sum_{i\in\mathcal{I}_2}\rho_{\ms}^{(i)}-\sum_{i=1}^{\Ne} \sum_{j=1}^{J_i^2} c_j^{(i)}\phi_{j,\ms}^{(i)}-\tilde{p}^c_\ms-\tilde{p}^d_\ms }_b^2\\
&= \norm{\sum_{i \in \mathcal{I}_2}(\rho_\glo^{(i)}-\rho_{\ms}^{(i)})+\sum_{i\not \in \mathcal{I}_2} \rho_\glo^{(i)}+\sum_{i=1}^{\Ne} \sum_{j=1}^{J_i^2} c_j^{(i)}(\phi_{j}^{(i)}-\phi_{j,\ms}^{(i)})+(\tilde{p}^c-\tilde{p}^c_\ms)+(\tilde{p}^d-\tilde{p}^d_\ms)}_b^2\\
& \lesssim \mathcal{P}_{\inf}^{(1)} + \mathcal{P}_{\inf}^{(2)}, 
\end{aligned}
$$
where 
\begin{eqnarray*}
\begin{split}
\mathcal{P}_{\inf}^{(1)}  & := \norm{\sum_{i \in \mathcal{I}_2}(\rho_\glo^{(i)}-\rho_{\ms}^{(i)})}_b^2
+\norm{\sum_{i=1}^{\Ne} \sum_{j=1}^{J_i^2} c_j^{(i)}(\phi_{j}^{(i)}-\phi_{j,\ms}^{(i)})}_b^2+
\norm{\tilde{p}^c-\tilde{p}^c_\ms\|_b^2+\|\tilde{p}^d-\tilde{p}^d_\ms}_b^2, \\
\mathcal{P}_{\inf}^{(2)}  & := \norm{\sum_{i \not \in \mathcal{I}_2}  \rho_{\glo}^{(i)}}_b^2 .
\end{split}
\end{eqnarray*}
Using Steps 2 and 3 as well as Lemma \ref{lm:onlinebasiserror}, one can see that 
$$
\begin{aligned}
\mathcal{P}_{\inf}^{(1)}  &\leq C(D_1, N_K) C_e C_r(\ell+1)^d (1+\Lambda^{-1})\left [ 
\beta^2 \norm{D_\tau u^n-D_\tau u^{n,k}_{\text{ms}}}_a^2 
+\norm{p^n-p_{\text{ms}}^{n,k}}^2_b
+\beta_2^2 \norm{D_\tau p^n-D_\tau p_{\text{ms}}^{n,k}}_c^2
\right ],
\end{aligned}
$$
where the constant $C(D_1, N_K)$ here depends on $D_1$ and $N_K$.
Next, we estimate the term $\mathcal{P}_{\inf}^{(2)}$. We write $\displaystyle \rho^*:=\sum_{i\not \in \mathcal{I}_2} \rho_\glo^{(i)}$. 
Using the definition \eqref{eqn:glo_online_basis} of global online basis functions, we have 
\begin{equation}
b(\rho^*, q) + s^2\left (\pi_2(\rho^*), \pi_2(q) \right )  = \sum_{i \in \mathcal{I}_2} r_{n,i}^2 (q)
\label{eqn:rho_star}
\end{equation}
for any $q \in Q$. Note that 
$$ \norm{\chi_{\omega_i}^2 \rho^*}_b^2 \leq (1+\Lambda^{-1}) \left ( \norm{\rho^*}_b^2 + \norm{\pi_2(\rho^*)}_{s^2}^2 \right ).$$
Taking $q = \rho^*$ in \eqref{eqn:rho_star}, we have 
$$
\begin{aligned}
\norm{\rho^*}_{b}^2+\norm{\pi_2(\rho^*)}_{s^2}^2 & = \sum_{i\not \in \mathcal{I}_2} r_{n}^{2,k}(\chi_{\omega_i}^2\rho^*)\\
&\leq \sum_{i\not \in \mathcal{I}_2} \left (\sup_{q\in Q(\omega_i) }\frac{ r_{n}^{2,k}(q)}{\|q\|_b} \right )
\norm{\chi_{\omega_i}^2 \rho^*}_b = \sum_{i \in \mathcal{I}_2} \eta_{n,i}^{2,k}
\norm{\chi_{\omega_i}^2 \rho^*}_b \\
&\leq \sqrt{2} (1+\Lambda^{-1})^{\frac{1}{2}} (\norm{\rho^*}_{b}^2+\norm{\pi_2(\rho^*)}_{s^2}^2)^{\frac{1}{2}} \sum_{i\not \in \mathcal{I}_2}\eta_{n,i}^{2,k} \\
&\leq \sqrt{2 N_K} (1+\Lambda^{-1})^{\frac{1}{2}} (\norm{\rho^*}_{b}^2+\norm{\pi_2(\rho^*)}_{s^2}^2)^{\frac{1}{2}} \left (  \sum_{i\not \in \mathcal{I}_2} \left (\eta_{n,i}^{2,k} \right )^2 \right )^{\frac{1}{2}}.
\end{aligned}$$
The last inequality follows from the fact that $\eta_{n,i}^{2,k}$ is defined over the coarse neighborhood $\omega_i$. As a result, 
we have 
$$
\mathcal{P}_{\inf}^{(2)} \leq 2N_K(1+\Lambda^{-1})  \sum_{i\not \in \mathcal{I}_2}\left ( \eta_{n,i}^{2,k} \right )^2 \leq 2N_K(1+\Lambda^{-1})\theta  \sum_{i=1}^{\Nv}\left ( \eta_{n,i}^{2,k} \right )^2, 
$$
where $\theta$ is the tolerance parameter in the online adaptive enrichment (see Section \ref{sec:algorithm}). 
It remains to analyze the term $\eta_{n,i}^{2,k}$. 
By definition, for any $q\in Q(\omega_i)$, we have 
$$\begin{aligned}
z_{n,i}^{2,k}(q)&=(f^n, q)-b(p_{\text{ms}}^{n,k}, q)-c(D_\tau p_{\text{ms}}^{n,k},q)-d(D_\tau u^{n,k}_{\text{ms}}, q)\\
&= d(D_\tau u^n-D_\tau u^{n}_{\text{ms}} , q)+ b(p^n-p_{\text{ms}}^{n,k}, \chi_{\omega_i}^2  q)+c(D_\tau p^n-D_\tau p_{\text{ms}}^{n,k}, q))\\
&\leq \left [ \beta \norm{D_\tau  u^n-D_\tau u^{n}_{\text{ms}}}_{a(\omega_i)}
+\norm{p^n-p_{\text{ms}}^{n,k}}_{b(\omega_i)} 
+\beta_2 \norm{D_\tau p^n-D_\tau p_{\text{ms}}^{n,k}}_{c(\omega_i)} \right ] \norm{q}_{b}.
\end{aligned}$$
Thus, we have 
$$\begin{aligned}
\sum_{i=1}^{\Nv} \left ( \eta_{n,i}^{2,k} \right )^2 &\lesssim \sum_{i=1}^{\Nv} \left [ 
\beta^2 \norm{D_\tau u^n-D_\tau u^{n,k}_{\text{ms}}}_{a(\omega_i)}^2
+\norm{p^n-p_{\text{ms}}^{n,k}}_{b(\omega_i)}^2
+\beta_2^2 \norm{D_\tau p^n-D_\tau p_{\text{ms}}^{n,k}}_{c(\omega_i)}^2 \right ] \\
&\leq N_K \left [ 
\beta^2\norm{D_\tau  u^n-D_\tau u^{n,k}_{\text{ms}}}_{a}^2
+\norm{p^n-p_{\text{ms}}^{n,k}}_{b}^2
+\beta_2^2 \norm{D_\tau p^n-D_\tau p_{\text{ms}}^{n,k}}_{c}^2 \right ]. 
\end{aligned}
$$
Overall, we have 
\begin{eqnarray}
\mathcal{P}_{\inf} \lesssim \left [ C(D_1, N_K) C_e C_r(\ell+1)^d + N_K^2 \theta \right ] (1+ \Lambda^{-1}) \mathcal{E}_p^k.
\label{eqn:conclu_1}
\end{eqnarray}

\item [(ii)] We estimate $\mathcal{U}_{\inf}$. It is similar to (i) and thus we only sketch the crucial steps. 
Let $\tilde{u}^d_\ms\in V_\ms$ be the discrete element corresponding to $\tilde{u}^d$ in Assumption \ref{conv_rate}.
Taking $$u_H=u_{\text{ms}}^{n,k}+\sum_{i\in\mathcal{I}_1}\delta_{\ms}^{(i)}+\sum_{i=1}^{\Ne} \sum_{j=1}^{J_i^1} d_j^{(i)}\psi_{j,\ms}^{(i)}+\tilde{u}^d_\ms \in V_{\text{ms}}^{(k+1)}$$ 
and using \eqref{eq:newrepres_u}, one can show that 
$$
\mathcal{U}_{\inf} \leq \norm{u^n-u_{\text{ms}}^{n,k}-\sum_{i\in\mathcal{I}_1}\delta_{\ms}^{(i)}-\sum_{i=1}^{\Ne} \sum_{j=1}^{J_i^1} d_j^{(i)}\psi_{j,\ms}^{(i)}}_a^2
\lesssim \mathcal{U}_{\inf}^{(1)} + \mathcal{U}_{\inf}^{(2)},
$$
where
\begin{eqnarray*}
\begin{split}
\mathcal{U}_{\inf}^{(1)} &:= \norm{\sum_{i \in \mathcal{I}_1}(\delta_\glo^{(i)}-\delta_{\ms}^{(i)})}_a^2
+\norm{\sum_{i=1}^{\Ne} \sum_{j=1}^{J_i^1} d_j^{(i)}(\psi_{j}^{(i)}-\psi_{j,\ms}^{(i)})}_a^2 +
\norm{\widetilde{u}^d-\tilde{u}^d_{\ms}}_a^2, \\
\mathcal{U}_{\inf}^{(2)} & := 
\norm{\sum_{i\not \in \mathcal{I}_1} \delta_\glo^{(i)}}_a^2. 
\end{split}
\end{eqnarray*}
Using  Steps 2 and 3, Lemma \ref{lm:onlinebasiserror}, and the definition of $\widetilde{u}^d$, one can show that 
$$\mathcal{U}_{\inf}^{(1)} \leq 
C(D_0, N_K)  C_e C_r(\ell+1)^d (1+\Lambda^{-1}) \left [ \norm{u^n-u^{n,k}_{\text{ms}}}_a^2+\norm{p^n-p^{n,k}_{\text{ms}}}_b^2 \right ],
$$
where $C(D_0, N_K)$ depends on the constants $D_0$ and $N_K$.
%
For the term $\mathcal{U}_{\inf}^{(2)}$, employing the same technique as in the procedure of analyzing $\mathcal{P}_{\inf}^{(2)}$, one can obtain the following estimate: 
$$ \mathcal{U}_{\inf}^{(2)} \lesssim N_K^2 (1+ \Lambda^{-1}) \theta \left [ \norm{u^n-u^{n,k}_{\text{ms}}}_a^2+\norm{p^n-p^{n,k}_{\text{ms}}}_b^2 \right ].$$
Consequently, we have 
\begin{eqnarray}
\mathcal{U}_{\inf} \lesssim \left [ C(D_0, N_K)  C_e C_r(\ell+1)^d + N_K^2 \theta \right ] ( 1+ \Lambda^{-1}) \left ( \norm{u^n-u^{n,k}_{\text{ms}}}_a^2+\norm{p^n-p^{n,k}_{\text{ms}}}_b^2 \right ).
\label{eqn:conclu_2}
\end{eqnarray}

\item [(iii)] Finally, we estimate $\eta_u^n$ and $\eta_p^n$.
The solution $(u_{\text{ms}}^{n,k+1},p_{\text{ms}}^{n,k+1})\in V_\ms^{(k+1)}\times Q_\ms^{(k+1)}$ satisfies
\begin{eqnarray*}
\begin{split}
	a(u_{\text{ms}}^{n,k+1},v) - d(v, p_{\text{ms}}^{n,k+1}) &= 0, \\
	d(u_{\text{ms}}^{n,k+1},q) + c(p_{\text{ms}}^{n,k+1},q) + \tau b(p_{\text{ms}}^{n,k+1},q) &= \tau(f^n,q)+d(u_{\text{ms}}^{n-1},q) + c(p_{\text{ms}}^{n-1},q), 
\end{split}
\end{eqnarray*}
\tforall $(v,q) \in V_\ms^{(k+1)}\times Q_\ms^{(k+1)}$.
Observe that
\begin{eqnarray*}
\begin{split}
a(\eta_u^n, v)-d(v,\eta_p^n)&=d(v,\rho_p^n), \\
c(\eta_p^n, q)+d(\eta_u^n, q)+\tau b(\eta_p^n,q)&=-\tau(f^n,q)-d(u_{\text{ms}}^{n-1},q)-c(p_{\text{ms}}^{n-1},q)+d(\mathfrak{R}_{a,\ms}^{k+1}(u^n),q)\\
& \quad +c(\mathfrak{R}_{b,\ms}^{k+1}(p^n),q)+\tau b(p^n,q)\\
&= -d(u^n-u^{n-1},q)-c(p^n-p^{n-1},q)-\tau b(p^n,q) -d(u_{\text{ms}}^{n-1},q)\\
& \quad c(p_{\text{ms}}^{n-1},q)+d(\mathfrak{R}_{a,\ms}^{k+1}(u^n),q)+c(\mathfrak{R}_{b,\ms}^{k+1}(p^n),q)+\tau b(p^n,q) \\
&=-d(\rho_u^n,q)-c(\rho_p^n,q)+d(u^{n-1}-u_\ms^{n-1},q)+c(p^{n-1}-p_\ms^{n-1},q), 
\end{split}
\end{eqnarray*}
for all $(v, q)\in V_{\text{ms}}^{n,k+1}\times Q_{\text{ms}}^{n,k+1}$. Testing with $v_H:=\eta_u^n$ and $q_H:=\eta_p^n$ and summing these two equations up yields 
\begin{eqnarray}
\begin{split}
\norm{\eta_u^n}_a^2+\norm{\eta_p^n}_c^2+\tau\norm{\eta_p^n}_b^2
&\lesssim \norm{\rho_p^n}_c^2+\beta_1^2\norm{\rho_u^b}^2_a+\beta_1^2\norm{u^{n-1}-u_\ms^{n-1}}_a^2+\norm{p^{n-1}-p_\ms^{n-1}}_c^2\\
&\leq\beta_1^2\norm{\rho_u^n}^2_a+ \beta_2^2\norm{\rho_p^n}_b^2 + 
\beta_1^2\norm{u^{n-1}-u_\ms^{n-1}}_a^2+\beta_2^2\norm{p^{n-1}-p_\ms^{n-1}}_b^2.
\end{split}
\label{eqn:conclu_3}
\end{eqnarray}
The terms $\rho_u^n$ and $\rho_p^n$ can be bounded by $\mathcal{U}_{\inf}$ and $\mathcal{P}_{\inf}$. 
\end{itemize}
Finally, combining \eqref{eqn:conclu_1}, \eqref{eqn:conclu_2}, and \eqref{eqn:conclu_3}, we obtain the desired convergence estimate. This completes the proof. 
\end{proof}

\section{Conclusion}\label{sec:con}

In this work, we proposed an online adaptive algorithm within the framework of the constraint energy minimizing generalized multiscale finite element method (CEM-GMsFEM) to solve the heterogeneous poroelasticity problem. The CEM-GMsFEM provides an offline multiscale method with a mesh-dependent convergence rate independent of the contrast and heterogeneities of the media. The proposed online adaptive algorithm can further reduce the error by adding online basis functions at the online stage based on the local residual information. The enrichment of these online basis functions results in fast error decay, especially when the online adaptive procedure is performed every several time steps. The convergence rate is determined by the online tolerance parameters, which are user-defined. Several numerical tests are shown to validate the theoretical estimates.

\section*{Acknowledgement}
SMP would like to thank the partial support from National Science Foundation (DMS-2208498) when he was affiliated with Department of Mathematics in Texas A\&M University. 

\bibliographystyle{abbrv}
\bibliography{ref_poro.bib}

\appendix 

\section{Implementation of Basis Function Construction} \label{appen:basis}

In this appendix, we discuss some implementation details of the construction of the multiscale basis functions. Recall that the multiscale basis functions are defined by \eqref{eq:mineq1_loc} and \eqref{eq:mineq2_loc}. 
In practice, one can construct the basis functions based on the underlying fine scale grid. 

Let $A_h$ and $B_h$ denote two symmetric and positive definite matrices, where $A_h$ is an $N_{V,\text{fine}}\times N_{V,\text{fine}}$ matrix with $(i,j)$-th entry $a(v_j, v_i)$ and $B_h$ is an $N_{Q, \text{fine}}\times N_{Q,\text{fine}}$ matrix with $(i,j)$-th entry $b(v_j, v_i)$. 
Let $M_{1,h}$ denote the  $N_{V,\text{fine}}\times N_{V,\text{fine}}$ matrix with the $(i,j)$-th entry $s^1(v_i,v_j),$ and $M_{2,h}$ denote the $N_{Q, \text{fine}}\times N_{Q,\text{fine}}$  matrix with the $(i,j)$-th entry $s^2(q_i,q_j)$.  $A_h^i$, $M^i_{1,h}$, and $M^i_{2,h}$ are the restriction of $A_h$, $M_{1,h}$, and $M_{2,h}$ on $K_{i,\ell}$, respectively.  
The number of fine scale basis in $V_h$ and $Q_h$ used on the coarse region $K_{i,\ell}$ is denoted by  $N^{i}_{V,\text{fine}}$ and $ N^{i}_{Q,\text{fine}}$, respectively. 
With these notations, we know that $A_h^i$, $M^i_{1,h}$, and $M^i_{2,h}$ have the sizes $N^{i}_{V,\text{fine}}\times N^{i}_{V,\text{fine}}$, $N^{i}_{V,\text{fine}}\times N^{i}_{V,\text{fine}}$, and $N^{i}_{Q,\text{fine}}\times N^{i}_{Q,\text{fine}}$.
$R^i_V$ and $R^i_Q$ are the matrices that include all the discrete  auxiliary basis in the space $V_{\text{aux}}(K_{i,\ell})$ and $Q_{\text{aux}}(K_{i,\ell})$, respectively. 
The number of auxiliary basis in $V_{\text{aux}}$ and $Q_{\text{aux}}$ defined on $K_{i,\ell}$ are denoted by $N^{i}_{V,\text{aux}}$ and $ N^{i}_{Q,\text{aux}}$,  then we know $R^i_V$ and $R^i_Q$  have the size $N_{V,\text{fine}}\times N_{V,\text{aux}}$ and $N_{Q,\text{fine}}\times N_{Q,\text{aux}}$.
$R^{i}_{V,j}$ are $R^{i}_{Q,j}$ are the j-th columns of $V_{\text{aux}}$ and $Q_{\text{aux}}$, respectively.
$\Psi_{j,h}^i$ is the discrete $ \psi_{j,\text{ms}}^{(i)}$ and $ \Phi_{j,h}^i$ is the discrete $ \phi_{j,\text{ms}}^i$.  Using the definitions above, the matrix formulation of \eqref{eq:mineq1_loc} is

\begin{equation}\label{eq:matrixmineq1_loc}
\left( A_h^i+ M_{1,h}^i (R^i_V R^{i,T}_V)M_{1,h}^{i,T}  \right) \Psi_{j,h}^i=M_{1,h}^{i}R^{i}_{V,j}  ,
\end{equation}
 while the matrix formulation of \eqref{eq:mineq2_loc} is

\begin{equation}\label{eq:matrixmineq2_loc}
\left( B_h^i+ M_{2,h}^i (R^i_QR^{i,T}_Q)M_{2,h}^{i,T}  \right) \Phi_{j,h}^i= M_{2,h}^{i}R^{i}_{Q,j}. 
\end{equation}


\section{Nomenclature}
\begin{table}[htbp!]
\begin{tabular}{c|l}
$\Omega, \partial \Omega$ & spatial domain and its boundary \\
$T$ & terminal time \\
$u$ & displacement variable \\
$p$ & pressure variable \\
$\sigma(u)$ & stress tensor \\
$\alpha$ & Biot-Willis fluid-solid coupling coefficient \\
$M$ & Biot modulus \\
$\kappa$ & permeability \\
$\nu$ & fluid viscosity \\
$f$ & source function \\
$p^0$ & initial condition (for pressure) \\
$\lambda, \mu$ & Lam\'e coefficients \\
$\varepsilon (u)$ & strain tensor \\
$E$ & Young's modulus \\
$\nu_p$ & Poisson ratio \\
$\mathcal{T}^h$ & spatial fine mesh \\
$h$ &fine mesh size \\
$\mathcal{T}^H$ & coarse mesh \\
$H$ & spatial coarse mesh size \\
$\tau$ & temporal step size \\
$J_i^1$ & number of local offline basis functions (displacement) \\
$J_i^2$ & number of local offline basis functions (pressure) \\
$\ell$ & oversampling parameter \\
$\gamma, \theta$ & online enrichment parameters \\

\end{tabular}
\end{table}

\end{document}